\documentclass[11pt,reqno]{amsart}

\usepackage[a4paper,margin=1.1in]{geometry}
\usepackage[T1]{fontenc}
\usepackage[utf8]{inputenc}
\usepackage{lmodern}
\usepackage{microtype}

\usepackage{fancyhdr}
\makeatletter
\let\ps@plain=\ps@fancy
\makeatother

\usepackage{amsmath,amssymb,amsthm,mathtools,mathrsfs,bm}
\numberwithin{equation}{section}
\allowdisplaybreaks

\usepackage{graphicx}
\usepackage{caption}
\usepackage{placeins}
\usepackage{tikz}

\usepackage{enumitem}
\setlist{itemsep=2pt,topsep=4pt}

\usepackage[dvipsnames]{xcolor}

\usepackage[hidelinks]{hyperref}
\usepackage[nameinlink,noabbrev]{cleveref}

\theoremstyle{plain}
\newtheorem{theorem}{Theorem}[section]
\newtheorem{proposition}[theorem]{Proposition}
\newtheorem{lemma}[theorem]{Lemma}

\theoremstyle{definition}
\newtheorem{definition}[theorem]{Definition}

\theoremstyle{remark}
\newtheorem{remark}[theorem]{Remark}

\newcommand{\R}{\mathbb{R}}
\newcommand{\C}{\mathbb{C}}
\newcommand{\Q}{\mathbb{Q}}

\newcommand{\N}{\mathbb{N}}

\newcommand{\Lc}{\mathcal{L}}
\newcommand{\Rc}{\mathcal{R}}
\newcommand{\Nc}{\mathcal{N}}
\newcommand{\Kc}{\mathcal{K}}

\newcommand{\Oh}{\mathcal{O}}

\newcommand{\ftil}{\widetilde{f}}
\newcommand{\wtil}{\widetilde{w}}
\newcommand{\vtil}{\widetilde{v}}

\newcommand{\Lctil}{\widetilde{\mathcal{L}}}
\newcommand{\Kctil}{\widetilde{\mathcal{K}}}
\newcommand{\Kchat}{\widehat{\mathcal{K}}}
\newcommand{\Gtil}{\widetilde{G}}
\newcommand{\Acn}{\mathcal{A}_0}
\newcommand{\Vtil}{\widetilde{V}}

\newcommand{\cnhat}{\hat{c}_n}
\newcommand{\sumast}{\sideset{}{^*}\sum}
\renewcommand{\Re}{\operatorname{Re}}

\newcommand{\rg}{\operatorname{rg}}
\newcommand{\B}{\mathbb{B}}
\newcommand{\la}{\lambda}

\newcommand{\mc}[1]{\mathcal{#1}}

\newcommand{\abs}[1]{\left\lvert #1\right\rvert}



\title{Existence of a stable shrinker for the corotational harmonic map heat flow in higher space dimensions}

\author{Johannes Angerer}
\address{Universit\"at Innsbruck, Institut f\"ur Mathematik,\\ Technikerstraße 13, 6020 Innsbruck, Austria}
\email{Johannes.Angerer@student.uibk.ac.at}

\author{Sarah Kistner}
\address{Universit\"at Innsbruck, Institut f\"ur Mathematik,\\ Technikerstraße 13, 6020 Innsbruck, Austria}
\email{Sarah.kistner@uibk.ac.at}
\author{Birgit Sch\"orkhuber}
\address{Universit\"at Innsbruck, Institut f\"ur Mathematik,\\ Technikerstraße 13, 6020 Innsbruck, Austria}
\email{Birgit.Schoerkhuber@uibk.ac.at}


\date{}

\begin{document}

\begin{abstract}
We study singularity formation for the heat flow of harmonic maps from $\R^d$ into $\mathbb{S}^d$ in supercritical dimensions $d \in \{3,4,5,6\}$. It is well known that in each of these dimensions there exist infinitely many self-similar solutions that provide examples of loss of regularity in finite time. In this paper, we extend the results of \cite{BieDon18}, \cite{BieDonSch17} for $d=3$ to higher space dimensions $d \in \{4,5,6\}$ and prove the existence of a monotonically increasing self-similar profile $f_0$, which is asymptotically stable under small corotational perturbations. To construct the solution and resolve the spectral problem, we use rigorous computer assistance. As a byproduct of our stability analysis, we also obtain finite-codimension stability of arbitrary self-similar profiles within the corotational class.
\end{abstract}

\maketitle

\section{Introduction}
We consider the heat flow of harmonic maps from $\R^d$ into $\mathbb{S}^d$, which corresponds to the negative $L^2$-gradient flow of the energy functional
\begin{align} 
S(U) = \frac{1}{2} \int_{\R^d} \tilde{h}^{ij}(x) h_{ab}(U(x)) \partial_i U^{a}(x) \partial_j U^{b}(x) \sqrt{|\tilde{h}(x)|} dx,
\end{align}
where $\tilde{h}$ and $h$ are the corresponding metrics on $\R^d$ and $\mathbb{S}^d$, respectively. The flow reads\footnote{Using Einstein's summation convention.}
\begin{align} \label{L2Gradient}
\partial_t U^{a} - \Delta_{\tilde{h}} U^{a} - \tilde{h}^{ij} \Gamma^{a}_{bc}(U) \partial_i U^{b} \partial_j U^{c} = 0, \quad t>0.
\end{align}

It is well known that singularity formation in finite time occurs in $d \geq 2$; see \cite{CorGhi89}, \cite{ChaDinYe92}. Classical results by Struwe \cite{Str1988} imply that scale-invariant, i.e., self-similar, solutions may play an important role in the characterization of blowup in $d \geq 3$. We restrict ourselves to this case for the rest of the paper and refer to \cite{LinWan08} for a more general discussion and further references. \\

To reduce the system of equations \eqref{L2Gradient} by means of corotational symmetry, we choose polar coordinates $(r, \omega) \in [0,\infty) \times \mathbb{S}^{d-1}$ on $\R^d$ and hyperspherical coordinates $(u,\Omega) \in (0, \pi) \times \mathbb{S}^{d-1}$ on $\mathbb{S}^d$, and make the ansatz $U(t, r, \omega) = (u(t,r), \omega)$. Under this symmetry, Eq.~\eqref{L2Gradient} reduces to a single semilinear radial heat equation for $u$, given by
\begin{align} \label{EQu}
\partial_t u(t,r) - \partial^2_r u(t,r) - \frac{d-1}{r} \partial_r u(t,r) + \frac{(d-1)}{2r^2} \sin(2 u(t,r)) = 0.
\end{align}

Within the corotational setting, the scaling law is $u \mapsto u_{\lambda}$, $\lambda >0$, where 
\begin{align*}
u_{\lambda}(t,r) = u(\lambda^2 t,\lambda r),
\end{align*} 
and self-similar solutions are either of the form $f \left( \tfrac{r}{\sqrt{-t}} \right)$ or $f \left( \tfrac{r}{\sqrt{t}} \right)$. The latter are usually referred to as \textit{expanders} and have been studied intensively by Germain and Rupflin \cite{GerRup2011}, as well as by Germain, Ghoul, and Miura \cite{GerGhoMiu2017}.
In contrast, \textit{shrinkers} provide explicit examples of singularity formation in finite time. More precisely, any solution of the profile equation 
\begin{align} \label{Eqf}
f^{\prime \prime}(\rho) + \left( \frac{d-1}{\rho} - \frac{\rho}{2} \right) f^{\prime}(\rho) - \frac{(d-1)}{2 \rho^2} \sin(2 f(\rho))=0, \quad \rho \geq 0, \qquad f(0)=0, \quad f(\infty) = \mathrm{const.}
\end{align}
with $f'(0)  > 0$ gives rise to a solution of \eqref{EQu} of the form 
\begin{align} \label{Shrinker}
u_T(t,r) = f \left( \frac{r}{\sqrt{T-t}} \right), \quad (t,r) \in [0,T) \times [0, \infty),
\end{align}
satisfying 
\[
|\partial_r u(t,0) | \sim (T-t)^{-\frac{1}{2}}, \quad t \to T^{-}.
\]

Interestingly, Bizo\'n and Wasserman \cite{BizonWasserman2015} proved that no such solutions exist in high space dimensions $d \geq 7$. Instead, blowup is of type II, with explicit constructions by Biernat and Seki \cite{BieSek19}, \cite{BieSek20}, as well as by Ghoul, Ibrahim, and Nguyen \cite{GhoIbrNgu19}. We also refer to \cite{Gas02} for results in the more general $k$-equivariant case. \\

The picture changes completely in dimensions $d \in \{3, 4, 5, 6\}$. In this case, Fan proved the existence of an infinite countable family of profiles solving Eq.~\eqref{Eqf}. Moreover, numerical experiments by Biernat and Bizo\'n \cite{BieBiz2011} within the corotational setting suggest that the \textit{ground state} $f_0$ of this family---which can be ordered according to the number of intersections with $\frac{\pi}{2}$---acts as an attractor for generic blowup solutions. In particular, it is expected to be asymptotically stable. A major obstacle to a rigorous stability analysis is the fact that neither $f_0$ nor any other shrinker profile is known in closed form.

In the lowest supercritical dimension $d=3$, this problem has been addressed using rigorous computer-assisted methods. More precisely, Biernat and Donninger \cite{BieDon18} proved, using Chebyshev spectral methods combined with interval arithmetic, that there is a monotonically increasing shrinker profile, called $f_0$, which is spectrally stable with respect to an underlying self-adjoint problem.
Based on this result, Biernat, Donninger, and the third author proved the nonlinear asymptotic stability of $f_0$ under small corotational perturbations \cite{BieDonSch17}. The main goal of this paper is to generalize this result to higher space dimensions $d \in \{4,5,6\}$.

\subsection{The main results}\label{Sec:Main_result}

We study the harmonic map heat flow \eqref{L2Gradient} in \textit{normal coordinates} $(u, \tilde{\Omega})$ on $\mathbb{S}^d$, where $\tilde{\Omega} \in \mathbb{S}^{d-1} \hookrightarrow \R^d$. In this way, we can identify any map into $\mathbb{S}^d$ with a map $U: \R^d \rightarrow \R^d$ of the form $U(x) = (u(x), \tilde{\Omega}(x))$ for $x \in \R^d$; see, e.g., \cite{ShaTah94}. Corotational solutions $U: [0, T) \times \R^d \rightarrow \R^d$ transform to
\begin{align*}
U(t,x) = \frac{x}{|x|} u(t,|x|),
\end{align*}
where $u$ solves equation \eqref{EQu}. A self-similar shrinking solution $u_T$ given by \eqref{Shrinker} leads to a corresponding shrinker $U_T: [0,T) \times \R^d \rightarrow \R^d$ in normal coordinates,
\begin{align*}
U_T(t,x) = \frac{x}{|x|} u_T(t,|x|) = \frac{y}{|y|} f \left( |y| \right),
\end{align*}
where $y := \frac{x}{\sqrt{T-t}}$ is the self-similar spatial variable. Since in every dimension $d \in \{3,4,5,6 \}$ there exists an infinite countable family of shrinking solutions, we divide the stability analysis into two parts. First, we prove a codimension-$N$ stability result for any shrinker $U_T$ with corresponding profile $f$. Then, we show that there exists a profile $f_0$ that solves equation \eqref{Eqf} and is stable in the sense that the only non-negative eigenvalue of the underlying linear operator is $1$, which is due to the freedom of choosing the blowup time $T$.\\

To formulate the first result, let $f \in C^{\infty}([0,\infty))$ be a solution of \eqref{Eqf} and define $F(y) := \frac{y}{|y|} f(|y|)$, as the corresponding self-similar solution to equation \eqref{L2Gradient} in normal coordinates. To study the stability of $F$, we set $T=1$ and consider corotational initial data of the form
\begin{align*}
U(0,x) = F(x) + H_0(x),
\end{align*}
where $H_0(x):= \frac{x}{|x|} h_0(|x|)$ is a perturbation in normal coordinates with radial profile $h_0 : [0, \infty) \rightarrow \R$.

\begin{theorem}[Finite-codimension stability] \label{Th:CodimensionStability}
Let $d \in \{3,4,5,6 \}$, and let $s,k >0$ satisfy
\begin{align*}
\frac{d}{2} < s \leq \frac{d}{2} + \frac{1}{2d}, \quad k =d+3.
\end{align*}
Then there exist $\varepsilon > 0$, $N_f \in \N$, and real valued functions $F^{i}_j \in \dot{H}^s \cap \dot{H}^k(\R^d)$ with $i \in \{1,\dots,N_j \}$ and $j \in \{1,\dots,N_f \}$ such that the following holds. Let $H_0: \R^d \rightarrow \R^d$ be a Schwartz function satisfying
\begin{align*}
\Vert H_0 \Vert_{\dot{H}^s \cap \dot{H}^k(\R^d)} \leq \varepsilon.
\end{align*}
Then there exist constants $a_{i,j}(H_0) \in \R$, $i \in \{1,\dots,N_j \}$ and $j \in \{1,\dots,N_f\}$, depending Lipschitz continuously on $H_0$, such that the initial value problem \eqref{L2Gradient} with initial datum
\begin{align}
U(0, \cdot)&  = F + H_0 + \sum_{j=1}^{N_f} \sum_{i=1}^{N_j} a_{i,j}(H_0) F^{i}_j
\end{align}
has a classical solution $U \in C^{\infty}([0, 1) \times \R^d)$ that blows up at the origin as $t \rightarrow 1^{-}$. Furthermore, $U$ can be decomposed as
\begin{align*}
U(t,x) =F \left( \frac{x}{\sqrt{1-t}} \right) + H \left( t, \frac{x}{\sqrt{1-t}} \right),
\end{align*}
where, for any $p \in [s,k]$, we have
\begin{align} \label{DecayH}
\Vert H(t, \cdot) \Vert_{\dot{H}^p(\R^d)} \rightarrow 0
\end{align}
as $t \rightarrow 1^-$.
\end{theorem}

Our second result concerns the existence and stability of a shrinker $U_T$ in any space dimension $d \in \{3,4,5,6\}$ without the necessity of correcting the initial condition along a codimension-$N$ Lipschitz manifold. For this, we exploit a rigorous computer-assisted method that allows us to solve the spectral problem.

\begin{theorem} \label{TheoremStability}
Let $d \in \{3,4,5,6 \}$, and let $s,k >0$ satisfy
\begin{align*}
\frac{d}{2} < s \leq \frac{d}{2} + \frac{1}{2d}, \quad k =d+3.
\end{align*}
There exists a self-similar profile $f_0 \in C^{\infty}([0,\infty))$ that solves Eq.~\eqref{Eqf} and defines, for any $T >0$, a self-similar blowup solution
\[
U_T(t,x) :=F_0 \left( \frac{x}{\sqrt{T-t}} \right), \quad F_0(y):=\frac{y}{|y|} f_0 \left( |y| \right),
\]
to Eq.~\eqref{L2Gradient}, such that the following holds. There exists an $\varepsilon >0$ such that for any corotational initial datum of the form
\begin{align*}
U_0 = F_0 + H_0,
\end{align*}
where $H_0: \R^d \rightarrow \R^d$ is a Schwartz function satisfying
\begin{align*}
\Vert H_0 \Vert_{\dot{H}^s \cap \dot{H}^k(\R^d)} \leq \varepsilon,
\end{align*}
there exist $T >0$ and a classical solution $U \in C^{\infty}([0,T) \times \R^d)$ to \eqref{L2Gradient} whose gradient blows up at the origin as $t \rightarrow T^{-}$. Furthermore, $U$ can be decomposed as
\begin{align} \label{Decompo}
U(t,x) = F_0 \left( \frac{x}{\sqrt{T-t}} \right) + H \left( t, \frac{x}{\sqrt{T-t}} \right),
\end{align}
where, for any $p \in [s,k]$, we have
\begin{align*}
\Vert H(t, \cdot) \Vert_{\dot{H}^{p}(\R^d)} \rightarrow 0
\end{align*}
as $t \rightarrow T^{-}$.
\end{theorem}

\begin{remark}
We note that the profile satisfies $0 \leq f_0 < \pi$. However, we do not have pointwise control over the perturbation $H$ and thus cannot guarantee that $U$ takes values only in $B^{d}_{\pi}$. In particular, it could wrap around the sphere multiple times. Nevertheless, we do control the gradient of $U$, since Sobolev embedding implies that $\| \nabla H(t,\cdot) \|_{L^{\infty}(\R^d)} \to 0$ as $t \to T^{-}$, which yields
\[
\sqrt{T-t} \| \nabla U(t,\sqrt{T-t} \cdot) - \nabla F_0 \|_{L^{\infty}(\R^d)}  \to 0
\]
as $t \to T^{-}$.
\end{remark}

\subsection{Outline of the proof and novelties}

We briefly outline the main steps of the proof. First, it is convenient to set $rv(t,r) := u(t,r)$ for $r > 0$ and observe that $v$ solves
\begin{align} \label{Eqv}
\partial_t v(t,r) - \partial^2_r v(t,r) - \frac{d+1}{r} \partial_r v(t,r) + \frac{(d-1)}{2r^3} \left( \sin(2 rv(t,r)) -2r v(t,r) \right) = 0.
\end{align}

Given a shrinker profile $f$, we insert a perturbative ansatz into Eq.~\eqref{Eqv}, which we regard as a radial heat equation on $\R^{n}$ with $n := d+2$. Transforming the resulting equation to similarity coordinates $(\tau, \xi) \in [0,\infty) \times \R^n$ yields
\begin{equation*}
  \partial_{\tau}\varphi(\tau,\cdot)
  = L_{f}\varphi(\tau,\cdot) + N(\varphi(\tau,\cdot)), \quad \tau > 0,
\end{equation*}
for the perturbation $\varphi$, where $L_{f} = L_0 + L'_{f}$ with
\begin{align*}
  L_0 u := \Delta u - \tfrac{1}{2}\xi\cdot\nabla u - \tfrac{1}{2}u,
  \qquad
  L'_f u := V_f(|\cdot|)\,u,
\end{align*}
and $V_f$ is the radial potential determined by the profile $f$. The natural function space for the linearized problem is the weighted space $L^{2}_{\sigma}(\R^n)$, $\sigma(x) = e^{-|x|^2/4}$, restricted to radial functions. In this setting, $L_{f}$ is self-adjoint and generates a strongly continuous semigroup $(S(\tau))_{\tau \geq 0}$. Naturally, the spectrum of $L_{f}$ depends on the potential and hence on the profile $f$.
Our key result proves the existence of a non-negative, monotonically increasing profile $f_0$ that is \textit{spectrally stable} modulo symmetry eigenvalues. More precisely, the self-adjoint operator $L_{f_0}$ satisfies
\begin{align*}
\sigma(L_{f_0}) \subset (-\infty,0) \cup \{ 1 \}.
\end{align*} 

The full result is stated in Theorem \ref{Th:Existence_f0}. The proof of Theorem \ref{Th:Existence_f0} follows the general approach of \cite{BieDon18}, which treats the case $d = 3$. We first construct a closed-form approximation $\ftil_0$ with rational coefficients determined using Chebyshev spectral methods. The approximation is then extended to an exact solution $f_0$ by a fixed-point argument. The main difficulty is the construction of the inverse of an approximate linear operator arising in the fixed-point argument and a rigorous proof of explicit bounds. This requires suitable approximations of the associated fundamental system, which is delicate, especially in even space dimensions.

The result on the spectrum of $L_{f_0}$ relies on analyzing the underlying Sturm--Liouville operator and rigorously counting the zeros of solutions to the associated spectral problem. Both the fixed-point construction of the profile $f_0$ and the subsequent spectral analysis require rigorous bounds for rational functions. In contrast to \cite{BieDon18}, where interval arithmetic is used, we implement the \textit{evaluation method} of \cite{DonSch26}. As the spatial dimension increases, the required estimates become more demanding, in particular because the spectral argument requires sharper bounds on the exact solution. For this reason, we optimize the evaluation method for our setting by developing an exact truncated Chebyshev expansion with rigorous tail bounds; see the discussion in Section \ref{Sec:ComAss}. For the sake of completeness, we include $d=3$ and thus reprove the statement in this case. \\

To handle the nonlinear stability problem, we work in the radial intersection Sobolev space
\[
  X := \dot{H}^s_{\mathrm{rad}}(\R^n) \cap \dot{H}^{k}_{\mathrm{rad}}(\R^n)
\]
for $s_c < s < s_c + \frac{1}{2(n-2)} < \frac{n}{2}$, where $s_c$ is the scaling-critical exponent and $k = n+1$. With this choice, one has
\[
X \hookrightarrow C_{\mathrm{rad}}(\R^n) \cap L^{\infty}(\R^n)
\hookrightarrow L^2_{\sigma}(\R^n),
\]
and local Lipschitz bounds for the nonlinearity follow from Schauder-type estimates established in \cite{Glo2025}. The linearized evolution on $X$ is obtained from that on $L^{2}_{\sigma}(\R^n)$ by restriction.
The most crucial step is to transfer the spectral information from Theorem \ref{Th:Existence_f0} to the non-self-adjoint operator $L_{f_0}^{X} := L_{f_0}|_{X}$ and to use it efficiently in order to derive suitable growth bounds for the linearized evolution.
In contrast to \cite{BieDonSch17}, which is based on energy estimates using graph norms, we employ a considerably shorter argument; see \cite{GloKisSch26}. This argument uses the representation
\[
  S_{f_0}^{X}(\tau) - S^{X}_0(\tau)
  = \int_0^{\tau} S_{f_0}^{X}(\tau-\tau') L'_{f_0} S_0^{X}(\tau')\,d\tau',
\]
where $(S^X_0(\tau))_{\tau \geq 0}$ is the semigroup generated by $L^X_0$. Exploiting the decay of $f_0$ at infinity and the explicit form of $S^X_0(\tau)$, we show that the integral defines a compact operator on $X$ for every $\tau \geq 0$. This allows us to apply standard results relating the spectrum of a semigroup to its growth bound and to the spectrum of its generator, thereby yielding the required bounds for the evolution on $X$. The nonlinear problem is then treated by a standard fixed-point argument, including a suitable adjustment of the blowup time. In fact, the stability analysis carries over to arbitrary self-similar profiles, except that only finite-codimension stability can be obtained without a precise spectral analysis.\subsection{Notation and conventions}
 We write $a \lesssim b$ if there exists a constant $C >0$, such that $a \leq C b$ and we write $a \simeq b$ if $a \lesssim b$ and $b \lesssim a$. If the constant $C$ depends on some parameter $\varepsilon$, then we write $a \lesssim_{\varepsilon} b$.
We denote the open ball with radius $R>0$ in $\R^d$ by $B_R(\R^d)$. We also use the common Japanese bracket notation $\langle x \rangle := \sqrt{1+ |x|^2}$.
By $C^{\infty}(\R^d)$ and $\mc S(\R^d)$ we denote the space of smooth functions and the space of Schwartz functions respectively. By $C^{\infty}_{c}(\R^d)$ we denote the standard test space consisting of smooth and compactly supported functions. In case of radial functions we use a lower index $r$ as in $C^{\infty}_{r}(\R^d)$, $\mc S_r(\R^d)$, $C^{\infty}_{c,r}(\R^d)$. For smooth functions being bounded and having all (partial) derivatives bounded we use the lower index $b$ as in $C^{\infty}_b([0,\infty))$. For convenience, we also write $C^{\infty}(\R^d)$, $C^{\infty}_{c}(\R^d)$ and $\mc S(\R^d)$ for sets of vector-valued functions whose every component belongs to that space. For a closed linear operator $(\mathcal{L}, \mathcal{D}(\mathcal{L}))$, we write $\rho(\mathcal{L})$ for the resolvent set, and $\sigma(\mathcal{L}) := \C \setminus \rho(\mathcal{L})$ for the spectrum. Given $\la \in \rho(\mc L)$, we use the following convention for the resolvent $R_{\mathcal{L}}(\lambda) := (\lambda - \mathcal{L})^{-1}$. For $f \in C^{\infty}_{c}(\R^d)$, we use the following definition of the Fourier transform
\begin{align*}
	\hat{f}(\xi) = \mathcal{F}f(\xi) := (2 \pi)^{-\frac{d}{2}} \int_{\R^d} e^{-i \xi \cdot x} f(x) dx, \quad \xi \in \R^d.
\end{align*}	
	
For $f, g \in C^{\infty}_c(\R^d)$ and $s \geq 0$, we define the homogeneous inner product
\begin{align*}
	\langle f,g \rangle_{\dot{H}^s(\R^d)} := \langle | \cdot |^s \mathcal{F} f, | \cdot |^s \mathcal{F} g \rangle_{L^2(\R^d)},
\end{align*}
which induces a homogeneous Sobolev norm on $C^{\infty}_c(\R^d)$
\begin{align}\label{Def:Sob_norm}
	\Vert f \Vert^2_{\dot{H}^s(\R^d)} := \langle f,f \rangle_{\dot{H}^s(\R^d)}.
\end{align}
As usual, the homogeneous Sobolev space $\dot{H}^s(\R^d)$ is defined as the completion of $C^{\infty}_c(\R^d)$ under the norm \eqref{Def:Sob_norm}. Analogously, we define intersection Sobolev spaces $\dot{H}^{s_1}(\R^d) \cap \dot{H}^{s_2}(\R^d)$ and radial intersection Sobolev spaces $\dot{H}_r^{s_1}(\R^d) \cap \dot{H}_r^{s_2}(\R^d)$   as the completion of $C^{\infty}_c(\R^d)$ and $C^{\infty}_{c,r}(\R^d)$, respectively, with respect to the inner product
\begin{align*}
	\langle u,v \rangle_{\dot{H}^{s_1}(\R^d) \cap \dot{H}^{s_2}(\R^d)} := \langle u,v \rangle_{\dot{H}^{s_1}(\R^d)}  + \langle u,v \rangle_{\dot{H}^{s_2}(\R^d)}.
\end{align*}

\section{Formulation of the stability problem - Spectral stability}

In the following, we assume that $f \in C^{\infty}([0,\infty))$ is a solution to Eq.~\eqref{Eqf}. Then we have the following characterization.

\begin{lemma} \label{Lemmaf}
Any solution $f \in C^{\infty}([0,\infty))$ to \eqref{Eqf} satisfies
\begin{align} \label{Lemma1}
f^{(2 \ell)}(0) =0, \quad \ell \in \N_0,
\end{align}
and for each $\ell \in \N$ there exists a constant $C_{\ell} >0$ such that 
\begin{align} \label{Lemma2}
|f^{(\ell)}(\rho)| \leq C_{\ell} \rho^{-2-\ell}, \quad \rho \geq 1.
\end{align}
\end{lemma}

\begin{proof}
For $d=3$ the proof is given in full detail in \cite{BieDon18}, Proposition 1.3. By inspection, one sees that the argument does not depend on $d$ and is true in all higher space dimensions $d \in \{4,5,6 \}$.
\end{proof}

We consider Eq.~\eqref{Eqv} as a starting point for the investigation of the stability of a shrinker profile $f$ under corotational perturbations.
More precisely, we consider the following Cauchy problem for $w(t,x) := v(t,|x|)$, $x \in \R^n$, $n = d+2$, 
\begin{equation} \label{Problemw}
\left\{
\begin{aligned} \partial_t w(t,x) &= \Delta w(t,x) + \frac{n-3}{2|x|^3} \left( 2 |x| w(t,x) - \sin(2 |x| w(t,x)) \right), \quad t>0,\\
w(0,\cdot) &= |\cdot|^{-1} f(|\cdot|) + |\cdot|^{-1} h_0(\cdot)
\end{aligned}
\right.
\end{equation}
where $f$ is a smooth shrinker profile and $h_0: \R^n \rightarrow \R$ a radial perturbation. We note that for $n \in \{5,6,7,8\}$ the existence of shrinkers is guaranteed by \cite{Fan1999}.
We study the equation in similarity coordinates $(\tau, \xi) \in [0,\infty) \times \R^n$,
\begin{align*}
\tau = \tau(t) := \log \left( \frac{T}{T-t} \right) \quad \text{and} \quad \xi := \frac{x}{\sqrt{T-t}},  \quad T >0,
\end{align*}
and set
\begin{align*}
\psi(\tau, \xi) = \sqrt{T} e^{- \frac{\tau}{2} } w(T(1-e^{-\tau}), \sqrt{T} e^{- \frac{\tau}{2} }\xi),
\end{align*}
which yields
\begin{equation} \label{Psi}
\left\{
\begin{aligned} \partial_{\tau} \psi(\tau, \xi) &= \Delta_\xi \psi(\tau, y) -\Lambda \psi(\tau, \xi) + \frac{n-3}{2|\xi|^3} \left( 2 |\xi| \psi(\tau, \xi) - \sin(2 |\xi| \psi(\tau, \xi)) \right), \quad \tau >0,\\
\psi (0, \cdot) &= \sqrt{T} ( \mathbf{f}(\sqrt{T}|\cdot|) +  \mathbf{h}_0(\sqrt{T} |\cdot| )),
\end{aligned}
\right.
\end{equation}
where $\mathbf{f}:=  |\cdot|^{-1} f(|\cdot|)$, $\mathbf{h}_0:=  |\cdot|^{-1} h_0(|\cdot|)$ , and 
\begin{align*}
[\Lambda f](\xi) := \frac{1}{2} \xi \cdot \nabla f(\xi) + \frac{1}{2} f(\xi).
\end{align*}
The ansatz 
\begin{align*}
\psi(\tau, \cdot) = \mathbf{f}(|\cdot|) + \varphi (\tau, \cdot),
\end{align*}
for a radial perturbation $\varphi: \R^n \rightarrow \R$ leads to the evolution equation
\begin{equation} \label{mainproblem}
	\begin{cases}
		\partial_{\tau} \varphi (\tau, \cdot) &= L_f \varphi(\tau, \cdot) +N(\varphi(\tau, \cdot)), \quad \tau >0,\\
		\varphi(0,\cdot) &= U(\mathbf{h}_0,T),
	\end{cases}
\end{equation}
for data
\begin{align} \label{Initialdataoperator}
U(\mathbf{h}_0,T) := \sqrt{T}  \mathbf{h}_0(\sqrt{T}|\cdot|)  +  \sqrt{T} \mathbf{f}(\sqrt{T}|\cdot|) -\mathbf{f}(|\cdot|)
\end{align}
and formal differential operators
\begin{align}\label{Def:LinOp_f}
L_{f}= L_0 + L'_f, \quad L_0 u:= \Delta u- \Lambda u, \quad L'_{f} u = V_{f}(|\cdot|) u,
\end{align}
with the potential
\begin{align} \label{PotentialV}
V_f(y) = \frac{n-3}{y^2}\big (1-\cos(2 f(y))\big ), \quad y \in [0,\infty).
\end{align} 

The nonlinear remainder is given by
\begin{align} \label{Nonlinearity}
[N(\varphi (\tau,\cdot))](\xi) = \frac{n-3}{|\xi|^3} ( \eta^{\prime}(f(|\xi|)) |\xi| \varphi (\tau, \xi) + \eta(f(|\xi|))- \eta(f(|\xi|)+|\xi| \varphi(\tau,\xi))),
\end{align}
where $\eta(z) = \frac{1}{2} \sin(2 z)$ for $z \in \R$. Obviously, the nonlinearity depends on the profile $f$ as well. However, we drop this dependence in the notation as the 
nonlinear part of the argument does not rely on the precise form of the profile.

For later use, we state the following properties of the potential. 

\begin{lemma}\label{Le:Potential}
Let $f \in C^{\infty}([0,\infty))$ be a solution to Eq.~\eqref{Eqf}. Then the corresponding potential satisfies $V_f \in C_b^{\infty}([0,\infty))$ and for every $k \in \N_0$ there is a constant 
$C_k > 0$ such that 
\[ |V_f^{(k)}(y)| \leq C_k \langle y \rangle^{-2-k} \]
for all $y \in [0,\infty)$. 
\end{lemma}

\begin{proof}
The decay at infinity is obvious and regularity at zero is obvious from Lemma \ref{Lemmaf} and Taylor expansion.
\end{proof}

\subsection{Self-adjoint linear theory - Spectral stability}\label{Sec:SpectralStab_1}

We set $\sigma(x) := e^{-|x|^2/4}$ for $x \in \R^n$ and define a weighted $L^2$-space of radial functions
\begin{align}\label{L2sigma}
H := \{ u \in L^2_{\sigma}(\R^n) : \text{$u$ is radial} \},
\end{align}
with induced norm $\Vert \cdot \Vert_{H}$ coming from the inner product
\begin{align*}
\langle u, v \rangle_{H} := \int_{\R^n} u(x) \overline{v(x)} \sigma(x) dx, \quad \text{for $u,v \in H$}.
\end{align*}

We note that the choice of a complex function space accommodates spectral theory.  We consider $L_f$ and $L_0$ as defined in Eq.~\eqref{Def:LinOp_f}  on $C^{\infty}_{c,r}(\R^n) \subset H$. Thus, they are densely defined and symmetric operators on $H$. 

It is well-known that 
$(L_0,C^{\infty}_{c,r}(\R^n))$ is closable in $H$ and the closure $L_0 : \mathcal{D}(L_0) \subseteq H\rightarrow H$ is self-adjoint, has compact resolvent and generates a strongly continuous semigroup $(S_0(\tau))_{\tau \geq 0}$ of bounded operators on $H$ given by
\begin{align} \label{Freesemigroup}
[S_0(\tau)u](x) = e^{-\frac{\tau}{2}}(G_{\alpha (\tau)} \ast u)(e^{-\frac{\tau}{2}} x), \quad x \in \R^n,
\end{align}
where $G_{\alpha(\tau)}(x) = e^{-\frac{|x|^2}{4 \alpha(\tau)}}(4 \pi \alpha(\tau))^{-\frac{n}{2}}$ with $\alpha(\tau) := 1-e ^{-\tau}$, see e.g. the proof of Proposition 3.1 in \cite{GloKisSch24}. \\

Lemma \ref{Le:Potential} implies that $L'_{f}$ extends to a bounded operator on $H$ and we obtain the following proposition on the linearized evolution in $H$, which follows from standard arguments.

\begin{proposition}\label{Prop:LinEvol_Lsigma}
The operator $(L_f,C^{\infty}_{c,r}(\R^n))$ is closable in $H$ and the closure $L_f : \mathcal{D}(L_f) \subseteq H \rightarrow H$ is self-adjoint, has compact resolvent and generates a strongly continuous semigroup $(S_f(\tau))_{\tau \geq 0}$ of bounded operators on $H$. Furthermore, the spectrum of $L_f$ consists only of real eigenvalues and each eigenvalue has finite multiplicity.
\end{proposition}

As we will see, the information on the unstable part of the spectrum of $L_f$ in $H$ is the key for the nonlinear stability analysis. Time-translation invariance implies that for any profile $f$, $1 \in \sigma(L_f)$ with eigenfunction $f'$. The crucial part is the detection of further unstable eigenvalues. In the self-adjoint setting, one can use results from the theory of Schrödinger operators provided sufficiently precise information on $f$ and thus on the potential $V_f$ is given. One of our main results of this paper is the following. 

\begin{theorem}\label{Th:Existence_f0}
For every $d \in \{3,4,5,6\}$, there exists a solution $f_0 \in C^{\infty}([0,\infty))$ to Eq.~\eqref{Eqf}, which is monotonically increasing and satisfies $f_0 > 0$ on
$(0,\infty)$. Furthermore, 
\begin{align*}
  \sigma(L_{f_0}) \subset (-\infty,0) \cup \{1\},
\end{align*}
where $\lambda = 1$ is a simple eigenvalue with eigenfunction $f_0'$.
\end{theorem}

For $d=3$, this has been proven in \cite{BieDon18}, but we include this case for the sake of completeness. For better readability of the paper, the proof, which is based on rigorous computer-assistence, is postponed to Section \ref{Sec:Existence_Spec_Stab}. \\

To study the nonlinear evolution, we have to consider a different functional analytic setup, which will be introduced in the next section.

\section{Nonlinear stability}

\subsection{Setup}
First, state an important feature of (radial) intersection Sobolev spaces.

\begin{lemma} \label{LemmaEmbedding}
For $s_1, s_2 >0$ with $s_1 < \frac{n}{2} < s_2$, the following embedding holds
\begin{align*}
\dot{H}^{s_1}_r(\R^n) \cap \dot{H}^{s_2}_r(\R^n) \hookrightarrow L^{\infty}(\R^n).
\end{align*}
Furthermore, the space $\dot{H}^{s_1}_r(\R^n) \cap \dot{H}^{s_2}_r(\R^n)$ is closed under multiplication, moreover,
\begin{align} \label{Multi}
\Vert u v  \Vert_{\dot{H}^{s_1}_r(\R^n) \cap \dot{H}^{s_2}_r(\R^n)} \lesssim \Vert u  \Vert_{\dot{H}^{s_1}_r(\R^n) \cap \dot{H}^{s_2}_r(\R^n)} \Vert v \Vert_{\dot{H}^{s_1}_r(\R^n) \cap \dot{H}^{s_2}_r(\R^n)}, 
\end{align}
for all $u, v  \in \dot{H}^{s_1}_r(\R^n) \cap \dot{H}^{s_2}_r(\R^n)$.
\end{lemma}

\begin{proof}
The embedding into $L^{\infty}(\R^n)$ follows from the choice $s_1 < \frac{n}{2} < s_2$. A detailed proof can be found in \cite{GloKisSch24}, Appendix A.1. The algebra property follows by the generalized Leibniz rule, see \cite{GraOh14}, Theorem 1, together with the embedding of $\dot{H}^{s_1}_r(\R^n) \cap \dot{H}^{s_2}_r(\R^n)$ into $L^{\infty}(\R^n)$.
\end{proof}

Now we introduce the main function space of this paper and gather some properties.

\begin{definition}
For $(s,k)$ satisfying 
\begin{align}\label{sk-range}
\frac{n}{2} - 1 < s < \frac{n}{2} - 1 + \frac{1}{2(n-2)}, \quad \text{ and } k = n+1,
\end{align}
we define $X_{s}^{k}(\R^n)$ as the completion of the space of radial test functions $C^{\infty}_{c,r}(\R^n)$ with respect to the inner product
\begin{align*}
	\langle u,v \rangle_{X_s^k(\R^n)} := \langle u,v \rangle_{\dot{H}^{s}(\R^n)}  + \langle u,v \rangle_{\dot{H}^{k}(\R^n)}.
\end{align*}
\end{definition}
 
\begin{lemma} \label{embeddinglemma}
For $(s,k)$ as in Eq.~\eqref{sk-range}, the following embeddings hold
\[ X_s^k(\R^n) \hookrightarrow L^{\infty}(\R^n) \cap C_{r}(\R^n) \hookrightarrow H.\] 
Furthermore, $X_s^k(\R^n)$ is closed under multiplication, i.e.,
\begin{align} \label{Embedding}
\Vert u v \Vert_{X_s^k(\R^n)} \lesssim \Vert u \Vert_{X_s^k(\R^n)} \Vert v \Vert_{X_s^k(\R^n)},
\end{align}
for all $u,v \in X_s^k(\R^n)$.
\end{lemma}

\begin{proof}
The $L^{\infty}-$ embedding as well as \eqref{Embedding} follow from Lemma \ref{LemmaEmbedding}. The last embedding is a consequence of the fast decay of the weight function.
\end{proof}

\subsection{The linear evolution on $X_s^k(\R^n) $ }

In this section, we make use of the embedding $X_s^k(\R^n) \hookrightarrow H$ by Lemma \ref{embeddinglemma} to analyze the linearized operator and its corresponding semigroup in $X_s^k(\R^n)$.

\begin{lemma} \label{Functions}
Let $f \in C^{\infty}([0,\infty))$ be a solution to Eq.~\eqref{Eqf}. Then the smooth functions $f'$, $\mathbf{f}=|\cdot|^{-1} f(|\cdot|)$ and $V_f$ belong to $X_s^k(\R^n)$.
\end{lemma}

\begin{proof}
By using the decay properties stated in Lemma \ref{Lemmaf} and Lemma \ref{Le:Potential} one can easily check that $f^{\prime}$ and $V_f$ belong to  $\dot{H}^{\lfloor s \rfloor}_r (\R^n) \cap \dot{H}^{k}_r(\R^n)$ since $\lfloor s \rfloor  \geq \lfloor \frac{n}{2} - 1 \rfloor  >  \frac{n}{2} - 2 $.  By the continuous embedding $\dot{H}^{\lfloor s \rfloor}_r (\R^n) \cap \dot{H}^{k}_r(\R^n) \hookrightarrow X_s^k(\R^n)$ the statement for $f^{\prime}$ and $V_f$ follows. To see that $\mathbf{f}$ belongs to $X_s^k(\R^n)$, we use the fact that $|\partial^{\beta} \mathbf{f}(x)| \lesssim \langle x \rangle^{-1-|\beta|}$ for all $\beta \in \N_0^{d}$ and apply Lemma $2.1$ from \cite{DonSchWit2026}.
\end{proof}

\begin{proposition}
The restrictions $L_0^X := L_0 \big|_{X_s^k(\R^n)}$ and $L_f^X := L_f \big|_{X_s^k(\R^n)}$ of the linear operators $L_0$ and $L_f$ to $X_s^k(\R^n)$ with domains $\mathcal{D}(L_0^X) = \mathcal{D}(L_f^X) = \{ u \in \mathcal{D}(L_f) \cap X_s^k(\R^n) : L_f u  \in X_s^k(\R^n) \}$, respectively, generate strongly continuous one-parameter semigroups of bounded linear operators $(S_0^X(\tau) )_{\tau \geq 0}$ and $(S_f^X(\tau) )_{\tau \geq 0}$ on $X_s^k(\R^n)$, that are given by the restrictions of $\left( S_0(\tau) \right)_{\tau \geq 0}$ and $\left( S_f(\tau) \right)_{\tau \geq 0}$ to $X_s^k(\R^n)$ respectively, i.e. $S_0^X(\tau) = S_0(\tau) \big|_{X_s^k(\R^n)}$ and $S_f^X(\tau) = S_f(\tau) \big|_{X_s^k(\R^n)}$ for all $\tau \geq 0$. Finally, the Schwartz functions $\mc S(\R^n)$ are a core for $L_f$.
\end{proposition}

\begin{proof}
Let $\tau \geq 0$ and $u \in C^{\infty}_{c,r}(\R^n)$, then a straightforward calculation shows
\begin{align} \label{Estimatefreesemigroup}
\Vert S_0(\tau) u \Vert_{X_s^k(\R^n)} \leq e^{\omega_0 \tau} \Vert u \Vert_{X_s^k(\R^n)},
\end{align}
where $\omega_0 := \frac{1}{2} \left( \frac{n}{2}-1-s \right) <0$. Furthermore, since 
\begin{align*}
\mathcal{F}(S_0(\tau) u)(\cdot) = e^{\left(\frac{n}{2}-\frac{1}{2} \right) \tau} e^{(1-e^{\tau})|\cdot|^2} \mathcal{F}(u)(e^{\frac{\tau}{2}} \cdot),
\end{align*}
we find
\begin{align*}
\Vert |\cdot|^s (\mathcal{F}(S_0(\tau) u-u)) \Vert_{L^2(\R^n)} = \Vert |\cdot|^s (e^{(\frac{n}{2}-\frac{1}{2})\tau} e^{(1-e^{\tau})|\cdot|^2} \mathcal{F}(u)(e^{\frac{\tau}{2}} \cdot) - \mathcal{F}(u)) \Vert_{L^2(\R^n)}.
\end{align*}
By the Dominated Convergence Theorem we infer 
\begin{align*}
\lim_{\tau \rightarrow 0^+} \Vert S_0(\tau)u-u \Vert_{X_s^k(\R^n)} = 0.
\end{align*}
Density of $C^{\infty}_{c,r}(\R^n)$ in $X_s^k(\R^n)$ implies that $S_0(\tau)$ defines a strongly continuous semigroup on $X_s^k(\R^n)$. Since  $X_s^k(\R^n)$ is invariant under $\left( S_0(\tau) \right)_{\tau \geq 0}$ we find by Proposition 2.3, Chapter II, \cite{EngelNagel}, that $L_0^X := L_0 \big|_{X_s^k(\R^n)}$ with domain $\mathcal{D}(L_0^X) = \{ u \in \mathcal{D}(L_0) \cap X_s^k(\R^n) : L_0 u \in X_s^k(\R^n) \}$ generates the restricted semigroup $( S_0^X(\tau) )_{\tau \geq 0}$ in $X_s^k(\R^n)$, where $S_0^X(\tau) = S_0(\tau) \big|_{X_s^k(\R^n)}$. The  Bounded Perturbation Theorem (see, e.g., \cite{EngelNagel}, Chapter III, Theorem 1.3) implies that $L_f^X = L_0^X +L_f^{\prime} = (L_0 + L_f^{\prime}) \big|_{X_s^k(\R^n)}$  with $\mathcal{D}(L_f^X) = \mathcal{D}(L_0^X)$ generates the strongly continuous semigroup $(S_f^X(\tau))_{\tau \geq 0}$ on $X_s^k(\R^n)$, which is given by the restriction $S_f^X(\tau) = S_f(\tau) \big|_{X_s^k(\R^n)}$ for all $\tau \geq 0$. This follows since
\begin{align*}
\Vert L_f^{\prime} u \Vert_{X_s^k(\R^n)} = \Vert V_f u \Vert_{X_s^k(\R^n)} \lesssim \Vert V_f \Vert_{X_s^k(\R^n)} \Vert u \Vert_{X_s^k(\R^n)},
\end{align*}
for all $u \in X_s^k(\R^n)$ by \eqref{Embedding}, where $V_f \in X_s^k(\R^n)$ by Lemma \ref{Functions}.\\

The last statement follows from the fact that $\mc S(\R^n)$ is dense in $X_s^k(\R^n)$ and $S^X_0(\tau)$ leaves Schwartz functions invariant, hence $\mc S(\R^n)$ is a core for $L_0^X$. Since $L_f^X$ is a bounded perturbation of $L_0^X$, the space $\mc S(\R^n)$ is also a core for $L_f^X$.
\end{proof}

\begin{lemma} \label{Lemmacompactperturbation}
The semigroup $(S_f^X(\tau))_{\tau \geq 0}$ is a compact perturbation of the free semigroup $(S^X_0(\tau))_{\tau \geq 0}$, meaning that the right-hand side of the variantion of parameters formula
\begin{align*}
S_f^X(\tau) - S^X_0(\tau) = \int_0^{\tau} S_f^X(\tau - \tau^{\prime}) L'_f S^X_0(\tau^{\prime}) d \tau^{\prime},
\end{align*}
is a compact operator for each $\tau \geq 0$.
\end{lemma}

\begin{proof}
The proof is based on a variant of the Rellich-Kondrachov Theorem. The details are provided in Appendix \ref{App1}.
\end{proof}

With this result at hand we are able to prove that there exist at most a finite number of unstable eigenvalues of the linear operator.

\begin{proposition} \label{K0}
The set $K_f := \sigma(L_f^X) \cap \{ \lambda \in \C : \Re(\lambda) \geq 0 \}$ consists of finitely many eigenvalues of $L_f^X$, all of which are real and have finite algebraic multiplicity.
\end{proposition}

\begin{proof}
See Proposition 3.10 in \cite{GloKisSch26}.
\end{proof}

To prove stability of a shrinking solution $F$ with profile $f$ we need to show exponential decay of the semigroup $(S_f^X(\tau))_{\tau \geq 0}$ in $X_s^k(\R^n)$ on a stable subspace. To define a stable subspace we consider for each unstable eigenvalue $\lambda \in K_f$ its Riesz-projection
\begin{align*}
P_{\lambda} := \frac{1}{2 \pi i} \int_{\gamma_{\lambda}} R_{L_f^X}(\lambda^{\prime}) d \lambda^{\prime},
\end{align*}
where $\gamma_{\lambda}$ is a positively oriented circle centered at $\lambda$, contained in $\rho(L_f^X)$ and which, apart from $\lambda$, contains no other spectral points in its interior.\\
For each $d \in \{3,4,5,6\}$ and each shrinker profile $f$ there exists $N_f \in \N$ with $K_f = \{\lambda_j \}_{j=1}^{N_f} \subset \mathbb R$. Furthermore, we can choose $\{ u_j^{i} \}_{i=1}^{N_j}$ with $u_j^{i}$ real valued, as a basis of the corresponding finite dimensional eigenspace to the eigenvalue $\lambda_j$. By definition of the operator $L_f^X$ as the restriction of $L_f$ to $X_s^k(\R^n)$ and the structure of the spectrum of $L_f^X$ it follows that each Riesz projection can be expressed as 
\begin{align*}
P_{\lambda_j} u = \sum_{i=1}^{N_j} P_j^{i} u, \quad \text{with} \quad P_j^{i} u = \alpha_j^{i} \langle u, u_j^{i} \rangle_{L^2_{\sigma}(\R^n)} u_j^{i}, \quad \alpha_j^{i} \in \R \setminus \{ 0 \},
\end{align*}
for $u \in X_{s}^k(\R^n)$. In particular,
\begin{align} \label{ProjectionPj}
\text{rg}(P_{\lambda_j}) = \text{span}(u_j^1,...,u_j^{N_j}),
\end{align}
for all $j \in \{ 1,...,N_f \}$. Furthermore, the projections are mutually transversal with $P_{\lambda_j} P_{\lambda_l} =0$ for $j \neq l$.
In the following we define the projection onto all unstable directions by
\begin{align}
P_f := \sum_{j=1}^{N_f} \sum_{i=1}^{N_j} P_j^{i}.
\end{align}

\begin{proposition}
Let $f_0$ be the shrinker profile from Theorem \ref{Th:Existence_f0} and $L_{f_0}^X = L_0^X + L_{f_0}^{\prime}$ the corresponding linear operator. Then we have $K_{f_0} = \{ 1 \}$, where the geometric and algebraic eigenspaces corresponding to the eigenvalue $\lambda = 1$ coincide, are one-dimensional and spanned by $f_0^{\prime}$.
\end{proposition}

\begin{proof}
By Proposition \ref{K0} it follows that $K_{f_0}$ consists of finitely many eigenvalues of $L_{f_0}^X$, all of which have finite algebraic multiplicity. By Lemma \ref{Functions} we know that $f_0^{\prime}$ belongs to $X_s^k(\R^n)$ and since $f_0^{\prime}$ is an eigenfunction of $L_{f_0}$ by Theorem \ref{Th:Existence_f0} it follows that $f_0^{\prime} \in \mathcal{D}(L^X_f)$ with $L_{f_0}^X f_0 = f_0$. This shows $\{ 1 \} \subseteq K_{f_0}$.
To prove the opposite inclusion let $\lambda \in \sigma_p(L_{f_0}^X) \cap \{ \lambda \in \C : \Re(\lambda) \geq 0 \}$, then $\lambda$ is also an eigenvalue of $L_{f_0}$. By Theorem \ref{Th:Existence_f0} we infer $\lambda =1$. That the geometric and algebraic eigenspaces coincide and are one-dimensional follows from \eqref{ProjectionPj} and Theorem \ref{Th:Existence_f0}.
\end{proof}

The following result will be one of the main ingredients to prove nonlinear stability of a shrinking solution.

\begin{proposition} \label{Semigroupdec}
For all $d \in \{3,4,5,6 \}$ and every shrinker profile $f$ there exists $\nu > 0$ and $C \geq 1$, such that
\begin{align} \label{Semigroupdecay}
\Vert S_f^X(\tau)(1-P_f)u \Vert_{X_s^k(\R^n)} \leq C e^{-\nu \tau} \Vert (1-P_f)u \Vert_{X_s^k(\R^n)},
\end{align}
for all $\tau \geq 0$ and $u \in X_s^k(\R^n)$.
\end{proposition}

\begin{proof}
The proof is a direct consequence of the decomposition of $X_s^k(\R^n)$ into $\ker(P_f)$ and $\rg(P_f)$. A detailed proof can be found in \cite{GloKisSch26}, Proposition 3.8.
\end{proof}

\subsection{The nonlinear time evolution}
From now on, we restrict ourselves to real valued functions and use again the symbol $X_s^k(\R^n)$ to denote the corresponding closed subspace of $X_s^k(\R^n)$.
We recall the definition of the nonlinear operator for a suitable function $\varphi : \R^n \rightarrow \R$
\begin{align*}
[N(\varphi)](\xi) = \frac{n-3}{|\xi|^3} ( \eta^{\prime}(f(|\xi|)) |\xi| \varphi (\xi) + \eta(f(|\xi|))- \eta(f(|\xi|)+|\xi| \varphi (\xi))),
\end{align*}
for $\xi \in \R^n$ and $\eta(z) = \frac{1}{2} \sin(2 z)$ for $z \in \R$. By the following result we infer that $N$ is locally Lipschitz continuous on $X_s^k(\R^n)$ which is necessary to construct strong solutions to our problem.

\begin{lemma}
The nonlinearity $N$ given by \eqref{Nonlinearity} extends to a map $N: X_s^k(\R^n) \rightarrow X_s^k(\R^n)$ satisfying 

\begin{align*}
\Vert N(\varphi) - N(\phi) \Vert_{X_s^k(\R^n)} \leq \gamma(\Vert \varphi \Vert_{X_s^k(\R^n)}, \Vert \phi \Vert_{X_s^k(\R^n)} )(\Vert \varphi \Vert_{X_s^k(\R^n)} + \Vert \phi \Vert_{X_s^k(\R^n)}) \Vert \varphi - \phi \Vert_{X_s^k(\R^n)},
\end{align*}

for all $\varphi, \phi \in X_s^k(\R^n)$, where $\gamma : [0, \infty) \times [0,\infty) \rightarrow [0, \infty)$ is a continuous function.
\end{lemma}

\begin{proof}
By density if suffices to prove the inequality for functions belonging to $C^{\infty}_{c,r}(\R^n)$. 
We first establish an auxiliary identity for functions $u \in C^3(\R)$ with $u^{\prime \prime}(0) = 0$. Let $a, b, c \in \R$, then by three times application of the Fundamental Theorem of Calculus we have
 
\begin{align*}
&u(a+c) - u(a+b) - u^{\prime}(a)(c-b)\\
&=(c-b) \int_{0}^1 (b + x(c-b)) \int_0^1 (a+ w(b+x(c-b))) \int_0^1 u^{\prime \prime \prime} (z (a+ w(b+x(c-b)))) dz dw dx.
\end{align*}

For the nonlinearity we obtain
\begin{align*}
&[N(\varphi)- N(\phi)](\xi)\\
&= \frac{n-3}{|\xi|^3}\bigg( \eta \Big (f(|\xi|) + |\xi| \phi (\xi) \Big ) - \eta \Big (f(|\xi|) + |\xi| \varphi (\xi) \Big ) - \eta^{\prime} \Big (f(|\xi| \Big )|\xi| (\phi(\xi) - \varphi(\xi)) \Big ).
\end{align*}
Since $\eta$ is an odd function we have $\eta^{\prime \prime}(0) = 0$ and the above identity yields
\begin{align*}
\frac{[N(\varphi)- N(\phi)](\xi)}{n-3} \\
= \int_0^1 \int_0^1 \int_0^1 \Big  (\phi(\xi))&-\varphi(\xi)\Big  ) \Big  (\varphi(\xi)+x (\phi(\xi)-\varphi(\xi))\Big )\Big (\mathbf{f}(|\xi|) + y(\varphi(\xi) + x(\phi(\xi)- \varphi(\xi)))\Big )\\
&\cdot \eta^{(3)}\bigg (|\xi| z \Big(\mathbf{f}(|\xi|) + y \big (\varphi(\xi)+x ( \phi(\xi)- \varphi(\xi)) \big ) \Big )\bigg  ) dz dy dx.
\end{align*}
Note that $\eta^{(3)}$ is an even function having all derivatives bounded. We use the generalized Schauder estimate given in \cite{Glo2025}, Proposition A.1, which reads
\begin{align}\label{Eq:Schauder}
\begin{split}
\| u_1 u_2 u_3 \eta^{(3)}(|\cdot| v) \|_{X_s^k(\R^n)} &  \lesssim \| u_1 u_2 u_3 \eta^{(3)}(|\cdot| v) \|_{\dot H^{k_1}\cap \dot H^k(\R^n)}  \\
& \lesssim  \prod_{i = 1} ^{3} \|u_i \|_{X_s^k(\R^n)} \sum_{j = 0}^{k} \|v \|_{X_s^k(\R^n)}^{2j},
\end{split}
\end{align}
for $k_1 := \lfloor\frac{n}{2}-2 \rfloor$, where $u_1,u_2,u_3, v \in C^{\infty}_{c,r }(\R^n)$ and $v$ is real-valued. By inspection, the same bound holds for $u_3$ being replaced by $\mathbf{f}(|\cdot|)$, respectively for $v \in C^{\infty}_r(\R^n) \cap X^{k}_s(\R^n)$.
\end{proof}

The last ingredient to construct strong solutions to \eqref{mainproblem} is the following characterization of the initial data operator. We recall that $U$ is given by
\begin{align*}
U(\mathbf{h},T) = \sqrt{T} \mathbf{f}(\sqrt{T} |\cdot|) + \sqrt{T} \mathbf{h}(\sqrt{T} |\cdot|) - \mathbf{f}(|\cdot|),
\end{align*}
for $\mathbf{h} \in X_s^k(\R^n)$.

\begin{lemma} \label{InitialMap}
Let $0 < \delta \leq \frac{1}{2}$. Then the map 
\begin{align*}
T \mapsto U(\mathbf{h}, T):[1-\delta,1+ \delta] \rightarrow X_s^k(\R^n),
\end{align*}
as defined in \eqref{Initialdataoperator} is continuous for $\mathbf{h} \in X_s^k(\R^n)$. Furthermore, we have 
\begin{align} \label{Initialestimate}
\Vert U(\mathbf{h},T)\Vert_{X_s^k(\R^n)} \lesssim \Vert \mathbf{h} \Vert_{X_s^k(\R^n)} + |T-1|,
\end{align}
for all $\mathbf{h} \in X_s^k(\R^n)$ and $T \in \left[ \frac{1}{2}, \frac{3}{2} \right]$.
\end{lemma}

\begin{proof}
To show continuity of the given map, let $\mathbf{h} \in X_s^k(\R^n)$, $T_1, T_2 \in [1- \delta, 1+ \delta]$ and write 
\begin{align*}
 &U(\mathbf{h}, T_1) -  U(\mathbf{h}, T_2)\\
 &= (\sqrt{T_1} - \sqrt{T_2})[\mathbf{f} + \mathbf{h}](\sqrt{T_1} |\cdot|) + \sqrt{T_2} ( [\mathbf{f} + \mathbf{h}](\sqrt{T_1} |\cdot|) - [\mathbf{f} + \mathbf{h}](\sqrt{T_2} |\cdot|) ).
\end{align*}
Next let $\varepsilon >0$, then there exists $\chi \in C^{\infty}_{c,r}(\R^n)$, such that $\Vert \chi - [\mathbf{f} + \mathbf{h}] \Vert_{X_s^k(\R^n)} \leq \varepsilon$. By this we can write the second term from above as
\begin{align*}
[\mathbf{f} +\mathbf{h}](\sqrt{T_1} |\cdot|) - [\mathbf{f} +\mathbf{h}](\sqrt{T_2} |\cdot|) &= ([\mathbf{f} + \mathbf{h}](\sqrt{T_1} |\cdot|) - \chi(\sqrt{T_1} \cdot)) + (\chi (\sqrt{T_1} \cdot) - \chi (\sqrt{T_2} \cdot))\\
& +( \chi (\sqrt{T_2} \cdot) - [\mathbf{f} + \mathbf{h}](\sqrt{T_2} |\cdot|)),
\end{align*}
to infer
\begin{align*}
\lim_{T_2 \rightarrow T_1} \Vert  U(\mathbf{h}, T_1) -  U(\mathbf{h}, T_2) \Vert_{X_s^k(\R^n)} \leq C \varepsilon,
\end{align*}
for some $C>0$. Here we used the fact that $\lim_{T_2 \rightarrow T_1} \Vert \chi (\sqrt{T_1} \cdot) - \chi (\sqrt{T_2} \cdot) \Vert_{X_s^k(\R^n)} =0$, as the function $\chi$ is smooth and compactly supported. As $\varepsilon >0$ was chosen arbitrarily, continuity follows. Next we show \eqref{Initialestimate}. Let $\mathbf{h} \in X_s^k(\R^n)$ and $T \in [\frac{1}{2}, \frac{3}{2} ]$, then we have
\begin{align*}
\Vert U(\mathbf{h}, T) \Vert_{X_s^k(\R^n)} &= \Vert \sqrt{T} \mathbf{f}(\sqrt{T} |\cdot|) + \sqrt{T} \mathbf{h}(\sqrt{T} |\cdot|) - \mathbf{f} \Vert_{X_s^k(\R^n)} \\
&\leq  \sqrt{T} \Vert \mathbf{h} (\sqrt{T} |\cdot|)  \Vert_{X_s^k(\R^n)} + \sqrt{T} \Vert  \mathbf{f}(\sqrt{T} |\cdot|) - \mathbf{f} \Vert_{X_s^k(\R^n)} + | \sqrt{T} -1| \Vert \mathbf{f} \Vert_{X_s^k(\R^n)} \\
&\lesssim \Vert \mathbf{h} \Vert_{X_s^k(\R^n)} + |T-1| \Vert \mathbf{f} \Vert_{X_s^k(\R^n)} + \Vert  \mathbf{f}(\sqrt{T} |\cdot|) - \mathbf{f} \Vert_{X_s^k(\R^n)}.
\end{align*}
To estimate the remaining term we do the following. For $z \in \R^+$ we obtain
\begin{align*}
\mathbf{f}( \sqrt{T} z) - \mathbf{f} (z) = z (\sqrt{T}-1) \int_0^1 \mathbf{f}^{\prime}(z ((\sqrt{T} -1) t +1)) dt,
\end{align*}
by the Fundamental Theorem of Calculus, where we have by definition
\begin{align*}
\mathbf{f}^{\prime} (z) = \frac{f^{\prime}(z)}{z} - \frac{f(z)}{z^2}, \quad z \in \R^+.
\end{align*}
This implies
\begin{align*}
\frac{\Vert \mathbf{f}(\sqrt{T} |\cdot|) - \mathbf{f} \Vert_{X_s^k(\R^n)}}{|\sqrt{T}-1|} &\leq \int_0^1 (1+t(\sqrt{T}-1))^{-1}(\Vert f^{\prime} (|\cdot|(1+ t(\sqrt{T}-1))) \Vert_{X_s^k(\R^n)} dt\\
&+ \int_0^1 (1+t(\sqrt{T}-1))^{-1} \Vert \mathbf{f}(|\cdot|(1+ t(\sqrt{T}-1))) \Vert_{X_s^k(\R^n)} dt.
\end{align*}
Using the fact that
\begin{align*}
&\Vert u(|\cdot|(1+ t(\sqrt{T}-1))) \Vert_{X_s^k(\R^n)}\\
&\leq \max \{ (1+ t(\sqrt{T}-1))^{s-\frac{n}{2}}, (1+ t(\sqrt{T}-1))^{k-\frac{n}{2}} \} \Vert u \Vert_{X_s^k(\R^n)},
\end{align*}
for $u \in X_s^k(\R^n)$ gives
\begin{align*}
\Vert \mathbf{f}(\sqrt{T} |\cdot|) - \mathbf{f} \Vert_{X_s^k(\R^n)} &\lesssim |\sqrt{T}-1| (\Vert f^{\prime} \Vert_{X_s^k(\R^n)} + \Vert \mathbf{f}  \Vert_{X_s^k(\R^n)}) \\
&\lesssim |T-1|,
\end{align*}
where we exploited that all components of the integrals are positive and bounded due to $k > s > \frac{n}{2}-1$ and that $f^{\prime},\mathbf{f} \in X_s^k(\R^n)$ by Lemma \ref{Functions}.
\end{proof}

\subsection{Construction of strong solutions} In this section we only give a short outline how to construct strong solutions to \eqref{mainproblem} for any shrinker profile $f$. A detailed proof can be found in \cite{GloKisSch26}, section 4.2. We emphasize that all the arguments used there are applicable to our problem since they are based on the exponential decay of the semigroup together with the local Lipschitz continuity of the nonlinear operator.\\
We call a function $\varphi$ a strong solution to \eqref{mainproblem}, if $\varphi \in C([0, \infty), X_s^k(\R^n))$ and if it satisfies
\begin{align} \label{Strongsolution}
\varphi(\tau) = S_f^X(\tau) U(\mathbf{h}_0,T) + \int_0^{\tau} S_f^X(\tau- \tau^{\prime}) N(\varphi(\tau^{\prime}) d \tau^{\prime}, \quad \tau \geq 0.
\end{align}
If $f$ is any shrinker profile, there might exist more unstable eigenvalues than $\lambda =1$ which is caused by the time translation symmetry of the blowup time $T$. This lack of information on the spectral problem only permits a codimension stability result. More precisely, by constructing a codimension $N_f$ Lipschitz manifold $M_f$ in the initial data space we correct the initial perturbations of $\mathbf{f}$ along the unstable directions such that the corresponding solution has the same blowup regime as the shrinking solution $\mathbf{f}$. For that we set $T=1$ which reduces the initial data map to $U(\mathbf{h}_0,1) = \mathbf{h}_0$.
We define the Banach space
\begin{align*}
\mathcal{X}_{s}^k := \{ \varphi \in C([0,\infty),X_s^k(\R^n)) : \Vert \varphi \Vert_{\mathcal{X}_{s}^k} := \sup_{\tau \geq 0} e^{\nu \tau} \Vert \varphi(\tau) \Vert_{X_{s}^k(\R^n)} < \infty \},
\end{align*}
with $\nu >0$ being the constant from \eqref{Semigroupdecay}. To run a fixed point argument on a stable subspace, we define for each unstable eigenvalue $\lambda_j \in K_f$, $j \in \{1,...,N_f \}$ a correction term
\begin{align}\label{Correction1}
C_j(\varphi,u) := P_j u + P_j \int_0^{\infty} e^{- \lambda_j \tau^{\prime}} N(\varphi(\tau^{\prime})) d \tau^{\prime},
\end{align}
and set
\begin{align}\label{Correction2}
C_f(\varphi,u) := \sum_{j=1}^{N_f} C_j(\varphi,u),
\end{align}
for $\varphi \in \mathcal{X}_{s}^k$ and $u \in X_s^k(\R^n)$.  With this at hand we are able to define the operator 
\begin{align} \label{MapK}
[K_f(\varphi,u)](\tau) := S_f^X(\tau)(u- C_f(\varphi,u)) + \int_0^{\tau} S_f^X(\tau-\tau^{\prime}) N(\varphi(\tau^{\prime}))d \tau^{\prime}, \quad \tau \geq 0.
\end{align}
The first step to prove stability is to apply the contraction mapping principle to $K_f$ provided the initial data $u$ is chosen sufficient small. This gives a unique fixed point $\varphi_u$ of $K_f$ associated to the initial condition $u$ and we abbreviate $C_f(u) := C_f(\varphi_u,u)$. The next step is to show that for every $M >0$ there exists a sufficiently small $\delta >0$ such that if $\mathbf{h}_0 \in X_s^k(\R^n)$ satisfies $\Vert \mathbf{h}_0 \Vert_{X_s^k(\R^n)} \leq \frac{\delta}{M^2}$, then there exist parameters $a_{i,j} \in [-\frac{\delta}{M}, \frac{\delta}{M} ]$, $(i,j) \in I$ with $I := \{(i,j) : 1 \leq i \leq N_j : 1 \leq j \leq N_f \}$ such that 
\begin{align*}
C\Big(\mathbf{h}_0 + \sum_{(i,j) \in I} a_{i,j} u_j^{i}\Big) =0.
\end{align*}
Furthermore, the parameters $a_{i,j} = a_{i,j}(\mathbf{h}_0)$ depend Lipschitz continuously on the data, i.e.
\begin{align*}
\sum_{(i,j) \in I} | a_{i,j}(\mathbf{h}_0) - a_{i,j}(\mathbf{h}_1) | \lesssim \Vert \mathbf{h}_0 - \mathbf{h}_1 \Vert_{X_{s}^k(\R^n)},
\end{align*}
for all $\mathbf{h}_0, \mathbf{h}_1 \in X_{s}^k(\R^n)$ satisfying $\Vert \mathbf{h}_0 \Vert_{X_{s}^k(\R^n)}, \Vert \mathbf{h}_1 \Vert_{X_{s}^k(\R^n)} \leq \frac{\delta}{M^2}$.
These results together with smoothing properties of the free semigroup $(S_0(\tau))_{\tau \geq 0}$ and Sobolev embeddings lead to a classical solution $\varphi \in C^{\infty}([0, \infty) \times \R^n)$ to the problem
\begin{align} \label{EQ11}
\begin{cases} \partial_{\tau} \varphi = L_f \varphi (\tau, \cdot) + N(\varphi(\tau, \cdot)), \\ \varphi(0, \cdot) = \mathbf{h}_0 + \sum_{(i,j) \in I} a_{i,j} u_j^{i},
\end{cases}
\end{align}
where $\mathbf{h}_0 \in \mathcal{S}_r(\R^n)$ satisfies $\Vert \mathbf{h}_0 \Vert_{X_s^k(\R^n)} \leq \frac{\delta}{M^2}$. Furthermore, $\varphi$ satisfies $\Vert \varphi(\tau, \cdot) \Vert_{X_{s}^k(\R^n)} \leq \delta e^{-\nu \tau}$ for $\tau \geq 0$.

\subsection{Proof of Theorem \ref{Th:CodimensionStability}}
Let $\mathbf{h}_0 \in \mathcal{S}_r(\R^n)$, then we can identify $\mathbf{h}_0$ also with a radial Schwartz function on $\R^d$. For $H_0(x) = x \mathbf{h}_0(|x|)$, $x \in \R^d$ we have an equivalence of Sobolev norms
\begin{align*}
\Vert H_0 \Vert_{\dot{H}^s \cap \dot{H}^k(\R^d)} \simeq \Vert \mathbf{h}_0 \Vert_{X_s^k(\R^n)},
\end{align*}
see  \cite{Glo22}, Proposition A.5. Hence, we can choose $\varepsilon > 0$ small enough to guarantee $\Vert \mathbf{h}_0 \Vert_{X_s^k(\R^n)} \leq \frac{\delta}{M^2}$ and $\Vert H_0 \Vert_{\dot{H}^s \cap \dot{H}^k(\R^d)} \leq \varepsilon$.
Let $\varphi \in C^{\infty}([0, \infty) \times \R^n)$ be the resulting solution to \eqref{EQ11}. Translating everything back into physical coordinates $x \in \R^n$ and $t \in [0, 1)$ via
\begin{align*}
w(t,x) = (1-t)^{-\frac{1}{2}} \left( f \left( \frac{|x|}{\sqrt{1-t}} \right)\left( \frac{|x|}{\sqrt{1-t}} \right)^{-1} + \varphi \left( \log \left( \frac{1}{1-t} \right), \frac{x}{\sqrt{1-t}} \right) \right),
\end{align*}
gives a solution $w$ to \eqref{Problemw}. This leads to a classical solution $U \in C^{\infty}([0,1) \times \R^d)$ via
\begin{align*}
U(t,x) &=\frac{x}{\sqrt{1-t}} \left( f \left( \frac{|x|}{\sqrt{1-t}} \right)\left( \frac{|x|}{\sqrt{1-t}} \right)^{-1} + \varphi \left( \log \left( \frac{1}{1-t} \right), \frac{x}{\sqrt{1-t}} \right) \right)\\
&= F \left( \frac{x}{\sqrt{1-t}} \right) + H \left(t, \frac{x}{\sqrt{1-t}} \right),
\end{align*}
where $H(t, x) := x \varphi \left( \log \left( \frac{1}{1-t} \right),x \right)$ for $x \in \R^d$ and $t \in [0,1)$. We point out that the spatial variable $x$ now belongs to $\R^d$ instead of $\R^n$ which is possible since $\varphi$ is radial. Furthermore, we have
\begin{align*}
\Vert H(t, \cdot) \Vert_{\dot{H}^s \cap \dot{H}^k(\R^d)} \simeq \Vert \varphi(\tau, \cdot) \Vert_{X_{s}^k(\R^n)} \leq \delta e^{-\nu \tau} = \delta (1-t)^{\nu},
\end{align*}
which shows \eqref{DecayH}.

\subsection{Stable blowup via $f_0$}
In case of the shrinker profile $f_0$ we know that the only unstable eigenvalue is $\lambda =1$ and the corresponding (algebraic and geometric) eigenspace is one-dimensional and spanned by $f_0^{\prime}$. Since this instability is caused by the freedom of choosing the blowup time $T$ we exploit this causality to prove stability of $f_0$ without correcting the initial data. For that, we consider Duhamel's formula \eqref{Strongsolution} with initial data $U(\mathbf{h}_0,T)$ for some $T >0$. As before, exponential decay of the semigroup $(S_{f_0}^X(\tau))_{\tau \geq 0}$ on the stable subspace $(1-P_{f_0})$ together with the local Lipschitz continuity of the nonlinearity $N$ lead to a unique fixed point $\varphi_u$ of the map $K$ given by \eqref{MapK} provided $u$ is chosen small enough.\\
To prove vanishing of the correction term \eqref{Correction2} we make use of Lemma \ref{InitialMap} that allows for a Taylor expansion of $\sqrt{T} \mathbf{f}_0 (\sqrt{T} \cdot) - \mathbf{f}_0$ in $T=1$. More precisely, we have
\begin{align} \label{TaylorofU}
\mathcal{U}(\mathbf{h}_0,T) = \sqrt{T} \mathbf{h}_0(\sqrt{T} \cdot) + C (1-T) f_0^{\prime}(\cdot) + (1-T)^2 R(T, \cdot),
\end{align}
for some $C \in \R \setminus \{  0 \}$ and a remainder term satisfying $\Vert R(T,\cdot) \Vert_{X_s^k(\R^n)} \lesssim 1$ for $T$ close to $1$. By definition of the correction term \eqref{Correction2} it follows that $C(\varphi_{U(\mathbf{h}_0,T)},U(\mathbf{h}_0,T)) = 0$ if and only if
\begin{align} \label{scalarproduct}
\left\langle P_{f_0} U(\mathbf{h}_0,T) + \int_0^{\infty} e^{-\tau^{\prime}} N(\varphi_{U(\mathbf{h}_0,T)}(\tau^{\prime})) d \tau^{\prime}, f_0^{\prime} \right\rangle = 0.
\end{align}
The Taylor expansion \eqref{TaylorofU} together with the local Lipschitz continuity of $N$ imply that \eqref{scalarproduct} is equivalent to 
\begin{align*}
T = W(T) +1,
\end{align*}  
for a continuous function $W$. The main idea is to show that for sufficient constants $M,\delta >0$ the map $T \mapsto W(T) +1$ maps the interval into itself. The Intermediate Value Theorem implies the existence of a strong solution $\varphi$ that solves \eqref{Strongsolution}. In summary, we have the following result.

\begin{proposition} \label{PropSolution}
There exists $M>0$ sufficiently large and $\delta >0$ sufficiently small, such that for all real-valued $\mathbf{h}_0 \in X_s^k(\R^n)$ with $\Vert \mathbf{h}_0 \Vert_{X_s^k} \leq \frac{\delta}{M^2}$, there exists a $T = T(\mathbf{h}_0) \in [1- \frac{\delta}{M}, 1+ \frac{\delta}{M} ]$ and a unique solution $\varphi \in C([0, \infty), X_s^k(\R^n))$ satisfying \eqref{Strongsolution} for all $\tau \geq 0$, such that 
\begin{align} \label{estimateforvarphi}
\Vert \varphi(\tau) \Vert_{X_s^k(\R^n)} \leq \delta e^{- \nu \tau}, \quad \forall \tau \geq 0,
\end{align}
where $\nu >0$ is the constant from Proposition \ref{Semigroupdec}.
\end{proposition}
To generate classical solutions to our problem we again exploit the smoothing of the free semigroup. It follows that if $\mathbf{h}_0 \in \mathcal{S}_r(\R^n)$ satisfies $\Vert \mathbf{h}_0 \Vert_{X_s^k(\R^n)} \leq \frac{\delta}{M^2}$, we infer a classical solution $\varphi \in C^{\infty}([0, \infty) \times \R^n)$ to the problem
\begin{align} \label{EQ11}
\begin{cases} \partial_{\tau} \varphi = L_{f_0} \varphi (\tau, \cdot) + N(\varphi(\tau, \cdot)), \\ \varphi(0, \cdot) = U(\mathbf{h}_0,T),
\end{cases}
\end{align}
with the blowup time $T = T(\mathbf{h}_0)$ from Proposition \ref{PropSolution}.

\subsection{Proof of Theorem \ref{TheoremStability}}
The proof is analogous to the proof of Theorem \ref{Th:CodimensionStability}. The only difference lies in the self-similar variables $(\tau,\xi)$ used to transform $\varphi$ back to $U$. They are given by $\tau = \log \left( \frac{T}{T-t} \right)$ and $\xi = \frac{x}{\sqrt{T-t}}$ with $T = T(\mathbf{h}_0)$ being the blowup time from Proposition \ref{PropSolution}.


\section{Existence of a spectrally stable profile $f_0$}\label{Sec:Existence_Spec_Stab}

The goal of this section is to prove Theorem \ref{Th:Existence_f0}. For this, we generalize the analytic framework of \cite{BieDon18} to higher dimensions and adapt the methods of \cite{DonSch26} to prove rigorous bounds using computer-assistance. The implementation and all required coefficient files are available at \url{https://git.uibk.ac.at/csaw4329/hmhf_paper_code}.

We start by introducing a closed form approximate solution to Eq.~\eqref{Eqf}.
\begin{definition}\label{def:ftil_0}
For $d\in\{3, 4, 5, 6\}$, we define
\begin{equation*}
    \ftil_0(y) := 2\arctan\left(\sum_{k=0}^{K} c_k(\ftil_0) T_{2k+1}\left(\frac{y}{\sqrt{y^2+4}}\right)\right),
\end{equation*}
with rational coefficients $(c_k(\ftil_0))_{k=0}^{K}$ given in Tables~\ref{tab:coef_f0_d3}-\ref{tab:coef_f0_d6}, and $K$ determined in each case by the number of listed coefficients. The corresponding CVS-file is \texttt{coefs\_f}.\footnote{\ Note that we use a different transformed variable than in \cite{BieDon18}, which is convenient for later steps. As a result, the coefficients and the subsequently derived bounds are not directly comparable.}
\end{definition}

We note that this ansatz simplifies the nonlinearity in Eq.~\eqref{Eqf} as
\begin{equation*}
    \sin(2\ftil_0)=\sin(4\arctan(g)) = \frac{4(g-g^3)}{(g^2+1)^2},
\end{equation*}
where we let
\begin{equation}\label{eqn:g_chebsum}
    g(y) = \sum_{k=0}^{K} c_k(\ftil_0) T_{2k+1}\left(\frac{y}{\sqrt{y^2+4}}\right).
\end{equation}
Furthermore, we use only odd Chebyshev polynomials, as extending Eq.~\eqref{Eqf} to negative $y$ shows that the solution must be odd.

\begin{remark}
The coefficients $(c_k(\ftil_0))_{k=0}^K$ are computed by inserting the ansatz for $\ftil_0$ into Eq.~\eqref{Eqf} and applying the variable transformation $x=\frac{y}{\sqrt{y^2+4}}$, with $x\in[0,1]$. We impose that the transformed differential equation must be satisfied at the points
\begin{equation*}
    x_k = \cos\left(\tfrac{(2k-1)\pi}{4(K+1)}\right), \qquad \text{for } k=1, 2, ..., K+1,
\end{equation*}
which are the positive Chebyshev nodes of degree $2K+2$. The resulting system of nonlinear equations is then solved numerically. The results are converted to rational numbers using an absolute tolerance based on the smallest coefficient and the number of coefficients.
\end{remark}

\subsection{Existence of $f_0$}
In this section, we prove that Eq.~\eqref{Eqf} has a solution $f_0$ which is well approximated by $\ftil_0$. To formulate the result, we define the following weight functions.

\begin{definition}\label{def:weights}
For $d\in\{3, 4, 5, 6\}$, we define $p_1$, $p_2$, $p_3 : (0,\infty)\to\R$ as
\begin{align*}
    p_1(y) &:= \frac{\sqrt{y^2+4}}{y}, \\
   p_2(y) &:= \frac{(y^2+4)^{5/2}}{y}\Bigg(\sum_{k=0}^{K_N} c_{k,N}(p_2)\ y^{k}\Bigg)\Bigg/\Bigg(\sum_{k=0}^{K_D} c_{k,D}(p_2)\ y^{k}\Bigg), \\
    p_3(y) &:= \frac{(y^2+4)^{3/2}}{4},
\end{align*}
with the coefficients $(c_{k,N}(p_2))_{k=0}^{K_N}$ and $(c_{k,D}(p_2))_{k=0}^{K_D}$, respectively given in the CSV-files \texttt{coefs\_p2\_n} and \texttt{coefs\_p2\_d}, and $K_N, K_D$ determined in each case by the number of listed coefficients. 
\end{definition}

\begin{remark}
The weight $p_2$ is positive and for $d\in\{3, 5\}$ we have $c_{k,N}(p_2)=c_{k,D}(p_2)=0$ for all odd $k$. Furthermore, the behavior of $p_2$ at both $0$ and $\infty$ is given by
\begin{equation*}
    p_2 \sim \frac{(y^2 + 4)^{5/2}}{(d+1)y^3 + (4d + 16)y}.
\end{equation*}

In this paper we only include in Appendix \ref{Appendix:Coeff} those coefficients which are necessary to to derive all other quantities and to entirely reproduce our results. We refrain from including all coefficients due to their length and instead refer to the files in the online repository. 
\end{remark}

To simplify the notation, we denote the $L^\infty$-norm on $[0,\infty)$ by $\|\cdot\|_\infty$ in the following.

\begin{definition}\label{def:weighted_norm}
We define 
\begin{equation*}
    \|f\|_{    \mc Y} := \|p_1 f\|_\infty + \|p_3 f'\|_\infty
\end{equation*}
and denote by $\mc Y$ be the space $C^1([0,\infty))$ equipped with $\| \cdot \|_{\mc Y}$.
\end{definition}

We now state the first main result of this section.

\begin{theorem}\label{thm:closeness}
Let $d\in\{3, 4, 5, 6\}$, then there exists $f_0\in C^\infty([0,\infty))$ solving Eq.~\eqref{Eqf} with
\begin{equation*}
    \|f_0 - \ftil_0 \|_{\mc Y} \leq \varepsilon = 
    \left\{\begin{array}{@{}l l}
        \frac{1}{20940}, & d=3, \\ 
        \frac{1}{143829}, & d=4, \\
        \frac{1}{1268555}, & d=5, \\
        \frac{1}{15875639}, & d=6.
    \end{array}\right.
\end{equation*}
\end{theorem}

The proof of Theorem \ref{thm:closeness} is based on a fixed-point argument. For this, we insert the ansatz 
\[ f_0 = \ftil_0 + \delta\]
into Eq.~\eqref{Eqf}  and write the resulting equation as 
\begin{equation}\label{eqn:delta_ode}
    \Lc(\delta) = \Rc(\ftil_0) + \Nc(\delta) 
\end{equation}
with
\begin{align}
    \Lc(\delta) & := -\delta'' - \left(\frac{d-1}{y} - \frac{y}{2}\right)\delta' + \frac{d-1}{y^2}\delta 
        - \frac{2(d-1)}{y^2}\sin^2(\ftil_0)\delta \label{eqn:linear_op} \\
    \Rc(\ftil_0) & : = \ftil_0'' + \left(\frac{d-1}{y} - \frac{y}{2}\right)\ftil_0' 
        - \frac{d-1}{2y^2}\sin(2\ftil_0) \label{eqn:remainder_op} \\
    \begin{split}
    \Nc(\delta) &: =\frac{d-1}{y^2}\cos(2\ftil_0)\delta + \frac{d-1}{2y^2}\sin(2\ftil_0) 
        - \frac{d-1}{2y^2}\sin(2\ftil_0 + 2\delta) \\
    &= \frac{d-1}{y^2}\big(\cos(2\ftil_0)\delta - \cos(2\ftil_0 + \delta)\sin(\delta)\big).
    \end{split} \label{eqn:nonlinear_op}
\end{align}

The obvious strategy is to show that
\begin{equation*}
    \Kc(\delta)=\Lc^{-1}\big(\Rc(\ftil_0) + \Nc(\delta)\big)
\end{equation*}
is a contraction on the closed subset
\begin{equation}\label{eqn:functionspace_delta}
    \mc Y_{\varepsilon}  =\{\delta\in C^1([0, \infty)),\ \|\delta\|_ {\mc Y} \leq\varepsilon\}.
\end{equation}

A formal inverse of $\Lc$ can be written as
\begin{equation*}
    \Lc^{-1}(\delta)(y) = \int_0^\infty G(x,y)\delta(x)\,dx
\end{equation*}
where $G$ is the Green's function defined by
\begin{equation*}
    G(x,y) = \frac{-1}{W(w_0,w_1)(x)} 
    \left\{\begin{array}{@{}l l}
        w_0(y)w_1(x), & y\leq x,\\
        w_0(x)w_1(y), & x < y ,
    \end{array}\right.
\end{equation*}
with $\{w_0, w_1\}$ a fundamental system of $\Lc$ and $W$ denoting the Wronskian. However, because of the complicated form of $\ftil_0$, it is impossible to find closed form expressions for $w_0$ and $w_1$, which would be required to obtain explicit bounds. So instead of working with $\Lc$ directly, we provide closed form approximations $\wtil_0$ and $\wtil_1$ and use them to define an approximate linear operator $\Lctil$, see Section \ref{Subsubsec:Approx_L}. Consequently, we consider the problem
\begin{equation}\label{eqn:FP}
    \widetilde{\Lc}(\delta)=\Rc(\ftil_0) + \Nc(\delta) + (\widetilde{\Lc}-\Lc)(\delta)
\end{equation}
and show that
\begin{equation}\label{eqn:contraction_op_ktil}
    \widetilde{\Kc}(\delta)=\widetilde{\Lc}^{-1}\big(\Rc(\ftil_0) + \Nc(\delta) + (\widetilde{\Lc}-\Lc)(\delta)\big)
\end{equation}
is a contraction on $ \mc Y_{\varepsilon}$. For $\Lctil^{-1} \alpha$ to be in $ \mc Y$, we impose restrictions on the endpoint behavior of $\alpha$. The admissible behavior is enforced through the weight $p_2$, with the actual defining properties given in Section~\ref{sec:inverse_op}. The bound on the inverse operator then has the form
\begin{equation*}
    \|\Lctil^{-1} \alpha\|_{\mc Y} \leq c_{\Lctil} \|p_2 \alpha\|_\infty,
\end{equation*}
for some constant $c_{\Lctil}>0$. Consequently, the contraction argument requires bounds on
\begin{equation*}
    \|p_2 \Rc(\ftil_0)\|_\infty,\quad \|p_2 \Nc(\delta)\|_\infty,\quad \|p_2(\Nc(\delta) - \Nc(\gamma))\|_\infty,\quad \|p_2(\Lctil - \Lc)(\delta)\|_\infty,
\end{equation*}
for $\delta, \gamma \in \mc Y_\varepsilon$.


\subsubsection{Bounds for the remainder term}\label{sec:remainder}
We start be rewriting Eq.~\eqref{eqn:remainder_op} using that $\ftil_0(y) = 2\arctan(g(y))$, for $g$ defined in (\ref{eqn:g_chebsum}), giving
\begin{align}
    \Rc(\ftil_0) &= \frac{2g''(1+g^2) - 4g(g')^2}{(1+g^2)^2} + \left(\frac{d-1}{y}-\frac{y}{2}\right)\frac{2g'}{1+g^2} 
        - \frac{d-1}{2y^2}\frac{4(g-g^3)}{(1+g^2)^2} \nonumber \\
    &= \frac{2}{1+g^2} \left(g'' + \left(\frac{d-1}{y} - \frac{y}{2} - \frac{2gg'}{1+g^2}\right)g' 
        + \frac{d-1}{y^2}\frac{g^3-g}{1+g^2} \right). \label{eqn:remainder_g}   
\end{align}
Since $T_{2k+1}(x)=x P_k(x^2)$ for a polynomial $P_k$, we have that
\begin{equation*}
    g(y) = \frac{y}{\sqrt{y^2+4}} F_1(y^2),
\end{equation*}
where $F_1$ is a rational function with only rational coefficients
\begin{equation*}
    F_1(y^2) = \sum_{k=0}^K c_k(\ftil_0) P_k\left(\frac{y^2}{y^2+4}\right).
\end{equation*}
Similarly we have that 
\begin{equation*}
    g'(y) = \frac{1}{\sqrt{y^2+4}} F_2(y^2),\qquad g''(y) = \frac{y}{\sqrt{y^2+4}} F_3(y^2),
\end{equation*}
where $F_2$ and $F_3$ are also rational functions with rational coefficients
\begin{align*}
    F_2(y^2) &= 2y^2 F_1'(y^2) + \frac{4}{y^2+4}F_1(y^2) \\
    F_3(y^2) &= 2 F_2'(y^2) - \frac{1}{y^2+4} F_2(y^2).
\end{align*}
Inserting these representations in \eqref{eqn:remainder_g} shows that for $d\in\{3, 5\}$ the remainder term $p_2\Rc(\ftil_0)$ is a rational function in $y^2$ with only rational coefficients. However due to the more complicated form of the weight $p_2$ for $d\in\{4, 6\}$ the weighted term is only a rational function in $y$. Since the evaluation method requires a finite domain, we apply a variable transformation $y\in[0,\infty)$ to  $z\in[0,1]$. For $d\in\{3, 5\}$, we use
\begin{equation}\label{eqn:variable_transform_01}
    y^2 = \frac{4z}{1-z},
\end{equation}
and for $d\in\{4, 6\}$,
\begin{equation}\label{eqn:variable_transform_02}
    y = \frac{4z}{1-z}.
\end{equation}
The coefficients for the numerator and denominator of the transformed rational function are given in the CSV-files \texttt{coefs\_resid\_n} and \texttt{coefs\_resid\_d}. After applying the evaluation method we have
\begin{equation}\label{eqn:bound_remainder}
    \| p_2\Rc(\ftil_0) \|_\infty \leq c_\Rc = 
    \left\{\begin{array}{@{}l l}
        \frac{93}{440690713}, & d=3, \\
        \frac{96}{7276398917}, & d=4, \\
        \frac{95}{377877930362}, & d=5, \\
        \frac{62}{17322232151807}, & d=6.
    \end{array}\right.
\end{equation}


\subsubsection{Bounds for the nonlinear term}\label{sec:bounds_nonlinear_term}
First we rewrite the nonlinear term \eqref{eqn:nonlinear_op} 
\begin{align*}
    \Nc(\delta) &= \frac{d-1}{y^2}\big(\cos(2\ftil_0)\delta - \cos(2\ftil_0+\delta)\sin(\delta)\big) \\
    &= \frac{d-1}{y^2}\big(\cos(2\ftil_0)(\delta-\sin(\delta)) + 4\cos(\delta/2)\sin^2(\delta/2)\sin(2\ftil_0+\delta/2)\big).
\end{align*}
After applying the bounds 
\begin{equation}\label{eqn:sin_bounds}
    |\delta-\sin(\delta)| \leq \frac{1}{6}|\delta|^3, \quad 
    |\sin^2(\delta/2)| \leq \frac{1}{4}|\delta|^2, \quad
    |\sin(2\ftil_0+\delta/2)| \leq |\sin(2\ftil_0)| + \frac{1}{2}|\delta|
\end{equation}
we get
\begin{equation*} 
    |\Nc(\delta)| \leq \frac{d-1}{y^2} \left( \frac{1}{6}|\delta|^3 
        + |\delta|^2 \left( |\sin(2\ftil_0)| + \frac{1}{2}|\delta| \right) \right)
    = \frac{d-1}{y^2} |\delta|^2 \left( |\sin(2\ftil_0)| + \frac{2}{3}|\delta| \right). 
\end{equation*}
Next we insert the weights for the norm of $\delta$ and the bound for the nonlinear term
\begin{equation*}
    |p_2\Nc(\delta)| \leq (p_1\delta)^2\left(\left| \frac{(d-1)p_2}{p_1^2y^2}\sin(2\ftil_0)\right| 
        + \frac{2(d-1)p_2}{3p_1^3y^2}|p_1\delta|\right).
\end{equation*}
For  $\delta\in \mc Y_{\varepsilon}$, it follows that
\begin{equation}
    |p_2\Nc(\delta)| \leq \left(\left| \frac{(d-1)p_2}{p_1^2y^2}\sin(2\ftil_0)\right| 
        + \frac{2(d-1)p_2}{3p_1^3y^2}\varepsilon\right) \| \delta \|_{\mc Y}^2.
\end{equation}
Similar to the analysis for the remainder term, the sine term is a rational function in $y^2$ for $d\in\{3, 5\}$ and a rational function in $y$ for $d\in\{4, 6\}$. The same also holds for the weight term. Applying the evaluation method with the variable transformation (\ref{eqn:variable_transform_01}) for odd $d$ and (\ref{eqn:variable_transform_02}) for even $d$ to the sine term and the weight term separately, gives the bounds \footnote{\ The corresponding CSV-files are \texttt{coefs\_sin\_n(d)} and \texttt{coefs\_weights\_n(d)}.} 
\begin{equation}\label{eqn:bound_nonlinear}
    \|p_2\Nc(\delta)\|_\infty \leq c_\Nc \|\delta\|_{\mc Y}^2, \qquad c_\Nc = 
 \left\{\begin{array}{@{}l l}
        \frac{295}{86}, & d=3, \\
        \frac{310}{89}, & d=4, \\
        \frac{361}{83}, & d=5, \\
        \frac{622}{97}, & d=6.
    \end{array}\right.
\end{equation}
Next let $\delta, \gamma\in \mc Y_{\varepsilon}$. Rewriting the difference gives
\begin{align*}
    \Nc(\delta)-\Nc(\gamma) &= \frac{d-1}{y^2}\big(\cos(2\ftil_0)(\delta-\gamma) 
        - \cos(2\ftil_0+\delta)\sin(\delta) + \cos(2\ftil_0+\gamma)\sin(\gamma)\big)\\
    &= \frac{d-1}{y^2}\big(\cos(2\ftil_0)((\delta-\gamma)-\sin(\delta-\gamma)) \\
    &\quad    + 2\sin((\delta+\gamma)/2)\sin(\delta-\gamma)\sin(2\ftil_0+(\delta+\gamma)/2)\big).
\end{align*}
Similar to before, applying the bounds \eqref{eqn:sin_bounds} yields
\begin{align*}
    |\Nc(\delta)-\Nc(\gamma)| &\leq \frac{d-1}{y^2}\left( \frac{1}{6}|\delta-\gamma|^3 
        + |\delta+\gamma||\delta-\gamma| \left(|\sin(2\ftil_0)| + \frac{1}{2}|\delta+\gamma|\right) \right) \\
    &\leq \frac{d-1}{y^2}|\delta-\gamma|(|\delta| + |\gamma|) \left( |\sin(2\ftil_0)| + \frac{2}{3}(|\delta| + |\gamma|) \right).
\end{align*}
With the appropriate weights included, we have
\begin{equation*}
    |p_2(\Nc(\delta)-\Nc(\gamma))| \leq |p_1(\delta-\gamma)|(|p_1\delta|+|p_1\gamma|)
        \left(\left|\frac{(d-1)p_2}{p_1^2y^2}\sin(2\ftil_0)\right| 
        + \frac{2(d-1)p_2}{3p_1^3y^2}(|p_1\delta| + |p_1\gamma|)\right),
\end{equation*}
which then gives
\begin{equation}
    \|p_2(\Nc(\delta)-\Nc(\gamma))\|_\infty \leq \left(\left|\frac{(d-1)p_2}{p_1^2y^2}\sin(2\ftil_0)\right| 
        + \frac{4(d-1)p_2}{3p_1^3y^2}\varepsilon\right)\|\delta-\gamma\|_{\mc Y}(\|\delta\|_{\mc Y}+\|\gamma\|_{\mc Y}).
\end{equation}
We can reuse the previous result to finally get the estimate
\begin{equation}\label{eqn:bound_nonlinear_lipschitz}
     \|p_2(\Nc(\delta)-\Nc(\gamma))\|_\infty \leq c_\Nc\|\delta-\gamma\|_{\mc Y}(\|\delta\|_{\mc Y}+\|\gamma\|_{\mc Y}).
\end{equation}
Note that (\ref{eqn:bound_nonlinear}) already accounts for the slightly larger constant needed in (\ref{eqn:bound_nonlinear_lipschitz}).

\subsubsection{Approximate Linear Operator}\label{Subsubsec:Approx_L}

We first isolate the endpoint behavior by considering the free part of $\Lc$, namely
\begin{equation}\label{eqn:linear_term_free}
    \Lc_0(\delta) : = -\delta'' - \left(\frac{d-1}{y}-\frac{y}{2}\right)\delta' + \frac{d-1}{y^2}\delta.
\end{equation}
A fundamental system $\{v_0, v_1 \}$ for $\Lc_0$ is given by
\begin{align}
    v_0(y) &= \left\{\begin{array}{@{}l l}
        \frac{3}{y^2}e^{y^2/4}\left((y^2+2)D_+\left(\frac{y}{2}\right)-y\right), & d=3,\\
        \frac{4e^{y^2/8}}{3y}\left(y^2 I_0\left(\frac{y^2}{8}\right)-(y^2+4)I_1\left(\frac{y^2}{8}\right)\right), & d=4,\\
        \frac{15}{4y^4}e^{y^2/4}\left((y^4+4y^2+12)D_+\left(\frac{y}{2}\right)-(y^3+6y)\right), & d=5, \\
        \frac{8e^{y^2/8}}{5y^3}\left((y^4 + 2y^2)I_0\left(\frac{y^2}{8}\right) - (y^4 + 6y^2 + 32) I_1\left(\frac{y^2}{8}\right)\right), & d=6
    \end{array}\right. \label{eqn:homogeneous_solution_v0}\\
    v_1(y) &= \left\{\begin{array}{@{}l l}
        1 + \frac{2}{y^2}, & d=3,\\
        \frac{8}{y^{3}} U\left(-\frac{3}{2}, -1, \frac{y^{2}}{4}\right), & d=4,\\
        1 + \frac{4(y^2+3)}{y^4}, & d=5, \\
        \frac{y}{2} U\left(\frac{1}{2}, 4, \frac{y^{2}}{4}\right), & d=6,
    \end{array}\right. \label{eqn:homogeneous_solution_v1}
\end{align}
where $D_+$ is the Dawson integral, $I_\nu$ is the modified Bessel function of the first kind and $U$ denotes Tricomi’s confluent hypergeometric function. The asymptotic behavior is given by
\begin{align*}
    v_0(y) &= y+\mathcal O(y^3), &&\text{for } y\to 0,\\
    v_0(y) &\sim y^{-d}e^{y^2/4}, &&\text{for } y\to\infty,
\end{align*}
and
\begin{align*}
    v_1(y) &\sim y^{1-d}, &&\text{for } y\to 0,\\
    v_1(y) &= 1+\mathcal O(y^{-2}), &&\text{for } y\to\infty.
\end{align*}

Since these special functions are difficult do deal with in the derivation of quantitative bounds, we note that is suffices to isolate their endpoint behavior. Hence, we provide rational approximations of $v_0$ and $v_1$, which simultaneously match the dominating behavior both at $0$ and $\infty$.

\begin{definition}\label{approximate_homogeneous_solutions}
For $d\in\{3, 4, 5, 6\}$ we define 
\begin{align*}
    \vtil_0(y) &:= y e^{y^2/4} \Bigg(\sum_{k=0}^{K_N} c_{k, N}(\vtil_0)\ y^{k}\Bigg)\Bigg/\Bigg(\sum_{k=0}^{K_D} c_{k, D}(\vtil_0)\ y^{k}\Bigg) \\
    \vtil_1(y) &:= y^{1-d} \Bigg(\sum_{k=0}^{K_N} c_{k, N}(\vtil_1)\ y^{k}\Bigg)\Bigg/\Bigg(\sum_{k=0}^{K_D} c_{k, D}(\vtil_1)\ y^{k}\Bigg),
\end{align*}
with the coefficients given in Appendix  \ref{Appendix:Coeff}, with $K_N, K_D$ depending on $d$. The corresponding CVS-files are \texttt{coefs\_v0\_n(d)} and \texttt{coefs\_v1\_n(d)}.
\end{definition}

\begin{remark}
For $d\in\{3, 5\}$ the approximation $\vtil_0$ is obtained by replacing the Dawson integral with a 4-term truncation of the 2-point expansion provided in \cite{McCabe_1974}, while $\vtil_1$ coincides with $v_1$. For $d\in\{4, 6\}$ the coefficients of $\vtil_0$ were computed by applying the method described in \cite{Andrade_2003} to $\frac{v_0}{ye^{y^2/4}}$. For $\vtil_1$, we expanded $v_1$ and the ansatz both at $0$ and $\infty$ and solved for the coefficients. 
\end{remark}

For later use, we state the following technical lemma. 

\begin{lemma}\label{lem:v_sign}
For all $y\in(0, \infty)$, the functions $\vtil_0$ and $\vtil_1$, as well as $\vtil_0'$ are positive, while $\vtil_1'$ is negative.    
\end{lemma}
\begin{proof}
Since all coefficients are positive or zero, it is clear that $\vtil_1$ is positive for $y\in(0, \infty)$. After a direct calculation, we have that $\vtil_1'$ has a positive denominator with all coefficients in the numerator non-positive. For $\vtil_0$ and $\vtil_0'$, we apply the evaluation method with variable transformations (\ref{eqn:variable_transform_01}) for $d\in\{3, 5\}$ and (\ref{eqn:variable_transform_02}) for $d\in\{4, 6\}$ to show that \footnote{\ The corresponding CSV-files are \texttt{coefs\_v0pos\_n(d)} and \texttt{coefs\_v0dpos\_n(d)}.}
\begin{align*}
    \vtil_0(y) e^{-y^2/4}\frac{1+y^{d+1}}{y} &\geq 
    \left\{\begin{array}{@{}l l}
        \frac{71}{78}, & d=3,\\
        \frac{68}{75}, & d=4,\\
        \frac{37}{41}, & d=5,\\
        \frac{22}{25}, & d=6
    \end{array}\right.\\ 
    \vtil_0'(y) e^{-y^2/4}(1+y^{d-1}) &\geq 
    \left\{\begin{array}{@{}l l}
        \frac{55}{61}, & d=3,\\
        \frac{89}{97}, & d=4,\\
        \frac{68}{71}, & d=5,\\
        \frac{10}{11}, & d=6.
    \end{array}\right. 
\end{align*}
for all $y \in (0,\infty)$. This implies the claim. 
\end{proof}

Next, we derive a closed-form approximate fundamental system for $\Lc$.

\begin{definition}\label{def:approximate_fundamental_system}
For $d\in\{3, 4, 5, 6\}$ we define 
\begin{align*}
    \wtil_0(y) &:= \vtil_0(y) \sum_{k=0}^{K} c_k(\wtil_0) T_{2k}\left(\frac{y}{\sqrt{y^2+4}}\right), \\
    \wtil_1(y) &:= \vtil_1(y) \sum_{k=0}^{K} c_k(\wtil_1) T_k\left(\frac{y-2}{y+2}\right),
\end{align*}
with the coefficients given in Tables~\ref{tab:coef_w0_d3}-\ref{tab:coef_w0_d6} and Tables~\ref{tab:coef_w1_d3}-\ref{tab:coef_w1_d6}, and $K$ determined in each case by the number of listed coefficients. The corresponding CVS-files are \texttt{coefs\_w0} and \texttt{coefs\_w1}.
\end{definition}

\begin{remark}
Extending $\Lc$ to negative arguments shows that the equation admits an odd solution. Since the antisymmetric behavior is already contained in $\vtil_0$, the Chebyshev expansion must be even to ensure the correct endpoint behavior. For the second solution, we use a general Chebyshev expansion and enforce the correct endpoint behavior through the coefficients. 
The computation of the coefficients is very similar to the computation of the coefficients of $\ftil_0$. However, after having determined $(c_k(\wtil_1))_{k=0}^{K-2}$, we exactly solve 
\begin{equation*}
    \sum_{k=0}^{K} c_k(\wtil_1) T_k'(-1) = 0, \qquad
    \sum_{k=0}^{K} c_k(\wtil_1) T_k'(1) = 0,
\end{equation*}
for $c_{K-1}(\wtil_1)$ and $c_{K}(\wtil_1)$, ensuring the correct endpoint behavior. These two coefficients are uniquely determined by these two conditions and are thus not explicitly given in the tables, but can be found in the CSV-files.
\end{remark}

\begin{lemma}\label{lem:approx_linear_operator}
Let $\mc L$ be given as in \eqref{eqn:linear_op} and $\wtil_0, \wtil_1 \in C^2((0,\infty))$ as defined above. Then $\{\wtil_0, \wtil_1 \}$ is a fundamental system for the linear operator $\Lctil = \Lc + P\partial_y + Q$ with
\begin{equation}\label{eqn:PQ_operator}
\begin{aligned}
    P &= \frac{-1}{W(\wtil_0,\wtil_1)}(\wtil_0\Lc(\wtil_1) - \Lc(\wtil_0)\wtil_1), \\
    Q &= \frac{1}{W(\wtil_0,\wtil_1)}(\wtil_0'\Lc(\wtil_1) - \Lc(\wtil_0)\wtil_1').
\end{aligned}
\end{equation}
\end{lemma}
\begin{proof}
By construction, $\Lctil \wtil_0 = 0$ and $\Lctil \wtil_1 = 0$.
\end{proof}

Next, we show that $\Lctil$ is a reasonable approximation to $\Lc$. More precisely, we consider the difference
\begin{equation*}
    p_2(\Lctil-\Lc)(\delta) = p_2(P\delta' + Q\delta),
\end{equation*}
where $P$ and $Q$ are defined in \eqref{eqn:PQ_operator}. After adding the appropriate weights, we get the bound
\begin{align}
    \|p_2(\Lctil-\Lc)(\delta)\|_\infty &\leq 
        \left(\left\|\frac{p_2}{p_3}P\right\|_\infty \|p_3\delta'\|_\infty 
        + \left\|\frac{p_2}{p_1}Q\right\|_\infty \|p_1\delta\|_\infty\right) \label{eqn:bound_approx_linear_parts} \\
    &\leq \left(\left\|\frac{p_2}{p_3}P\right\|_\infty + \left\|\frac{p_2}{p_1}Q\right\|_\infty\right) 
        \|\delta\|_{\mc Y} .\nonumber
\end{align}
A direct calculation shows that both fractions are rational functions in $y$. Applying the evaluation method with variable transformation \eqref{eqn:variable_transform_02}, we get \footnote{\ The corresponding CVS-files are \texttt{coefs\_P\_n(d)} and \texttt{coefs\_Q\_n(d)}.}
\begin{equation}\label{eqn:bound_approx_linear}
    \|p_2(\Lctil-\Lc)(\delta)\|_\infty \leq (c_P + c_Q) \|\delta\|_{\mc Y} = c_\Lc \|\delta\|_{\mc Y},
\end{equation}
where
\begin{align}
    \left\|\frac{p_2}{p_3}P\right\|_\infty \leq c_P &= 
    \left\{\begin{array}{@{}l l}
        \frac{73}{1227666}, & d=3, \\
        \frac{4}{68085}, & d=4, \\
        \frac{35}{17955088}, & d=5, \\
        \frac{17}{401628416}, & d=6
    \end{array}\right. \label{eqn:c_linear_P} \\
    \left\|\frac{p_2}{p_1}Q\right\|_\infty \leq c_Q &= 
    \left\{\begin{array}{@{}l l}
        \frac{25}{484609}, & d=3, \\
        \frac{29}{422128}, & d=4, \\ 
        \frac{87}{9583418}, & d=5, \\ 
        \frac{88}{490373645}, & d=6.
    \end{array}\right. \label{eqn:c_linear_Q}
\end{align}

\subsubsection{Bounds for the inverse operator}\label{sec:inverse_op}

The goal is then to derive a bound of the form
\begin{equation}\label{eqn:bound_inverse_shape}
    \|\Lctil^{-1}\alpha\|_{\mc Y} = \|p_1\Lctil^{-1}\alpha\|_\infty + \|p_3(\Lctil^{-1}\alpha)'\|_\infty \leq c_{\Lctil} \|p_2\alpha\|_\infty,
\end{equation}
where $p_2$ is the weight introduced in Definition~\ref{def:weights}. With $\Gtil$ denoting the Green's function of $\Lctil$, 
\begin{align}
    \Gtil(x,y) &= \frac{-1}{W(\wtil_0,\wtil_1)(x)} \left\{ 
    \begin{array}{@{}l l}
        \wtil_0(y)\wtil_1(x), & y \leq x,\\
        \wtil_0(x)\wtil_1(y), & x < y,
    \end{array}\right. \label{eqn:greens_approx}
\end{align}
a formal inverse of $\Lctil$ is given by 
\begin{align*}
  \Lctil^{-1}(\alpha)(y) &= \int_0^\infty \Gtil(x, y) \alpha(x)\, dx.
\end{align*}

For further use, we introduce the approximate free part of the linear operator $\Lc$
\begin{equation}
    \Lctil_0 := \Lc_0 + P_0\partial_y + Q_0,
\end{equation}
for 
\begin{equation}
\begin{aligned}
    P_0 &= \frac{-1}{W(\vtil_0, \vtil_1)}(\vtil_0\Lc_0(\vtil_1) - \Lc_0(\vtil_0)\vtil_1), \\
    Q_0 &= \frac{1}{W(\vtil_0, \vtil_1)}(\vtil_0'\Lc_0(\vtil_1) - \Lc_0(\vtil_0)\vtil_1').
\end{aligned}
\end{equation}
By construction, a fundamental system of $\Lctil_0$ is given by $\{\vtil_0, \vtil_1\}$. The associated Green's function is denoted as
\begin{align}
    \Gtil_0(x,y) &= \frac{-1}{W(\vtil_0,\vtil_1)(x)} \left\{ 
    \begin{array}{@{}l l}
        \vtil_0(y)\vtil_1(x), & y \leq x,\\
        \vtil_0(x)\vtil_1(y), & x < y.
    \end{array}\right. \label{eqn:greens_approx_homo}
\end{align}

We start by considering the $p_1$=term in \eqref{eqn:bound_inverse_shape}. After including the appropriate weights and factoring out the endpoint behavior, we have the bound
\begin{equation}\label{eqn:bound_inverse_first}
    \|p_1\Lctil^{-1}\alpha\|_\infty \leq \underset{y\geq 0}{\sup} \left(\int_0^\infty \left|\Gtil_0(x, y)\frac{p_1(y) }{p_2(x)}\right|\, dx \right)\left\|\frac{\Gtil}{\Gtil_0}\right\|_{L^\infty([0,\infty)^2)}\|p_2\alpha\|_\infty.
\end{equation}
Similarly for the $p_3$-term, we have 
\begin{equation} \label{eqn:bound_inverse_second}
    \|p_3(\Lctil^{-1}\alpha)'\|_\infty \leq \underset{y\geq 0}{\sup} \left(\int_0^\infty \left|\partial_y\Gtil_0(x, y)\frac{p_3(y) }{p_2(x)}\right|\, dx \right)\left\|\frac{\partial_y\Gtil}{\partial_y\Gtil_0}\right\|_{L^\infty([0,\infty)^2)}\|p_2\alpha\|_\infty.
\end{equation}
Next, we derive upper bounds on the supremum of the two integrals. For this, we note that the weight functions satisfy 
\begin{equation}\label{eqn:defining_properties_weigths}
    \Lctil_0\left(\frac{1}{p_1}\right)=\frac{1}{p_2}, \qquad \left(\frac{1}{p_1}\right)' = \frac{1}{p_3},
\end{equation}
which are actually their defining properties. Using the formal inverse of $\Lctil_0$ given through $\Gtil_0$, this implies 
\begin{align}
    \int_0^\infty \Gtil_0(x,y)\frac{p_1(y)}{p_2(x)}\, dx &= 1 \label{eqn:weight_int_ident_1}, \\
    \int_0^\infty \partial_y \Gtil_0(x,y)\frac{p_3(y)}{p_2(x)}\, dx &= 1. \label{eqn:weight_int_ident_2}
\end{align}
However, we cannot directly apply these identities to the inequalities \eqref{eqn:bound_inverse_first} and \eqref{eqn:bound_inverse_second} without first considering the signs of the integrands. The weights $p_1$, $p_2$ and $p_3$ are positive on $(0, \infty)$ so it remains to determine the signs of $\Gtil_0(x,y)$ and $\partial_y \Gtil_0(x,y)$.
It follows from Lemma~\ref{lem:v_sign}, that $\Gtil_0$ is non-negative, while $\partial_y \Gtil_0(x,y) \geq 0$ if  $y\leq x$ and $\partial_y \Gtil_0(x,y)\leq0$ if $x<y$. The absolute value in the integral of (\ref{eqn:bound_inverse_second}) can thus be omitted by splitting the integration domain
\begin{align*}
    \int_0^\infty \left|\partial_y \Gtil_0(x,y)\frac{p_3(y)}{p_2(x)}\right|\, dx &= 
       \vtil_1'(y) p_3(y) \int_0^y \frac{1}{W(\vtil_0, \vtil_1)(x)} \frac{\vtil_0(x)}{p_2(x)}\, dx \\
        & - \vtil_0'(y) p_3(y) \int_y^\infty \frac{1}{W(\vtil_0, \vtil_1)(x)} \frac{\vtil_1(x)}{p_2(x)}\, dx.
\end{align*}
To abbreviate the notation we define
\begin{align}
    I_0(y) &:= \vtil_1'(y) p_3(y) \int_0^y \frac{1}{W(\vtil_0, \vtil_1)(x)} \frac{\vtil_0(x)}{p_2(x)}\, dx
        \label{eqn:definition_I_0}, \\
    I_1(y) &:= -\vtil_0'(y) p_3(y) \int_y^\infty \frac{1}{W(\vtil_0, \vtil_1)(x)} \frac{\vtil_1(x)}{p_2(x)}\, dx
        \label{eqn:definition_I_1},
\end{align}
and note that (\ref{eqn:weight_int_ident_2}) gives
\begin{equation}\label{eqn:integral_one}
    I_1(y)-I_0(y) = 1.
\end{equation}
Direct integration is infeasible, instead we factor out the exponential behavior
\[     I_1(y) = -\vtil_0'(y) p_3(y) e^{-y^2/4} e^{y^2/4}  \int_y^\infty e^{-x^2/4}  \frac{e^{x^2/4}}{W(\vtil_0, \vtil_1)(x)} \frac{\vtil_1(x)}{p_2(x)}\, dx \]
and first show that 
\begin{equation}\label{eqn:upper_bound_integrant}
    -\frac{e^{x^2/4}}{W(\vtil_0, \vtil_1)(x)} \frac{\vtil_1(x)}{p_2(x)} \leq 
    \left\{\begin{array}{@{}l l}
        \frac{81}{101}, & d=3 \\
        \frac{195}{704}x + \frac{93}{1517}, & d=4,\\
        \frac{1}{10}x^2 + \frac{80}{63}, & d=5, \\
        \frac{91}{2816}x^3 + \frac{65}{2816} x + \frac{53}{97}, & d=6.
    \end{array}\right.  
\end{equation}
More precisely, to obtain these bounds, we apply the evaluation method to compute a negative upper bound on the maximum of the difference, using the variable transformation
\begin{equation}\label{eqn:variable_transform_03}
    x = \frac{4z}{1-z^2},
\end{equation}
which maps $x\in[0, \infty)$ to $z\in[0, 1]$ and converts the $\sqrt{x^2+4}$ term introduced by $p_2$ to a rational function.\footnote{\ The corresponding CSV-files are \texttt{coefs\_ingrtneg\_n(d)}.} Repeated integration by parts shows that
\begin{equation*}
    \int_y^\infty e^{-x^2/4} (c_3 x^3 + c_2 x^2 + c_1 x + c_0)\, dx 
    = 2 e^{-y^2/4} (c_3 (y^2+4) + c_2 y + c_1) + \int_y^\infty e^{-x^2/4}(2c_2 + c_0)\, dx.
\end{equation*}
After inserting the upper bound \eqref{eqn:upper_bound_integrant}, we have
\begin{equation}\label{eqn:upper_bound_integral}
    e^{y^2/4} \int_y^\infty \frac{-1}{W(\vtil_0, \vtil_1)(x)} \frac{\vtil_1(x)}{p_2(x)}\, dx \leq 
    \left\{\begin{array}{@{}l l}
        \frac{81}{101}\frac{16y + 18}{8y^2 + 9y + 10}, & d=3,\\
        \frac{195}{352} + \frac{93}{1517}\frac{16y + 18}{8y^2 + 9y + 10}, & d=4,\\
        \frac{1}{5}y + \frac{463}{315}\frac{16y + 18}{8y^2 + 9y + 10}, & d=5,\\
        \frac{91}{1408}y^2 + \frac{39}{128} + \frac{53}{97}\frac{16y + 18}{8y^2 + 9y + 10}, & d=6,
    \end{array}\right.
\end{equation}
where we used a root free variation of the upper bound \cite[7.1.13]{Abramowitz_1964} 
\begin{equation}
    e^{y^2/4} \int_y^\infty e^{-x^2/4}\, dx \leq \frac{4}{y + \sqrt{y^2+\tfrac{16}{\pi}}} \leq \frac{16y + 18}{8y^2 + 9y + 10},\qquad y\geq 0.
\end{equation}
Applying the evaluation method with the transformation \eqref{eqn:variable_transform_03}, with $y$ in place of $x$, to the product of $\vtil_0' p_3 e^{-y^2/4}$ and the upper bound in \eqref{eqn:upper_bound_integral} yields an upper bound on $I_1$.\footnote{\ The corresponding CSV-files are \texttt{coefs\_I1\_n(d)}.} Combining this bound with \eqref{eqn:integral_one}, we conclude that
\begin{equation}
    \int_0^\infty \left|\partial_y \Gtil_0(x,y)\frac{p_3(y)}{p_2(x)}\right|\, dx = I_0(y) + I_1(y) = 2I_1(y) - 1 \leq 
    \left\{\begin{array}{@{}l l}
        \frac{82}{17}, & d=3,\\
        \frac{59}{13}, & d=4,\\
        \frac{271}{28}, & d=5,\\
        \frac{67}{9}, & d=6.
    \end{array}\right. 
\end{equation}
Lastly, we derive rigorous bounds for the Green function ratios appearing in \eqref{eqn:bound_inverse_first} and \eqref{eqn:bound_inverse_second}. Since
\begin{align*}
    \left\|\frac{\Gtil}{\Gtil_0}\right\|_{L^\infty([0,\infty)^2)} &\leq \left\|\frac{W(\vtil_0, \vtil_1)}{W(\wtil_0, \wtil_1)}\right\|_\infty 
        \left\|\frac{\wtil_0}{\vtil_0}\right\|_\infty \left\|\frac{\wtil_1}{\vtil_1}\right\|_\infty, \\
    \left\|\frac{\partial_y\Gtil}{\partial_y\Gtil_0}\right\|_{L^\infty([0,\infty)^2)} &\leq \left\|\frac{W(\vtil_0, \vtil_1)}{W(\wtil_0, \wtil_1)}\right\|_\infty 
        \max\left(\left\|\frac{\wtil_0'}{\vtil_0'}\right\|_\infty \left\|\frac{\wtil_1}{\vtil_1}\right\|_\infty,
        \left\|\frac{\wtil_0}{\vtil_0}\right\|_\infty \left\|\frac{\wtil_1'}{\vtil_1'}\right\|_\infty\right),
\end{align*}
it remains to apply the evaluation method to each of the fractions. All terms are rational functions in $y$. Furthermore, $\frac{\wtil_0}{\vtil_0}$ for $d\in\{3,4,5,6\}$, and $\frac{\wtil_0'}{\vtil_0'}$ for $d\in\{3,5\}$ are rational functions in $y^2$. Applying the evaluation method with the variable transformation (\ref{eqn:variable_transform_01}) for the terms in $y^2$ and (\ref{eqn:variable_transform_02}) for all other terms and combining the results, we have the final estimate \footnote{\ The corresponding CSV-files are \texttt{coefs\_wron\_n(d)}, \texttt{coefs\_w0\_n(d)}, \texttt{coefs\_w1\_n(d)}, \texttt{coefs\_w0d\_n(d)} and \texttt{coefs\_w1d\_n(d)}.}
\begin{equation}\label{eqn:bound_inverse}
    \|\Lctil^{-1} \alpha\| \leq \left(\left\|\frac{\Gtil}{\Gtil_0}\right\|_{L^\infty([0,\infty)^2)} + (2I_1-1)\left\|\frac{\partial_y\Gtil}{\partial_y\Gtil_0}\right\|_{L^\infty([0,\infty)^2)}\right) \|p_2\alpha\|_\infty 
        \leq c_{\Lctil} \|p_2\alpha\|_\infty, 
\end{equation}
for all $\alpha \in \mc Z$, where 
\begin{equation}\label{eqn:c_approx_linear}
    c_{\Lctil} = \left\{\begin{array}{@{}l l}
        \frac{6467}{31}, & d=3,\\
        \frac{24119}{52}, & d=4,\\
        2941, & d=5,\\
        \frac{1042526}{65}, & d=6.
    \end{array}\right. 
\end{equation}

\subsubsection{Proof of Theorem~\ref{thm:closeness}}\label{sec:proof_closeness}
We show that $\Kctil$ defined in Eq.~\eqref{eqn:contraction_op_ktil} is a contraction on $\mc Y_{\varepsilon}$ for $\varepsilon$ as stated in Theorem \ref{thm:closeness}. Given $\delta \in \mc Y_{\varepsilon}$, we apply the bounds derived in the previous sections to get
\begin{align*}
    \|\Kctil(\delta)\|_{\mc Y} &= \|\Lctil^{-1}(\Rc(\ftil_0) + \Nc(\delta) + (\Lctil-\Lc)(\delta))\|_{\mc Y} \\
    &\leq c_{\Lctil}(\|p_2 \Rc(\ftil_0)\|_\infty + \|p_2 \Nc(\delta)\|_\infty 
        + \|p_2 (\Lctil-\Lc)(\delta)\|_\infty) \\
    &\leq c_{\Lctil}(c_\Rc + c_\Nc\|\delta\|_{\mc Y}^2 + c_\Lc\|\delta\|_{\mc Y}) \\
    &\leq c_{\Lctil}(c_\Rc\varepsilon^{-1} + c_\Nc\varepsilon + c_\Lc)\varepsilon.
\end{align*}
Then $\Kctil(\delta) \in \mc Y_{\varepsilon}$ follows from
\begin{equation}\label{eqn:c_self_mapping}
    c_{\Lctil}(c_\Rc\varepsilon^{-1} + c_\Nc\varepsilon + c_\Lc) \leq 
    \left\{\begin{array}{@{}l l}
        \frac{95}{97}, & d=3, \\ 
        \frac{77}{81}, & d=4, \\
        \frac{51}{52}, & d=5, \\
        \frac{47}{51}, & d=6.
    \end{array}\right.
\end{equation}
Similarly, given $\delta, \gamma \in \mc Y_{\varepsilon}$ we have \begin{align*}
    \|\Kctil(\delta) - \Kctil(\gamma)\|_{\mc Y} &= 
        \|\Lctil^{-1}(\Nc(\delta) - \Nc(\gamma) + (\Lctil-\Lc)(\delta-\gamma))\|_{\mc Y} \\
    &\leq c_{\Lctil}(\|p_2(\Nc(\delta)-\Nc(\gamma))\|_\infty+\|p_2(\Lctil-\Lc)(\delta-\gamma)\|_\infty) \\
    &\leq c_{\Lctil}(c_\Nc(\|\delta\|_{\mc Y} + \|\gamma\|_{\mc Y}) + c_\Lc) \|\delta - \gamma\|_{\mc Y} \\
    &\leq c_{\Lctil}(2c_\Nc\varepsilon + c_\Lc) \|\delta - \gamma\|_{\mc Y},
\end{align*}
where $\Kctil$ being a contraction then follows with
\begin{equation}\label{eqn:c_contraction}
    c_{\Lctil}(2c_\Nc\varepsilon + c_\Lc) \leq 
    \left\{\begin{array}{@{}l l}
        \frac{9}{98}, & d=3, \\ 
        \frac{4}{49}, & d=4, \\ 
        \frac{1}{19}, & d=5, \\ 
        \frac{9}{545}, & d=6.
    \end{array}\right.
\end{equation}

By the contraction mapping principle there exists a $\delta_0\in  \mc Y_{\varepsilon}$ solving \eqref{eqn:delta_ode}. Since $\ftil_0\in C^\infty([0,\infty))$ and the ODE (\ref{eqn:delta_ode}) has smooth coefficients on $(0,\infty)$, we have $\delta_0 \in C^\infty((0, \infty))$. Furthermore, the definition of $\mc Y$ implies $\delta_0(y) = \Oh(y)$ as $y\to0$, excluding the singular behavior at the origin. Thus $\delta_0\in C^\infty([0,\infty))$ and $f_0 = \ftil_0 + \delta_0$ is a smooth solution of Eq.~\eqref{Eqf} satisfying $\|f_0-\ftil_0\|_{\mc Y} \leq\varepsilon$, completing the proof of Theorem~\ref{thm:closeness}.


\subsection{Spectral stability of $f_0$}

In this section, we prove that the solution $f_0$ constructed in Theorem~\ref{thm:closeness} is spectrally stable. This yields Theorem \ref{Th:Existence_f0}. The proof follows the general approach of \cite{BieDon18} and utilizes the closed form approximations from the previous section. 

First, with the result of Theorem~\ref{thm:closeness}, we can defined the self-adjoint linearized operator 
\[ L_{f_0}: \mc D(L_{f_0}) \subset H \to H  \]
where $n = d+2$. By definition, $H$ is a space of radially symmetric functions. Hence, by introducing the weighted space $\mc H:= L^2_{\rho}([0,\infty))$ with $\rho(y) = y^{d-1} e^{-y^2/4}$ and a unitary map 

\[  \mc U: \mc H \to  H, \quad \tilde u \mapsto  | \mathbb S^{n-1} |^{-\frac{1}{2}} |\cdot|^{-1} \tilde u(|\cdot|) \]

we can define the unitarily equivalent  operator $\Acn := - \mc U^{-1} L_{f_0} \mc U$ with domain 
$\mc D( \Acn) := \mc U^{-1} \mc D(L_{f_0}) \subset \mc H$ and 

\begin{equation}\label{eqn:acn}
    \Acn w(y) = - w''(y)  - \left(\frac{d-1}{y} - \frac{y}{2}\right)  w'(y) 
        + \frac{d-1}{y^2} \cos(2f_0(y)) w(y).
\end{equation}
After reducing the operator to Schrödinger form with the substitution $v(y) = \sqrt{\rho(y)} w(y)$, Theorems 6.4 and 6.6 in \cite{Weidmann_1987} imply that $\Acn$ is in the limit-point case at both $0$ and $\infty$.

For $\lambda\in\R$, we consider the solution regular at the origin, satisfying the spectral ODE 
\begin{equation}\label{eqn:eigenvalue_ivp}
    \Acn W = \lambda W, \qquad W(0)=0, \qquad W'(0)=1.
\end{equation} 
By the limit-point case at $0$, the admissible behavior is unique up to normalization, which we fix by choosing $W'(0)=1$. We denote the solution of \eqref{eqn:eigenvalue_ivp} by $W_\lambda$.

Since the operator is limit-point at both endpoints, it admits a unique self-adjoint realization and the domain $\mc D(\Acn)$ must coincide with the maximal domain of definition. Thus, for a solution $W_\lambda$ of \eqref{eqn:eigenvalue_ivp} to be an eigenfunction, it suffices to check if $W_\lambda\in\mc H$.

Our goal is to prove Theorem~\ref{Th:Existence_f0} by applying Theorem 1.2 of \cite{GesSimTes1996}. This reduces counting the negative eigenvalues to determining the number of zeros of $W_\lambda$. In particular, we count the zeros of $W_\lambda$ for $\lambda=0$, $\lambda=-1$ and $\lambda<-1$.


\subsubsection{The case $\lambda=0$}\label{sec:6_1}

The goal is to determine the number of zeros $W_0$ has on $(0,\infty)$. For additional control, we define $q:=\frac{W_0}{y}$, which solves
\begin{equation}\label{eqn:q_ivp}
\begin{aligned}
    q''(y) &= \left(\frac{y}{2} - \frac{d+1}{y}\right)q'(y) 
        + \left(\frac{1}{2} - \frac{2(d-1)}{y^2} + \frac{2(d-1)}{y^2}\cos^2(f_0(y))\right) q(y),
\end{aligned}
\end{equation}
and note that $W_0$ and $q$ share the same zeros on $(0,\infty)$. Notice that for $y\geq y^\ast$, where
\begin{equation}
    y^\ast = \frac{d+3}{2}
\end{equation}
the following inequalities hold
\begin{equation*}
    \frac{y}{2} - \frac{d+1}{y} > 0, \qquad \frac{1}{2} - \frac{2(d-1)}{y^2} \geq 0.
\end{equation*}
If we assume $q(y^\ast)<0$ and $q'(y^\ast)<0$, then by \eqref{eqn:q_ivp}, it follows that $q''(y^\ast)<0$. Consequently for any $y\geq y^\ast$ we have
\begin{equation*}
    q(y)<0, \qquad q'(y)<0,
\end{equation*}
which implies that any zeros of $q$ on $(0,\infty)$ must lie in $(0, y^\ast)$. This reduces determining the number of zeros of $W_0$ on $(0,\infty)$ to finding rigorous bound on the behavior of $W_0$ and $q$ on the interval $[0, y^\ast]$.

For this, we compare the operators $\Acn$ in Eq.~\eqref{eqn:acn} and $\Lc$ in Eq.~\eqref{eqn:linear_op}, and notice that
\begin{equation*}
    \Acn = \Lc + (V_0-\Vtil_0),
\end{equation*}
where $V_0$ is given in (\ref{eqn:acn}) and
\begin{equation}
    \Vtil_0 = -\frac{2(d-1)}{y^2}\sin^2(\ftil_0).
\end{equation}
By Theorem~\ref{thm:closeness}, we expect the difference $\Vtil_0-V_0$ to be small, implying that an approximate solution of $\Lc w=0$ should also be a reasonable approximate solution of $\Acn W=0$. 

This is established in the following lemma. The proof closely follows that of Theorem~\ref{thm:closeness}, allowing us to reuse the previously established bounds.

\begin{lemma}\label{lem:wtil_approx_W0}
There exists a constant $c_0>0$ such that
\begin{equation}
  \|\chi p_1(W_0 - c_0\wtil_0)\|_{\infty} + \|\chi p_3(W'_0 - c_0\wtil'_0)\|_{\infty}   \leq c_0\widehat{\varepsilon},
\end{equation}
where $\wtil_0$ is defined in (\ref{def:approximate_fundamental_system}), $\chi$ denotes the indicator function on $[0, y^\ast]$ and
\begin{equation}\label{eqn:varepsilon_hat}
    \widehat{\varepsilon} = \left\{\begin{array}{@{}l l}
        \frac{5}{3}, & d=3, \\ 
        \frac{6}{5}, & d=4, \\ 
        \frac{9}{13}, & d=5, \\
        \frac{7}{29}, & d=6.
    \end{array}\right.
\end{equation}
\end{lemma}

\begin{proof}
The proof is divided into three key steps. We first formulate the fixed-point problem, then derive two additional bounds, and finally apply them in the contraction argument. 

Let $\wtil_0 + \widehat{\delta}$ be a solution of $\Acn W=0$. Then $\widehat{\delta}$ satisfies
\begin{equation*} 
    \Lctil(\widehat{\delta}) = (\Lctil - \Lc)(\wtil_0 + \widehat{\delta}) + (\Vtil_0 - V_0)(\wtil_0 + \widehat{\delta}). 
\end{equation*} 
As we are only interested in the behavior of $\widehat{\delta}$ on $[0, y^\ast]$, we apply a cutoff $\chi$
\begin{equation*} 
    \Lctil(\delta) = \chi(\Lctil-\Lc)(\wtil_0) 
        + \chi(\Vtil_0-V_0)(\wtil_0) + (\Lctil-\Lc)(\delta) + (\Vtil_0-V_0)(\delta).
\end{equation*}
The cutoff eliminates the exponential growth of $\wtil_0$ for large $y$ without changing the behavior of $\widehat{\delta}$ on $[0, y^\ast]$. This leads to the definition of
\begin{equation}\label{eqn:contraction_op_khat}
    \Kchat(\delta) = \Lctil^{-1} \big(\chi(\Lctil-\Lc)(\wtil_0) 
        + \chi(\Vtil_0-V_0)(\wtil_0) + (\Lctil-\Lc)(\delta) + (\Vtil_0-V_0)(\delta)\big).
\end{equation}

We show that $\Kchat$ is a contraction on
\begin{equation}\label{eqn:functionspace_delta_hat}
 \mc Y_{\hat \varepsilon} :=\{\delta\in C^1([0, \infty)),\ \|\delta\|_{\mc Y} \leq\widehat{\varepsilon}\}.
\end{equation}
While we can reuse the bounds (\ref{eqn:bound_approx_linear}) and (\ref{eqn:bound_inverse}), we also need to derive two additional estimates. Starting with the bound for $\Vtil_0-V_0$, we have
\begin{equation}\label{eqn:bound_potential_part}
    \|p_2(\Vtil_0-V_0) \delta\|_\infty \leq 
        \left\| \frac{p_2}{p_1} (\Vtil_0-V_0) \right\|_\infty \|p_1\delta\|_\infty \leq
        \left\| \frac{p_2}{p_1} (\Vtil_0-V_0) \right\|_\infty \|\delta\|_{\mc Y}.
\end{equation}
Substituting $f_0 = \ftil_0 + \delta_0$ gives
\begin{align*}
    \left| \frac{p_2}{p_1} (\Vtil_0-V_0) \right| &= 
        \left| \frac{2(d-1)p_2}{p_1 y^2}(\sin^2(\ftil_0 + \delta_0) - \sin^2(\ftil_0)) \right| \\
    &= \left| \frac{2(d-1)p_2}{p_1 y^2}\sin(2\ftil_0 + \delta_0)\sin(\delta_0) \right|.
\end{align*}
After applying bounds 
\begin{equation*}
    |\sin(2\ftil_0 + \delta_0)| \leq |\sin(2\ftil_0)| + |\delta_0|, \qquad |\sin(\delta_0)| \leq |\delta_0|
\end{equation*}
and including the appropriate weights, we have 
\begin{align*}
    \left| \frac{p_2}{p_1} (\Vtil_0-V_0) \right| &\leq
        |p_1\delta_0| \left(\left| \frac{2(d-1)p_2}{p_1^2y^2}\sin(2\ftil_0)\right| 
        + \frac{2(d-1)p_2}{p_1^3y^2} |p_1\delta_0|\right) \\
    &\leq \left(2 \left|\frac{(d-1)p_2}{p_1^2y^2}\sin(2\ftil_0)\right| 
        + \frac{2(d-1)p_2}{p_1^3y^2}\varepsilon\right) \varepsilon.
\end{align*}
Here we used $\delta_0 \in \mc Y_{\varepsilon}$, where $\mc Y_{\varepsilon}$ is defined in (\ref{eqn:functionspace_delta}). We can reuse the bound on the sine and the weight terms, calculated in Section~\ref{sec:bounds_nonlinear_term}, to get
\begin{equation}\label{eqn:bound_potential}
    \|p_2(\Vtil_0-V_0) \delta\|_\infty \leq c_V \|\delta\|_{\mc Y},
\end{equation}
where
\begin{equation}\label{eqn:c_potential}
    c_V = \left\{\begin{array}{@{}l l}
        \frac{25}{76309}, & d=3, \\ 
        \frac{97}{2002923}, & d=4, \\
        \frac{87}{12687862}, & d=5, \\
        \frac{69}{85422589}, & d=6.
    \end{array}\right.
\end{equation}
Having established the bound on the potential term, we now continue with the $\wtil_0$ terms. Due to the cutoff, it suffices to find rigorous bounds for $\|p_1\wtil_0\|_{L^\infty([0,y^\ast])}$ and $\|p_3\wtil_0'\|_{L^\infty([0,y^\ast])}$. However, both $p_1\wtil_0$ and $p_3\wtil_0'$ contain an exponential function and a square root, preventing us from directly applying the evaluation method. However, with a bit more work, we find that 

\begin{equation}
    \|\chi p_2(\Lctil-\Lc)(\wtil_0)\|_\infty 
        + \|\chi p_2(\Vtil_0-V_0)(\wtil_0)\|_\infty \leq c_w,
\end{equation}
where
\begin{equation}\label{eqn:c_w}
     c_w = \left\{\begin{array}{@{}l l}
        \frac{27}{4492}, & d=3, \\
        \frac{69}{29468}, & d=4, \\
        \frac{20}{98769}, & d=5, \\
        \frac{99}{6763105}, & d=6,
    \end{array}\right.
\end{equation}
and refer to Appendix~\ref{sec:app_bound2} for a detailed derivation. To finish the proof, let $\delta, \gamma\in  \mc Y_{\hat \varepsilon}$, then $\Kchat$ is self mapping since
\begin{align*}
    \|\Kchat(\delta)\|_{\mc Y} &\leq c_{\Lctil}(c_w + (c_\Lc + c_V)\|\delta\|_{\mc Y}) \leq c_{\Lctil}(c_w\widehat{\varepsilon}^{-1} + (c_\Lc + c_V)) \widehat{\varepsilon},
\end{align*}
where 
\begin{equation}\label{eqn:c_self_mapping_hat}
    c_{\Lctil}(c_w\widehat{\varepsilon}^{-1} + (c_\Lc + c_V)) \leq 
    \left\{\begin{array}{@{}l l}
        \frac{92}{109}, & d=3, \\ 
        \frac{74}{75}, & d=4, \\ 
        \frac{21}{23}, & d=5, \\
        \frac{92}{93}, & d=6.
    \end{array}\right.
\end{equation}
Similarly, $\Kchat$ being a contraction follows with
\begin{equation*}
    \|\Kchat(\delta) - \Kchat(\gamma)\|_{\mc Y} \leq c_{\Lctil}(c_\Lc + c_V) \|\delta - \gamma\|_{\mc Y},
\end{equation*}
where we have
\begin{equation}\label{eqn:c_contraction_hat}
    c_{\Lctil}(c_\Lc + c_V) \leq 
    \left\{\begin{array}{@{}l l}
        \frac{9}{98}, & d=3, \\ 
        \frac{4}{49}, & d=4, \\
        \frac{1}{19}, & d=5, \\
        \frac{9}{545}, & d=6.
    \end{array}\right.
\end{equation}

Since $\Kchat$ is a contraction on $ \mc Y_{\hat \varepsilon}$, there exists a unique $\widehat{\delta}_0 \in  \mc Y_{\hat \varepsilon}$ such that $\widehat{W}_0 := \wtil_0 + \widehat{\delta}_0$ is a solution of
\begin{equation*}
    \Acn W = (1-\chi)(\Lctil - \Lc)(\wtil_0) + (1-\chi)(\Vtil_0 - V_0)(\wtil_0).
\end{equation*}

Thus, $\widehat{W}_0$ is also a solution of $\Acn W = 0$ for $y\in[0,y^\ast]$. 

Next we adjust $\widehat{W}_0$ to satisfy the boundary conditions in (\ref{eqn:eigenvalue_ivp}). We already have that $\widehat{W}_0(0)=0$, since $\wtil_0(0)=0$ and $\widehat{\delta}_0(0)=0$. After normalizing with $c_0 := 1/\widehat{W}_0'(0)$ we have $W_0(y)=c_0\widehat{W}_0(y)$ for $y\in[0,y^\ast]$. Note that $c_0>0$, since $\wtil_0'(0)>0$ and $|\widehat{\delta}_0'(0)| \leq \frac{\widehat{\varepsilon}}{p_3(0)} < \wtil_0'(0)$. 

Finally with
\begin{equation*}
 \|\chi p_1(W_0 - c_0\wtil_0)\|_{\infty} + \|\chi p_3(W'_0 - c_0\wtil'_0)\|_{\infty}  \leq \| c_0 \widehat{\delta}_0\|_{\mc Y} \leq c_0\widehat{\varepsilon},
\end{equation*}
the result of the lemma follows.
\end{proof}

\begin{lemma}\label{lem:one_zero}
The function $W_0$ has exactly one zero in $(0,\infty$).    
\end{lemma}

\begin{proof}
The bounds throughout this proof are on $[0,y^\ast]$, unless stated otherwise. As a direct result of Lemma~\ref{lem:wtil_approx_W0} we have the bounds
\begin{align}
    |W_0 - c_0\wtil_0| &\leq \frac{c_0\widehat{\varepsilon}}{p_1} \label{eqn:bound_W}, \\
    |W_0' - c_0\wtil_0'| &\leq \frac{c_0\widehat{\varepsilon}}{p_3} \label{eqn:bound_Wd},
\end{align}
which, in turn, imply 
\begin{align}
    \left|q -c_0\frac{\wtil_0}{y}\right| &\leq \frac{c_0\widehat{\varepsilon}}{p_1y} \label{eqn:bound_q}, \\
    \left|q' - c_0\frac{\wtil_0'}{y} + c_0\frac{\wtil_0}{y^2}\right| &\leq 
        \frac{c_0\widehat{\varepsilon}}{p_1y^2}+\frac{c_0\widehat{\varepsilon}}{p_3y}. \label{eqn:bound_qd}
\end{align}
We will show that $q>0$ on $[0, 1]$, by using the lower bound provided by (\ref{eqn:bound_q})
\begin{equation*}
    \frac{c_0\wtil_0}{y} - \frac{c_0\widehat{\varepsilon}}{p_1 y} \leq q.
\end{equation*}
Similarly, we will derive $W_0'<0$ on $[1, y^\ast]$, with the upper bound of (\ref{eqn:bound_Wd})
\begin{equation*}
    W_0' \leq c_0\wtil_0' + \frac{c_0\widehat{\varepsilon}}{p_3}.
\end{equation*}
As these bounds contain both exponential functions and square roots, we will refer to the Appendix~\ref{sec:app_bound3} for the derivation of the bounds

\begin{align}
    \frac{\wtil_0}{y}-\frac{\widehat{\varepsilon}}{p_1 y} > 0,& \qquad\text{for } y\in [0, 1], \label{eqn:lower_bound_q}\\
    \wtil_0' + \frac{\widehat{\varepsilon}}{p_3} < 0,& \qquad\text{for } y \in [1, y^\ast] \label{eqn:upper_bound_w0}.
\end{align}
Evaluating the upper bounds on $q$ and $q'$ at $y^\ast$ gives
\begin{align}
    q(y^\ast) &\leq 
        \left. c_0\left(\frac{\wtil_0}{y} + \frac{\widehat{\varepsilon}}{p_1 y}\right)\right\vert_{y=y^\ast} \leq 
        -c_0\left\{\begin{array}{@{}l l}
            \frac{1}{94}, & d=3, \\ 
            \frac{5}{1652}, & d=4, \\
            \frac{9}{716}, & d=5, \\
            \frac{2}{197}, & d=6, 
        \end{array}\right. \label{eqn:q_negative_bound}\\
    q'(y^\ast) &\leq 
        \left. c_0\left(\frac{\wtil_0'}{y} - \frac{\wtil_0}{y^2} 
            + \frac{\widehat{\varepsilon}}{p_1 y^2} 
            + \frac{\widehat{\varepsilon}}{p_3 y}\right)\right\vert_{y=y^\ast} \leq
        -c_0\left\{\begin{array}{@{}l l}
            \frac{7}{40}, & d=3, \\
            \frac{5}{57}, & d=4, \\
            \frac{1}{28}, & d=5, \\
            \frac{5}{721}, & d=6.
        \end{array}\right. \label{eqn:dq_negative_bound}
\end{align}
With $q>0$ on $[0, 1]$, it follows that $W_0$ is positive on $(0,1]$. Since $W_0'<0$ on $[1, y^\ast]$ and $q(y^\ast)<0$, we know that $W_0$ crosses zero exactly once in $(0, y^\ast]$.

As previously outlined, if $q(y^\ast)<0$ and $q'(y^\ast)<0$, then both $q$ and $W_0$ remain negative. Therefore, $W_0$ has exactly one zero on $(0, \infty)$, completing the proof of the lemma.
\end{proof}

We note that $\widehat{\varepsilon}$ was chosen to be almost as large as possible without (\ref{eqn:lower_bound_q}), (\ref{eqn:upper_bound_w0}), (\ref{eqn:q_negative_bound}) or (\ref{eqn:dq_negative_bound}) failing. This minimizes the necessary number of coefficients in the definitions of $\ftil_0$, $\wtil_0$ and $\wtil_1$, thereby reducing the computational cost of the evaluation method.

\begin{lemma}\label{lem:zero_no_eigenvalue}
    The operator $\Acn$ has no eigenvalue $\lambda=0$.
\end{lemma}

\begin{proof}
For $y\to\infty$ the solution $W_0$ has, up to normalization, the two possible asymptotics
\begin{align*}
    W_0(y) &= 1+ \Oh(y^{-2}) \\
    W_0(y) & = y^{-d} e^{y^2/4} (1 + \Oh(y^{-2})).
\end{align*}
The first gives 
\begin{equation*}
    q(y) = y^{-1} + \Oh(y^{-3}), \qquad q'(y) = -y^{-2} + \Oh(y^{-4}),
\end{equation*}
which implies that for large enough $y$, $q$ and $q'$ would have different signs. This directly contradicts the previous conclusion that $q$ and $q'$ are both negative for all $y\geq y^\ast$.

This only leaves the second behavior. However, then
\begin{equation*}
    \big| W_0(y)\big|^2\rho(y) = \big|y^{-d}e^{y^2/4}(1+\Oh(y^{-2}))\big|^2\,y^{d-1}e^{-y^2/4} = y^{-d-1}e^{y^2/4}(1+\Oh(y^{-2}))^2
\end{equation*}
and we conclude that  $W_0\notin\mathcal{H}$.
\end{proof}


\subsubsection{The case $\lambda=-1$}\label{sec:6_2}

In this section, we continue with the case of $\lambda=-1$. We show that $W_{-1}$ is the eigenfunction associated with the gauge eigenvalue $\lambda=-1$ and determine its number of zeros on $(0, \infty)$.

\begin{lemma}\label{lem:w1_determined}
There exists a constant $c_{-1}>0$ such that $W_{-1}(y)=c_{-1}yf_0'(y)$. Furthermore, $W_{-1}$ is the eigenfunction of $\Acn$ corresponding to the eigenvalue $\lambda=-1$.
\end{lemma}
\begin{proof}
A direct calculation, using that $f_0$ solves \eqref{Eqf}, shows
\begin{equation*}
    \Acn(yf_0') = -yf_0'.
\end{equation*}
From Theorem~\ref{thm:closeness}, we know that $f_0'(0)>0$, implying $c_{-1}=1/f_0'(0)>0$. Finally, Theorem~\ref{thm:closeness} also implies that $W_{-1} \in \mc H$, completing the proof.
\end{proof}




\begin{lemma}\label{lem:Wm1_positive}
The function $W_{-1}$ is positive on $(0,\infty)$.
\end{lemma}
\begin{proof}
From Lemma~\ref{lem:w1_determined} we have $W_{-1}=c_{-1}yf_0'$. Theorem~\ref{thm:closeness} implies
\begin{equation*}
    |f_0'-\ftil_0'| \leq \frac{\varepsilon}{p_3},
\end{equation*}
which gives the lower bound
\begin{equation*}
    W_{-1} \geq c_{-1}y\ftil_0' - \frac{c_{-1}\varepsilon y}{p_3}.
\end{equation*}
After rewriting this, we get
\begin{equation*}
    y\ftil_0' - \frac{\varepsilon y}{p_3} = \frac{y}{(y^2+4)^{3/2}}\left((y^2+4)^{3/2}\,\ftil_0' 
        - 4\,\varepsilon\right).
\end{equation*}
Based on the Analysis in Section~\ref{sec:remainder}, the second term is a rational function in $y^2$. After applying the evaluation method with variable transformation (\ref{eqn:variable_transform_01}) we have \footnote{\ The corresponding CSV-files are \texttt{coefs\_fdpos\_n(d)}.}
\begin{equation}\label{eqn:fd_positive}
    (y^2+4)^{3/2}\,\ftil_0' - 4\,\varepsilon \geq 
    \left\{\begin{array}{@{}l l}
        \frac{7}{4}, & d=3, \\
        \frac{5}{3}, & d=4, \\
        1, & d=5, \\
        \frac{4}{7}, & d=6,
    \end{array}\right.
\end{equation}
for all $y\in [0, \infty)$, which in turn implies the result of the lemma.
\end{proof}


\subsubsection{The case $\lambda<-1$}

For the final case of $\lambda<-1$, we use the Picone Identity \cite{Picone_1910} to determine the number of zeros of $W_{\lambda}$ on $(0,\infty)$, closely following the proof of Lemma~3.5 in \cite{BieDon18}. 

\begin{lemma}[Picone Identity]\label{lem:picone}
For $a_1, a_2, b_1, b_2 \in C^1(\R)$ let $u$ and $v$ be solutions of
\begin{equation*}
    (a_1 u')' + b_1u = 0 \qquad \text{and} \qquad (a_2v')'+b_2v = 0.
\end{equation*}
Then for all $x\in\R$ with $v(x)\neq0$ we have
\begin{equation*}
    \bigg(\frac{u}{v}(a_1u'v - a_2uv')\bigg)' = (b_2-b_1)u^2 + (a_1-a_2)(u')^2 
        + a_2\left(u'-\frac{uv'}{v}\right)^2.
\end{equation*}
\end{lemma}
\begin{proof}
    The result follows immediately through a direct computation.
\end{proof}

\begin{lemma}\label{lem:smaller_m1}
For any $\lambda<-1$, $W_\lambda$ has no zeros in $(0,\infty)$.
\end{lemma}
\begin{proof}
Applying the Picone identity for $u=W_\lambda$ and $v=W_{-1}$, with $a_1=a_2=\rho$ and $b_2-b_1=(-1-\lambda)\rho$ gives
\begin{equation}\label{eqn:aux_lemma_Wl}
    \left(\rho\frac{W_\lambda}{W_{-1}}(W_\lambda'W_{-1} - W_\lambda W_{-1}')\right)' = 
    (-1-\lambda)\rho W_\lambda^2 + \rho\left(W_\lambda' - \frac{W_\lambda W_{-1}'}{W_{-1}}\right)^2.
\end{equation}
By Lemma~\ref{lem:Wm1_positive} this must hold for all $y\in(0,\infty)$. Now suppose there exists $y_0\in(0,\infty)$, the first zero of $W_\lambda$. Integrating then gives
\begin{equation*}
    \int_0^{y_0} \left(\rho\frac{W_\lambda}{W_{-1}}(W_\lambda'W_{-1} - W_\lambda W_{-1}')\right)'\, dy =
        \rho\frac{W_\lambda}{W_{-1}}(W_\lambda'W_{-1} - W_\lambda W_{-1}')\bigg|_0^{y_0} = 0,
\end{equation*}
where we used $\rho(0)=0$ and $W_\lambda(y_0)=0$. Since $-1-\lambda>0$ and $\rho>0$ the right-hand side of (\ref{eqn:aux_lemma_Wl}) can only be zero if the both squared terms are zero on $[0, y_0]$. This then gives the contradiction that $W_\lambda(y)=0$ for all $y\in[0,y_0]$. 
\end{proof}

\subsubsection{Proof of Theorem \ref{Th:Existence_f0}}

By Lemma~\ref{lem:one_zero}, the function $W_0$ has exactly one zero on $(0, \infty)$. Moreover, Lemmas~\ref{lem:Wm1_positive} and~\ref{lem:smaller_m1}, show that for any $\lambda \in (-\infty, -1]$, the function $W_\lambda$ has no zeros on $(0, \infty)$.

Applying Theorem 1.2 of \cite{GesSimTes1996}, we conclude that for any $\lambda \in (-\infty, -1]$, the interval $[\lambda, 0)$ contains exactly one simple eigenvalue, which is $\lambda = -1$, by Lemma~\ref{lem:w1_determined}. Finally, since Lemma~\ref{lem:zero_no_eigenvalue} ensures that $\lambda = 0$ is not an eigenvalue, we infer that $\sigma(\mc A_0) \cap (-\infty,0] = \{-1\}$. 

The spectral result follows by unitary equivalence of $L_{f_0}$ and $-\mc A_0$, noting that
\begin{equation*}
    \mc U(yf_0') = | \mathbb S^{n-1} |^{-\frac{1}{2}} f_0'(|\cdot|).
\end{equation*}
The monotonicity of $f_0$ follows from Lemma~\ref{lem:Wm1_positive}, which also implies $f_0>0$ on $(0,\infty)$. This completes the proof of Theorem~\ref{Th:Existence_f0}.

\section{Computer-assisted methods}\label{Sec:ComAss}

In this section, we describe the computer-assisted methods used to derive the rigorous bounds required throughout the previous section. The approach, introduced in Section~5 of \cite{DonSch26}, combines exact rational arithmetic and analytical bounds. We first recall the fundamental ideas, then derive a truncated Chebyshev expansion tailored to our application.The implementation and all required coefficient files are available at \url{https://git.uibk.ac.at/csaw4329/hmhf_paper_code}.

\subsection{Evaluation Methods}\label{sec:evaluation_method}
The main idea of the evaluation method is to evaluate the function at multiple points and control the behavior in between using rigorous bounds on the derivative. These bounds are established in the following lemma.

\begin{lemma}\label{lem:mvt_bound}
For $a, b \in\R$ with $a<b$, let $f\in C^1([a, b])$. Then for any $M$ satisfying $\| f' \|_{L^\infty([a,b])}\leq M$ and any $x \in [a, b]$, we have
\begin{equation*}
    \left| f(x) - \frac{f(b)+f(a)}{2}\right| \leq \frac{b-a}{2}M.
\end{equation*}
\end{lemma}
\begin{proof}
Applying the Mean Value Theorem for $x \in [a, b]$ directly gives
\begin{align*}
    f(a) - (x-a)M &\leq f(x) \leq f(a) + (x-a)M\\
    f(b) - (b-x)M &\leq f(x) \leq f(b) + (b-x)M
\end{align*}    
The result follows after adding the inequalities.
\end{proof}

For $N\in\mathbb N$, evaluate the function on the set
\begin{equation}
    \Omega_N = \left\{\left. x_k := a + \frac{b-a}{N}k \right| k=0, ..., N \right\}
\end{equation}
and applying Lemma~\ref{lem:mvt_bound} on the intervals $[x_k, x_{k+1}]$ gives rigorous bounds on the function. These bounds can be improved arbitrarily by increasing the number of evaluation points.

To obtain the required derivative bounds, we consider polynomials in Chebyshev basis \footnote{\ The star indicates that the first summand of the summation is halved.} 

\begin{equation*}
    p(x) = \sumast_{n=0}^K \cnhat T_n\left(\frac{2x}{b-a} - \frac{a+b}{b-a}\right)
\end{equation*}
and define the $T$-norm of $p$ on $[a, b]$ as
\begin{equation*}
    \|p\|_{T([a, b])} := \sumast_{n=0}^K \abs{\cnhat}
\end{equation*}

The following lemma then applies this approach to rational functions.

\begin{lemma}\label{lem:frac_bound}
For $a,b \in\R$ with $a<b$, let $p,q : [a, b] \to\R$ be polynomials and assume that $\underset{[a,b]}{\min}\ |q| > 0$. For $N\in\mathbb N$, set
\begin{equation*}
    \varepsilon = \frac{b-a}{2N} \frac{\| p'q - pq' \|_{T([a,b])}}{\underset{[a,b]}{\min}\ q^2}
\end{equation*}
we then have
\begin{equation*}
    \underset{[a,b]}{\max}\ \frac{p}{q} \leq 
        \underset{\Omega_N}{\max}\ \frac{p}{q} + \varepsilon, \qquad
    \underset{[a,b]}{\min}\ \frac{p}{q} \geq 
        \underset{\Omega_N}{\min}\ \frac{p}{q} - \varepsilon
\end{equation*}
and
\begin{equation*}
    \left\| \frac{p}{q} \right\|_{L^\infty([a,b])} \leq 
        \underset{\Omega_N}{\max}\ \left| \frac{p}{q} \right| + \varepsilon.
\end{equation*}
\end{lemma}
\begin{proof}
Follows after applying Lemma~\ref{lem:mvt_bound} to each $[x_k, x_{k+1}]$ and using that $|T_n|\leq 1$ on $[-1, 1]$. For further details we refer to \cite{DonSch26}.
\end{proof}

The bounds in the above lemma still hold if we replace the minimum of $q^2$ with a lower bound, as long as this bound remains positive. To determine such a bound, we apply the same approach to compute a positive lower bound on $\underset{[a,b]}{\min}\ \abs{q}$.

Even with exact values on
\begin{equation*}
    \|pq' - p'q\|_{L^\infty([a,b])} \quad \text{and} \quad\underset{[a,b]}{\min}\ q^2
\end{equation*}
the bound on the derivative in Lemma \ref{lem:frac_bound} may be very inefficient if these values are attained at two distant points. In some applications the resulting number of evaluations will be too large to be computed in any reasonable time frame. 

To resolve this, we split the interval $[a,b]$ and recursively apply the methods on $[a,(a+b)/2]$ and $[(a+b)/2,b]$ before combining the results. Note that this happens at the very start of the method, to avoid computing unnecessary Chebyshev expansions and further subdivisions take place as needed when applying the methods to the subintervals.

\begin{remark}
In all applications throughout this paper, the endpoints and all coefficients are rational numbers. Hence the evaluations on $\Omega_N$, as well as the resulting bounds, can be computed using exact rational arithmetic. This ensures that the computer-assisted estimates are rigorous.
\end{remark}


\subsection{Truncated Chebyshev Expansion}\label{sec:truncated_chebyshev_expansion}
The main computational cost of the evaluation method lies in computing the Chebyshev coefficients and evaluating the function on the grid. Working with truncated Chebyshev expansions instead of the full Chebyshev transformation significantly reduces the cost of both steps. The corresponding rigorous tail bounds are then carried through all subsequent estimates, ensuring that the method remains rigorous.

In this section, we introduce a method for efficiently computing the exact low-degree coefficients and deriving rigorous bounds on high-degree tail.

Let $p(x) = \sum_{k=0}^K c_k x^k$, for $c_k \in \Q$. Then for $\alpha, \beta \in\Q$, with $\beta>0$ the Chebyshev coefficients of $p$ of the interval $[\alpha-\beta, \alpha+\beta]$ are
\begin{equation}
    \cnhat = \frac{2}{\pi} \int_{-1}^1 p(\alpha + \beta x)\, T_n(x)\, \frac{dx}{\sqrt{1-x^2}}, \qquad n=0, ..., K,
\end{equation}
where $T_n$ denotes the $n$-th Chebyshev polynomial, defined by
\begin{equation}
    T_{n+1} = 2xT_n - T_{n-1}\quad\text{with}\quad T_0(x)=1,\ T_1(x)=x.
\end{equation}
The polynomial can now be written as
\begin{equation*}
    p(\alpha + \beta x) = \sumast_{n=0}^K \cnhat T_n(x), 
\end{equation*}
and truncating with $N\leq K$ gives the bounds
\begin{equation}
    \Bigg| p(\alpha + \beta x) - \sumast_{n=0}^N \cnhat T_n(x) \Bigg| \leq \sum_{n=N+1}^K \abs{\cnhat}.
\end{equation}
The goal is now to efficiently compute $(\cnhat)_{n=0}^N$ and find a rigorous upper bound on $\sum_{n=N+1}^K \abs{\cnhat}$. For this we define
\begin{equation}\label{eqn:Ikn_definition}
    I_{k,n}(\alpha, \beta) := \frac{2}{\pi} \int_{-1}^1 (\alpha + \beta x)^k\, T_n(x)\, \frac{dx}{\sqrt{1-x^2}}.
\end{equation}
\begin{lemma}\label{lem:Ikn_closed_form}
    Let $k, n \in \N_0$, then $I_{k,n}(\alpha, \beta)$ evaluates to 
\begin{equation*}
    I_{k,n}(\alpha, \beta) = \sum_{m=0}^{\left\lfloor\frac{k-n}{2}\right\rfloor} \binom{k}{n+2m}\binom{n+2m}{m} 2^{1-n-2m} \alpha^{k-n-2m} \beta^{n+2m}.
\end{equation*}
\end{lemma}
\begin{proof}
After expanding \eqref{eqn:Ikn_definition} with the binomial theorem, we have
\begin{equation*}
    I_{k,n}(\alpha,\beta) = \sum_{j=0}^k \binom{k}{j} \alpha^{k-j} \beta^j \left(\frac{2}{\pi} \int_{-1}^1 x^j\, T_n(x)\, \frac{dx}{\sqrt{1-x^2}}\right).
\end{equation*}
For $j \geq n$ and $j-n$ even the integral is non-zero with
\begin{equation*}
    \frac{2}{\pi} \int_{-1}^1 x^j\, T_n(x)\, \frac{dx}{\sqrt{1-x^2}} = 2^{1-j}\binom{j}{\frac{j-n}{2}}.
\end{equation*}
After the substitution $j=n+2m$ for $0\leq m \leq \left\lfloor\frac{k-n}{2}\right\rfloor$, the claimed result follows.
\end{proof}

To shorten the notation we will henceforth only write $I_{k,n}$. Given $I_{k,n}$ for $k,n\leq K$, we can write the Chebyshev coefficients as 
\begin{equation}\label{eqn:cnhat_Ikn}
    \cnhat = \sum_{k=0}^K c_k I_{k, n}.
\end{equation}

\begin{remark}
Note that $I_{k,n}=0$ for $k<n$ and that If \(\alpha\neq0\), then
\[
\text{sign}(I_{k,n}(\alpha,\beta))=
\text{sign}(\alpha)^{k-n}\text{sign}(\beta)^n.
\]
In any case, for \(\beta>0\),
\[
|I_{k,n}(\alpha,\beta)|=I_{k,n}(|\alpha|,\beta).
\]

\end{remark}

Due to the computational cost of evaluating the closed form formula we will not directly
apply it to determine $I_{k,n}$. Instead we use the following recurrence relation.

\begin{lemma}\label{lem:Ikn_recurrence}
The values $(I_{k,n})_{k,n\geq0}$ satisfy the recurrence relation
\begin{equation*}
    I_{k+1, n} = \alpha I_{k,n} + \frac{\beta}{2}(I_{k, n+1} + I_{k, n-1}) \quad\textup{with}\quad I_{0,0}=2,\ I_{k,n}=0 \textup{  for  } k<n,
\end{equation*}
in the case of $n=0$ we have
\begin{equation*}
    I_{k+1, 0} = \alpha I_{k, 0} + \beta I_{k, 1}.
\end{equation*}
\end{lemma}
\begin{proof}
Follows directly after splitting $(\alpha + \beta x)^{k+1} = \alpha (\alpha + \beta x)^k + \beta x(\alpha + \beta x)^k$ and inserting the recurrence relation $x T_n = \frac{T_{n+1} + T_{n-1}}{2}$.
\end{proof}

Starting from $I_{0,0}=2$ and $I_{k,n}=0$ for $k<n$, we can now compute $I_{k, n}$ for $k, n = 0, ..., K$. However, for the truncated Chebyshev expansion we only require the moments $I_{k,n}$ for $k=0, ..., K$ and $n=0, ..., N$. As illustrated in Figure~\ref{fig:momenta}, the desired moments form the region $\mathcal D$, while the recurrence relation also requires the additional moments indicated by the auxiliary region $\mathcal A$.

\begin{figure}[h]
\centering

\begin{minipage}{0.45\textwidth}
\raggedright
\scalebox{0.8}{
\begin{tikzpicture}[>=stealth]

\pgfmathsetmacro{\Kval}{5.75}
\pgfmathsetmacro{\Nval}{1}
\pgfmathsetmacro{\DiagHit}{(\Kval+\Nval)/2}

\fill[gray!50] (\Nval,\Nval) -- (\Kval,\Nval) -- (\Kval,0) -- (0,0) -- cycle;
\node at (3.25,0.4) {$\mathcal D$};

\fill[gray!15] (\Nval, \Nval) -- (\Kval, \Nval) -- (\DiagHit,\DiagHit) -- cycle;
\node at (3.25,1.75) {$\mathcal A$};

\draw[dashed] (\Kval,\Nval) -- (\DiagHit,\DiagHit);
\node[left] at (0,\DiagHit) {$\frac{K+N}{2}$};

\draw[->,thick] (0,0) -- (7,0) node[right] {$k$};
\draw[->,thick] (0,0) -- (0,6) node[above] {$n$};

\draw[thick] (0,0) -- (6,6) node[pos=0.75, above left] {$k=n$};

\draw[dashed] (\Kval,0) -- (\Kval,\Kval);
\node[below] at (\Kval,0) {$K$};

\draw[dashed] (\Nval,\Nval) -- (\Kval,\Nval);
\node[left] at (0,\Nval) {$N$};

\end{tikzpicture}
}
\end{minipage}
\hfill
\begin{minipage}{0.45\textwidth}
\raggedleft
\scalebox{0.8}{
\begin{tikzpicture}[>=stealth]

\pgfmathsetmacro{\Kval}{5.75}
\pgfmathsetmacro{\Nval}{1}
\pgfmathsetmacro{\DiagHit}{(\Kval+\Nval)/2}

\fill[gray!50] (\Nval,\Nval) -- (\Kval,\Nval) -- (\Kval,0) -- (0,0) -- cycle;
\node at (3.25,0.4) {$\mathcal D$};

\fill[gray!15] (\Kval, \Nval) -- (\Kval,0) -- (\Nval+\Kval,0) -- cycle;
\node at (6.0,0.4) {$\mathcal A$};

\draw[dashed] (\Kval,\Nval) -- (\Nval+\Kval,0);
\node[below] at (\Kval+\Nval,0) {$K+N$};

\draw[->,thick] (0,0) -- (7.5,0) node[right] {$k$};
\draw[->,thick] (0,0) -- (0,6) node[above] {$n$};

\draw[thick] (0,0) -- (6,6) node[pos=0.75, above left] {$k=n$};

\draw[dashed] (\Kval,0) -- (\Kval,\Kval);
\node[below] at (\Kval,0) {$K$};

\draw[dashed] (\Nval,\Nval) -- (\Kval,\Nval);
\node[left] at (0,\Nval) {$N$};

\end{tikzpicture}
}
\end{minipage}

\caption{Comparison of the computational complexity for the direct use of Lemma~\ref{lem:Ikn_recurrence} on the left and the modified procedure using also the $n=0$ recurrence from Lemma~\ref{lem:I_n0_recurrence} on the right.}
\label{fig:momenta}
\end{figure}
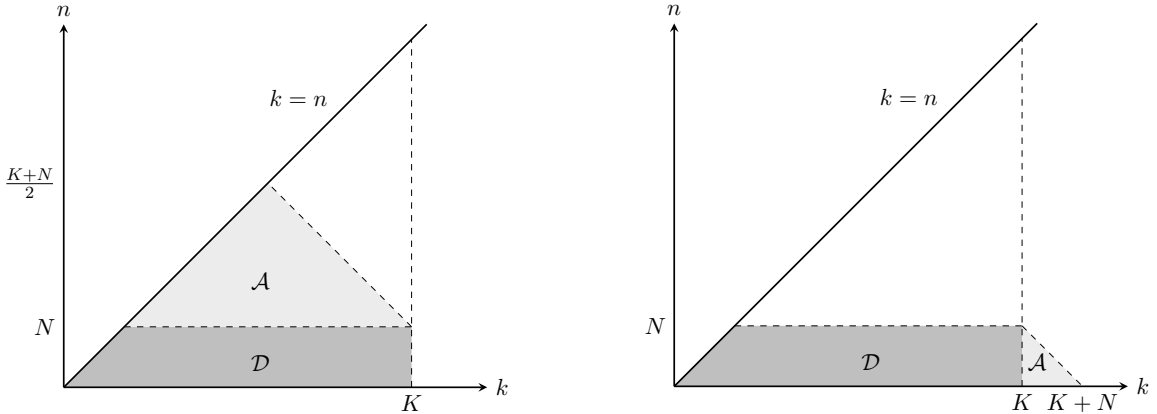

\begin{remark}
While Lemma~\ref{lem:Ikn_recurrence} could also be used to determine $I_{k,n}$ for $k=0, ..., K$ and $n=0, ..., N$, this would again scale as $\Oh(K^2)$. If instead $I_{k,0}$ for $k=0, ..., K+N$ can be computed efficiently, the method scales as $\Oh(NK)$. This, however, requires a different recurrence relation.
\end{remark}

\begin{lemma}\label{lem:I_n0_recurrence}
The sequence $(I_{k,0})_{k\geq0}$ satisfies the recurrence relation
\begin{equation*}
    I_{k, 0} = \frac{2k-1}{k}\alpha I_{k-1, 0} + \frac{k-1}{k}(\beta^2-\alpha^2)I_{k-2, 0} \quad\textup{with}\quad I_{0,0} = 2,\ I_{1, 0} = 2\alpha.
\end{equation*}
\end{lemma}
\begin{proof}
We start by splitting $k I_{k,0} = A_k + B_k$ for
\begin{align*}
    A_k &= \sum_{m=0}^{\left\lfloor\frac{k}{2}\right\rfloor} (k-2m) \binom{k}{2m}\binom{2m}{m} 2^{1-2m} \alpha^{k-2m} \beta^{2m} \\
    B_k &= \sum_{m=1}^{\left\lfloor\frac{k}{2}\right\rfloor} 2m \binom{k}{2m}\binom{2m}{m} 2^{1-2m} \alpha^{k-2m} \beta^{2m}.
\end{align*}
For the first term we use $(k-2m)\binom{k}{2m} = k\binom{k-1}{2m}$ and have
\begin{align*}
    A_k &= \sum_{m=0}^{\left\lfloor\frac{k}{2}\right\rfloor} k \binom{k-1}{2m}\binom{2m}{m} 2^{1-2m} \alpha^{k-2m} \beta^{2m} \\
    &= \sum_{m=0}^{\left\lfloor\frac{k-1}{2}\right\rfloor} k \binom{k-1}{2m}\binom{2m}{m} 2^{1-2m} \alpha^{k-2m} \beta^{2m} \\
    &= k \alpha I_{k-1, 0}.
\end{align*}
Similarly for the second term we use $\binom{k}{2m} = \binom{k-1}{2m-1} + \binom{k-1}{2m}$ to further split the sum to get
\begin{align*}
    &\sum_{m=1}^{\left\lfloor\frac{k}{2}\right\rfloor} 2m \binom{k-1}{2m-1}\binom{2m}{m} 2^{1-2m} \alpha^{k-2m} \beta^{2m} \\
    &\qquad = \sum_{m=0}^{\left\lfloor\frac{k-2}{2}\right\rfloor} (2m+2) \binom{k-1}{2m+1}\binom{2m+2}{m+1} 2^{-1-2m} \alpha^{k-2m-2} \beta^{2m+2} \\
    &\qquad = \sum_{m=0}^{\left\lfloor\frac{k-2}{2}\right\rfloor} (k-1) \binom{k-2}{2m}\binom{2m}{m} 2^{1-2m} \alpha^{k-2m-2} \beta^{2m+2} \\
    &\qquad = (k-1) \beta^2 I_{k-2, 0}
\end{align*}
and  
\begin{align*}
    &\sum_{m=1}^{\left\lfloor\frac{k}{2}\right\rfloor} 2m \binom{k-1}{2m}\binom{2m}{m} 2^{1-2m} \alpha^{k-2m} \beta^{2m} \\
    &\qquad = \sum_{m=1}^{\left\lfloor\frac{k-1}{2}\right\rfloor} 2m \binom{k-1}{2m}\binom{2m}{m} 2^{1-2m} \alpha^{k-2m} \beta^{2m} \\
    &\qquad = \alpha B_{k-1}.
\end{align*}
Similarly splitting $I_{k-1, 0}$ gives 
\begin{equation*}
    B_{k-1} = (k-1)I_{k-1, 0} - A_{k-1} = (k-1)(I_{k-1, 0} - \alpha I_{k-2, 0}).
\end{equation*}
Combing these results gives the claimed recurrence relation.
\end{proof}

The remaining moments in the truncated strip can then be computed from
\[
I_{k,1}=\frac{I_{k+1,0}-\alpha I_{k,0}}{\beta},
\]
and, for \(n\ge2\),
\[
I_{k,n}
=
\frac{2}{\beta}
\left(I_{k+1,n-1}-\alpha I_{k,n-1}\right)
-
I_{k,n-2}.
\]
Thus, after computing $I_{k,0}$ for $k=0,\dots,K+N$, all moments $I_{k,n}$ with $0\le k\le K$ and $0\le n\le N$  are obtained in $\Oh(NK)$ arithmetic operations.\\

It remains to find a rigorous upper bound on $\sum_{n=N+1}^K \abs{\cnhat}$. For this we define 
\begin{equation}
    S_k := \sum_{n=N+1}^k \abs{I_{k,n}}.
\end{equation}
\begin{lemma}
For $N\leq K$ we have that
\begin{equation*}
    \sum_{n=N+1}^K \abs{\cnhat} \leq \sum_{k=0}^K \abs{c_k} S_k.
\end{equation*}    
\end{lemma}
\begin{proof}
Using the representation in (\ref{eqn:cnhat_Ikn}) the claimed bound follows with
\begin{equation*}
    \sum_{n=N+1}^K \abs{\cnhat} =  \sum_{n=N+1}^K \abs{\sum_{k=0}^K c_k I_{k, n}} \leq \sum_{k=0}^K \abs{c_k} \sum_{n=N+1}^K \abs{I_{k, n}} = \sum_{k=0}^K \abs{c_k} S_k,
\end{equation*}
where we used that $I_{k,n}=0$ for $k<n$.
\end{proof}

The last step is to efficiently compute $S_0, S_1, ..., S_K$.

\begin{lemma}
For $1 \leq N\leq K$, the sequence $(S_k)_{k\geq0}$ satisfies the recurrence relation
\begin{equation*}
    S_{k+1} = (\abs{\alpha} + \beta) S_k + \left(\abs{\alpha} + \frac{\beta}{2}\right)\abs{I_{k,N}} + \frac{\beta}{2}\abs{I_{k, N-1}} - \abs{I_{k+1, N}}\quad\textup{with}\quad S_0 = 0.
\end{equation*}
\end{lemma}
\begin{proof}
Since $\abs{I_{k,n}(\alpha, \beta)} = I_{k,n}(\abs{\alpha}, \beta)$, we start by assuming that $\alpha\geq 0$. Then $I_{k,n}\geq 0$ for all $k,n\geq 0$ and after inserting the recurrence from Lemma \ref{lem:Ikn_recurrence} we get
\begin{align*}
    S_{k+1} &= \sum_{n=N+1}^{k+1} \alpha I_{k,n} + \frac{\beta}{2}(I_{k, n+1} + I_{k, n-1}) \\
    &= \alpha S_k + \frac{\beta}{2}(S_k - I_{k, N+1}) + \frac{\beta}{2}(S_k + I_{k, N}) \\
    &= (\alpha + \beta) S_k + \left(\alpha + \frac{\beta}{2}\right)I_{k,N} + \frac{\beta}{2}I_{k, N-1} - I_{k+1, N}.
\end{align*}
The claimed result follows after taking absolute values of all terms depending on $\text{sign}(\alpha)$.
\end{proof}

This gives a method for computing a truncated Chebyshev expansion of degree $N$ with a rigorous tail bound for a polynomial of degree $K$ in $\Oh (NK)$ operations. In practice, this significantly reduces the computational cost, since $N$ is typically much smaller than $K$.

\appendix

\section{Proof of Lemma \ref{Lemmacompactperturbation}} \label{App1}

Let $\tau > 0$, then by \cite{EngelNagel}, Theorem C.7 it suffices to show that $S_f^X(\tau - \tau^{\prime}) L_f^{\prime} S^X_0(\tau^{\prime})$ is a compact operator for each $\tau^{\prime} \in (0, \tau)$. Furthermore, since $S_f^X(\tau - \tau^{\prime})$ is a linear and bounded operator, it suffices to show that $L_f^{\prime} S^X_0(\tau^{\prime})$ is a compact operator for each $\tau^{\prime} \in (0, \tau)$. To prove this we use a variant of the Rellich-Kondrachov Theorem (see Theorem 10, \cite{HanOls2010}) and show that $[L_f^{\prime} S_0^X(\tau)](B_1^X)$ is relatively compact in $\dot{H}^{\lfloor s \rfloor}_r(\R^n) \cap \dot{H}^k_r(\R^n)$, where $B_1^X = \{ u \in X_s^k(\R^n) : \Vert u \Vert_{X_s^k(\R^n)} \leq 1\}$. The continuous embedding of $\dot{H}^{\lfloor s \rfloor}_r(\R^n) \cap \dot{H}^k_r(\R^n)$ into $X_s^k(\R^n)$ then implies relative compactness of $[L_f^{\prime} S_0^X(\tau)](B_1^X)$ in $X_s^k(\R^n)$.\\
Let $\alpha_1, \alpha_2 \in \N_0^n$ with $|\alpha_1| = \lfloor s \rfloor$ and $|\alpha_2| = k$ and define $K_{\alpha_1} := \partial^{\alpha_1} ( V \cdot S_0^X(\tau)(B_1^X))$ and $K_{\alpha_2} := \partial^{\alpha_2} (V \cdot S_0^X(\tau)(B_1^X))$. To apply Theorem 10 in \cite{HanOls2010} we need to show that each of these subsets is bounded in $H^1(\R^3)$ and for all $\varepsilon > 0$ there exists $R >0$ such that for all $u$ belonging to the corresponding subspace we have that 
\begin{align} \label{CompactnessTheorem}
\int_{|x| \geq R} |u(x)|^2 + |\nabla u(x)|^2 dx < \varepsilon.
\end{align}
If $u \in B_1^X$ and $\alpha_i \in \N_0^n$, $i \in \{ 1,2 \}$ we have 
\begin{align*}
\partial^{\alpha_i} (V_f \cdot S_0^X(\tau)(u)) = \sum_{\beta \leq \alpha_i} \binom{\alpha_i}{\beta} \partial^{\beta} V_f \cdot \partial^{\alpha_i - \beta} (S_0^X(\tau)(u)).
\end{align*}
For each $\beta \in \N_0^n$, $\beta \leq \alpha_i$, $i \in \{1,2 \}$ we define
\begin{align*}
u_{\beta}^{i} := \partial^{\beta} V_f \cdot \partial^{\alpha_i - \beta} u_{S_0},
\end{align*}
where we abbreviate $u_{S_0} := S_0^X(\tau)(u)$. By additivity, it suffices to prove that any $u_{\beta}^{i}$ with $\beta \in \N_0^n$, $\beta \leq \alpha_i$, $i \in \{1,2 \}$ is bounded in $H^1(\R^n)$ and that for all $\varepsilon > 0$ there exists $R >0$ such that for all $u_{\beta}^{i}$ estimate \eqref{CompactnessTheorem} holds. Let $u_{\beta}^{i}$ with $\beta \in \N_0^n$, $\beta \leq \alpha_i$, $i \in \{1,2 \}$, then we observe
\begin{align} \label{H1norm}
\Vert u_{\beta}^{i} \Vert_{H^1(\R^n)}^2 &= \int_{\R^n} |u_{\beta}^{i}(x)|^2 + |\nabla u_{\beta}^{i}(x)|^2 dx \nonumber \\
&\lesssim \int_{\R^n} \left( |\partial^{\beta} V_f(x)|^2 +  |\nabla \partial^{\beta} V_f(x)|^2 \right) |\partial^{\alpha_i - \beta} u_{S_0}(x)|^2 + |\partial^{\beta} V_f(x)|^2 |\nabla \partial^{\alpha_i - \beta} u_{S_0}(x)|^2 dx.
\end{align}
In the following we proof boundedness in $H^1(\R^n)$ by splitting the integral above into one that integrats over $B_R(\R^n)$ and one that integrats over $B_R^c(\R^n)$ for some $R >1$. In addition, we are able to proof \eqref{CompactnessTheorem} in this way too. In the following we will make extensive use of estimate  of Lemma \ref{Le:Potential} for any partial derivative of the potential $V_f$. Furthermore, we use that we can estimate
\begin{align*}
\Vert u_{S_0} \Vert_{X_s^k(\R^n)} \leq e^{-\omega_0 \tau} \Vert u \Vert_{X_s^k(\R^n)} \leq 1,
\end{align*}
by \eqref{Estimatefreesemigroup} for $u \in B_1^X$.\\
\\
First, let $i=1$, $\alpha_1 \in \N_0^n$ with $|\alpha_1| = \lfloor s \rfloor$ and $\beta \in \N_0^n$ with $\beta \leq \alpha_1$. By Lemma \ref{LemmaEmbedding} we infer
\begin{align} \label{Estimate1}
\Vert \partial^{\alpha_1-\beta} u_{S_0} \Vert_{L^{\infty}(\R^n)} \lesssim \Vert \partial^{\alpha_1-\beta} u_{S_0} \Vert_{\dot{H}_r^{s-|\alpha_1|+|\beta|} \cap \dot{H}_r^{k-|\alpha_1|+|\beta|}(\R^n)} \lesssim \Vert u_{S_0} \Vert_{X_s^k(\R^n)} \lesssim 1.
\end{align}
By this the first integral over the ball with radius $R >1$ can be estimated by
\begin{align*}
\int_{B_R(\R^n)} \left( |\partial^{\beta} V_f(x)|^2 +  |\nabla \partial^{\beta} V_f(x)|^2 \right) |\partial^{\alpha_i - \beta} u_{S_0}(x)|^2 dx \lesssim 1,
\end{align*}
where we used \eqref{Estimate1} and boundedness of all partial derivatives of $V_f$. On $B_R^c(\R^n)$ we need to distinguish two cases. If $|\beta|=|\alpha_1|$, i.e. if $\beta = \alpha_1$ we observe by the decay of $V_f$ and \eqref{Estimate1}
\begin{align*}
\int_{B_R^c(\R^n)} \left( |\partial^{\alpha_1} V_f(x)|^2 +  |\nabla \partial^{\alpha_1} V_f(x)|^2 \right) |u_{S_0}(x)|^2 dx \lesssim \int_{B_R^c(\R^n)} |x|^{-4-|\alpha1|} dx \lesssim R^{-(4-n+2 |\beta|)},
\end{align*}
where $4-n+2|\beta| \geq 2$ since $|\alpha_1| = \lfloor s \rfloor = \lfloor \frac{n}{2} \rfloor -1 \geq \frac{n}{2}-\frac{3}{2}$. If $|\beta| < |\alpha_1|$ we choose 
\begin{align} \label{Gamma1}
\gamma \in \left( \frac{n}{2}-2-|\beta|, \min \left\lbrace \frac{n-1}{2}, |\alpha_1|-|\beta| + \frac{n}{2}-s \right\rbrace \right),
\end{align}
to estimate
\begin{align*}
&\int_{B_R^c(\R^n)} \left( |\partial^{\beta} V_f(x)|^2 +  |\nabla \partial^{\beta} V_f(x)|^2 \right) |\partial^{\alpha_i - \beta} u_{S_0}(x)|^2 dx \\
&\lesssim \left( \int_{B_R^c(\R^n)} |x|^{-4-2|\beta|-2\gamma} dx \right) \Vert |\cdot|^{\gamma} \partial^{\alpha_i - \beta} u_{S_0} \Vert_{L^{\infty}(\R^n)}^2 \\
&\lesssim R^{-(4-n+2|\beta|+2 \gamma)} \Vert u_{S_0} \Vert_{\dot{H}^{|\alpha_1|-|\beta|+\frac{n}{2}-\gamma}} \\
&\lesssim R^{-(4-n+2|\beta|+2 \gamma)},
\end{align*}
where we used an weighted $L^{\infty}(\R^n)$ estimate for radial homogenous Sobolev spaces,  see Proposition B.1 in \cite{Glo25}, that is applicable due to $\frac{1}{2} < \gamma < \frac{n}{2}$ by \eqref{Gamma1}. Furthermore, also by \eqref{Gamma1} we find $s \leq |\alpha_1|-|\beta|+\frac{n}{2}-\gamma \leq k$ which means we can estimate 
\begin{align*}
\Vert u_{S_0} \Vert_{\dot{H}^{|\alpha_1|-|\beta|+\frac{n}{2}-\gamma}} \lesssim \Vert u_{S_0} \Vert_{X_s^k(\R^n)} \lesssim 1.
\end{align*}
To treat the second part of the integral \eqref{H1norm} we distinguish three cases. First, if $|\beta| =0$ we find
\begin{align*}
&\int_{B_R(\R^n)} | V_f(x)|^2 |\nabla \partial^{\alpha_i} u_{S_0}(x)|^2 dx + \int_{B_R^c(\R^n)} + | V_f(x)|^2 |\nabla \partial^{\alpha_i} u_{S_0}(x)|^2 dx \\
&\lesssim \left(1 + \Vert V_f \Vert_{L^{\infty}(\B_R^c(\R^n)}^2 \right) \Vert \nabla \partial^{\alpha_1} u_{S_0} \Vert_{L^2(\R^n)}^2 \\
&\lesssim 1 + R^{-4},
\end{align*}
where we exploited $k > |\alpha_1|+1 = \lfloor s \rfloor +1 >s$ and the decay of $V_f$. If $|\beta| = |\alpha_1|$, i.e. $\beta = \alpha_1$ we observe
\begin{align*}
&\int_{B_R(\R^n)} |\partial^{\alpha_1} V_f(x)|^2 |\nabla u_{S_0}(x)|^2 dx + \int_{B_R^c(\R^n)} |\partial^{\alpha_1} V_f(x)|^2 |\nabla u_{S_0}(x)|^2 dx \\
&\lesssim \left(1 + \int_{B_R^c(\R^n)} |x|^{-4-2|\alpha_1|} dx \right) \Vert \nabla u_{S_0} \Vert_{L^{\infty}(\R^n)}^2 \\
&\lesssim (1 + R^{-2}) \Vert \nabla u_{S_0} \Vert_{\dot{H}_r^{s-1} \cap \dot{H}_r^{k-1}(\R^n)}^2\\
&\lesssim 1 + R^{-2},
\end{align*}
by Lemma \ref{LemmaEmbedding}. If $1 \leq |\beta| \leq |\alpha_1|-1$ we have 
\begin{align*}
\Vert \nabla \partial^{\alpha_i -\beta} u_{S_0} \Vert_{L^{\infty}(\R^n)} \lesssim \Vert \nabla \partial^{\alpha_i -\beta} u_{S_0} \Vert_{\dot{H}_r^{s-|\alpha_1|+|\beta|-1} \cap \dot{H}_r^{k-|\alpha_1|+|\beta|-1}(\R^n)} \lesssim \Vert u_{S_0} \Vert_{X_s^k(\R^n)} \lesssim 1,
\end{align*}
by Lemma \ref{LemmaEmbedding}. This implies
\begin{align*}
\int_{B_R(\R^n)} |\partial^{\beta} V_f(x)|^2  |\nabla \partial^{\alpha_i - \beta} u_{S_0}(x)|^2 dx \lesssim 1.
\end{align*}
To estimate the integral over $B_R^c(\R^n)$ we choose
\begin{align} \label{Gamma2}
\gamma \in \left( \frac{n}{2}-2-|\beta|, \min \left\lbrace \frac{n-1}{2}, |\alpha_1|-|\beta| + \frac{n}{2}+1 -s \right\rbrace \right),
\end{align}
to find 
\begin{align*}
&\int_{B_R(\R^n)^c} |\partial^{\beta} V_f(x)|^2  |\nabla \partial^{\alpha_i - \beta} u_{S_0}(x)|^2 dx\\
&\lesssim \left( \int_{B_R^c(\R^n)} |x|^{-4-2|\beta|-2 \gamma} dx \right) \Vert |\cdot|^{\gamma} \nabla \partial^{\alpha_i - \beta} u_{S_0} \Vert_{L^{\infty}(\R^n)}^2\\
&\lesssim R^{-(4-n+2|\beta|+2 \gamma)} \Vert u_{S_0} \Vert_{\dot{H}^{|\alpha_1|-|\beta|+1+\frac{n}{2}-\gamma}} \\
&\lesssim R^{-(4-n+2|\beta|+2 \gamma)},
\end{align*}
where we again used Proposition B.1 in \cite{Glo25}. Furthermore, by \eqref{Gamma2} we have $4-n+2|\beta|+2 \gamma >0$.\\
\\
Next, let $i=2$, $\alpha_2 \in \N_0^n$ with $|\alpha_2| = k = n+1$ and $\beta \in \N_0^n$ with $\beta \leq \alpha_2$. We distinguish again between different values of $|\alpha_2|-|\beta|$. If $|\alpha_2|-|\beta| \leq \lfloor \frac{n+1}{2} \rfloor$ we can estimate
\begin{align*}
\Vert \partial^{\alpha_2-\beta} u_{S_0} \Vert_{L^{\infty}(\R^n)} \lesssim \Vert \partial^{\alpha_2-\beta} u_{S_0} \Vert_{\dot{H}^s \cap \dot{H}^{\frac{n+1}{2} (\R^n)}} \lesssim \Vert u_{S_0} \Vert_{\dot{H}^{s+|\alpha_2|-|\beta|} \cap \dot{H}^{\frac{n+1}{2} + |\alpha_2|-|\beta|} (\R^n)} \lesssim \Vert u_{S_0} \Vert_{X_s^k(\R^n)},
\end{align*}
by Lemma \ref{LemmaEmbedding} and the fact that $\frac{n+1}{2} + |\alpha_2|-|\beta| \leq n+1$. This implies
\begin{align*}
&\int_{\R^n} \left( |\partial^{\beta} V_f(x)|^2 +  |\nabla \partial^{\beta} V_f(x)|^2 \right) |\partial^{\alpha_2 - \beta} u_{S_0}(x)|^2 dx \\
&\lesssim \left( 1 + \int_{B_R^c(\R^n)} |x|^{-4-2|\beta|} dx \right) \Vert \partial^{\alpha_2-\beta} u_{S_0} \Vert_{L^{\infty}(\R^n)}^2 \\
&\lesssim 1 + R^{-5},
\end{align*}
since $|\beta| \geq \frac{n+1}{2}$ in this case. If $|\alpha_2|-|\beta| \geq \lfloor \frac{n+1}{2} \rfloor +1$ we estimate
\begin{align*}
&\int_{\R^n} \left( |\partial^{\beta} V_f(x)|^2 +  |\nabla \partial^{\beta} V_f(x)|^2 \right) |\partial^{\alpha_2 - \beta} u_{S_0}(x)|^2 dx \\
&\lesssim \left( 1 + \Vert \partial^{\beta} V_f \Vert_{L^{\infty}(B_R^c(\R^n))}^2 + \Vert \nabla \partial^{\beta} V _f\Vert_{L^{\infty}(B_R^c(\R^n))}^2 \right) \Vert \partial^{\alpha_2-\beta} u_{S_0} \Vert_{L^{2}(\R^n)}^2 \\
&\lesssim 1 + R^{-4-2|\beta|},
\end{align*}
as we have $s < |\alpha_2|-|\beta| \leq n+1$. For the second part of the integral \eqref{H1norm} we distinguish three cases. First, if $|\alpha_2|-|\beta|+1 \leq \lfloor \frac{n+1}{2} \rfloor$ we do the same as for the first integral part of \eqref{H1norm}. We have
\begin{align*}
\Vert \nabla \partial^{\alpha_2-\beta} u_{S_0} \Vert_{L^{\infty}(\R^n)} \lesssim \Vert \nabla \partial^{\alpha_2-\beta} u_{S_0} \Vert_{\dot{H}^s \cap \dot{H}^{\frac{n+1}{2} (\R^n)}} \lesssim \Vert u_{S_0} \Vert_{\dot{H}^{s+1+|\alpha_2|-|\beta|} \cap \dot{H}^{\frac{n+1}{2} +1+ |\alpha_2|-|\beta|} (\R^n)} \lesssim \Vert u_{S_0} \Vert_{X_s^k(\R^n)},
\end{align*}
by Lemma \ref{LemmaEmbedding} and the fact that $\frac{n+1}{2} +1  + |\alpha_2|-|\beta| \leq n+1$. Again, we infer 
\begin{align*}
&\int_{\R^n} |\partial^{\beta} V_f(x)|^2 |\nabla \partial^{\alpha_2 - \beta} u_{S_0}(x)|^2 dx \\
&\lesssim \left( 1 + \int_{B_R^c(\R^n)} |x|^{-4-2|\beta|} dx \right) \Vert \nabla \partial^{\alpha_2-\beta} u_{S_0} \Vert_{L^{\infty}(\R^n)}^2 \\
&\lesssim 1 + R^{-7},
\end{align*}
since $|\beta| \geq \frac{n+3}{2}$. Next, if $n+1 \geq |\alpha_2|-|\beta| +1 \geq \lfloor \frac{n+1}{2} \rfloor +1$, we have
\begin{align*}
&\int_{\R^n} |\partial^{\beta} V_f(x)|^2 |\nabla \partial^{\alpha_2 - \beta} u_{S_0}(x)|^2 dx \\
&\lesssim \left( 1 + \Vert \partial^{\beta} V_f \Vert_{L^{\infty}(B_R^c(\R^n))}^2 \right) \Vert \nabla \partial^{\alpha_2-\beta} u_{S_0} \Vert_{L^{2}(\R^n)}^2 \\
&\lesssim 1 + R^{-4-2|\beta|},
\end{align*}
since $s < |\alpha_2|-|\beta| +1 \leq n+1$. For the remaining case $|\beta| =0$ we exploit the following property of the free semigroup $(S_0^X(\tau))_{\tau \geq 0}$. 
\begin{align*}
\Vert \nabla \partial^{\alpha_2} (S_0^X(\tau) u) \Vert_{L^2(\R^n)} &\lesssim \Vert S_0^X(\tau) u \Vert_{\dot{H}^{n+2}(\R^n)} \\
&\lesssim e^{-\frac{1}{2} \left( \frac{n}{2}+3 \right) \tau} \Vert |\cdot|^2 \mathcal{F}(H_{\alpha(\tau)}) \Vert_{L^{\infty}(\R^n)} \Vert u \Vert_{\dot{H}^{n+1}(\R^n)} \\
&\lesssim m(\tau) \Vert u \Vert_{X_s^k(\R^n)},
\end{align*}
by definition of the free semigroup \eqref{Freesemigroup} for some constant $m(\tau) >0$. This shows
\begin{align*}
\int_{\R^n} |V_f(x)|^2 |\nabla \partial^{\alpha_2} u_{S_0}(x)|^2 dx \lesssim (1+ \Vert V_f \Vert_{L^{\infty}(B_R^c(\R^n))}^2 \Vert \nabla \partial^{\alpha_2} (S_0^X(\tau) u) \Vert_{L^2(\R^n)} \lesssim 1 + R^{-4}.
\end{align*}
By Theorem 10 in \cite{HanOls2010}, we infer that each subset $K_{\alpha_1}, K_{\alpha_2}$ for all possible choices of $\alpha_1, \alpha_2 \in \N_0^n$ is relatively compact in $L^2(\R^n)$, implying that $[L_f^{\prime} S_0^X(\tau)](B_1^X)$ is relatively compact in $\dot{H}^{\lfloor s \rfloor}_r(\R^n) \cap \dot{H}^k_r(\R^n)$ and hence also in $X_s^k(\R^n)$.

\section{Additional bounds} \label{App2}
\subsection{Bounds required for Lemma~\ref{lem:wtil_approx_W0}}\label{sec:app_bound2}

In this section we derive the upper bounds for $\|\chi p_2(\Lctil-\Lc)(\wtil_0)\|_\infty$ and $\|\chi p_2(\Vtil_0-V_0)(\wtil_0)\|_\infty$, which are required for the proof of Lemma~\ref{lem:wtil_approx_W0}. Since we now apply these operators to the explicitly know $\wtil_0$, the bounds $\|p_1\wtil_0\|_\infty \leq \|\wtil_0\|$ and $\|p_3\wtil_0'\|_\infty \leq \|\wtil_0\|$, previously required in (\ref{eqn:bound_approx_linear_parts}) and (\ref{eqn:bound_potential_part}), are no longer necessary. This gives the more accurate bounds
\begin{align}
    \|\chi p_2(\Lctil-\Lc)(\wtil_0)\|_\infty &\leq c_P\|p_3\wtil_0'\|_{L^\infty([0, y^\ast])} 
        + c_Q\|p_1\wtil_0\|_{L^\infty([0, y^\ast])}\\
    \|\chi p_2(\Vtil_0-V_0)(\wtil_0)\|_\infty &\leq c_V \|p_1\wtil_0\|_{L^\infty([0, y^\ast])}.
\end{align}
It now remains to find rigorous bounds on $\|p_1\wtil_0\|_{L^\infty([0, y^\ast])}$ and $\|p_3\wtil_0'\|_{L^\infty([0, y^\ast])}$, which include both an exponential and a root term. Since we are considering compact intervals, we can can derive simple bounds on the behavior of the exponential by truncating the Taylor expansion and bounding the tail. Truncating after the $n$-th term, then for $y\in[0, n+2)$ we have the tail bound
\begin{equation*}
    \sum_{k=n+1}^\infty \frac{y^k}{k!} \leq \frac{y^{n+1}}{(n+1)!} \sum_{k=0}^\infty \frac{y^k}{(n+2)^k} = \frac{y^{n+1}}{(n+1)!} \frac{n+2}{n+2-y},
\end{equation*}
which gives the bounds 
\begin{equation}\label{eqn:exp_bounds}
    \sum_{k=0}^n \frac{y^k}{k!} \leq e^y \leq \sum_{k=0}^n \frac{y^k}{k!} + \frac{y^{n+1}}{(n+1)!} \frac{n+2}{n+2-y}.
\end{equation}
For our application we truncate after the 10-th term and since $\frac{{y^\ast}^2}{4} < 6$ we get good bounds on the behavior of $e^{y^2/4}$ for all $y\in[0, y^\ast]$. After replacing the exponential term in both $\wtil_0$ and $\wtil_0'$ with the upper bound, we apply the evaluation method with the variable transformation (\ref{eqn:variable_transform_03}) to convert the $\sqrt{y^2+4}$ term to a rational function.\footnote{\ The corresponding CSV-files are \texttt{coefs\_w0norm\_n(d)} and \texttt{coefs\_w0dnorm\_n(d)}.} Note that this transformation maps $y^\ast$ to an irrational number, hence we apply the method to a slightly larger interval with 
\begin{equation}\label{eqn:yast_transformed}
    \frac{\sqrt{{y^\ast}^2+4} - 2}{y^\ast} \leq 
    \left\{\begin{array}{@{}l l}
        \frac{7}{13}, & d=3, \\
        \frac{7}{12}, & d=4, \\
        \frac{5}{8}, & d=5, \\
        \frac{2}{3}, & d=6.
    \end{array}\right.
\end{equation}
This gives the bounds
\begin{align}
    \|p_1\wtil_0\|_{L^\infty([0, y^\ast])} &\leq 
    \left\{\begin{array}{@{}l l}
        \frac{1107}{89}, & d=3, \\
        \frac{32}{3}, & d=4, \\
        \frac{791}{72}, & d=5, \\
        \frac{554}{39}, & d=6
    \end{array}\right. \\
    \|p_3\wtil_0'\|_{L^\infty([0, y^\ast])} &\leq 
    \left\{\begin{array}{@{}l l}
        \frac{1284}{59}, & d=3, \\
        \frac{725}{39}, & d=4, \\
        \frac{999}{71}, & d=5, \\
        \frac{1351}{93}, & d=6.
    \end{array}\right.
\end{align}


\subsection{Bounds required for Lemma~\ref{lem:one_zero}}\label{sec:app_bound3}
The proof of Lemma \ref{lem:one_zero} required $\frac{\wtil_0}{y} - \frac{\widehat{\varepsilon}}{p_1 y} > 0$ on $[0, 1]$ and $\wtil_0' + \frac{\widehat{\varepsilon}}{p_3}<0$ on $[1, y^\ast]$. Both terms again contain an exponential and a root term, preventing a direct application of the evaluation method.

We proceed similarly to the previous section. After replacing the exponential term with the lower bound in (\ref{eqn:exp_bounds}) we use the evaluation method with the variable transformation (\ref{eqn:variable_transform_03}) and compute a lower bound on the minimum. Note that $1$ is not mapped to a rational number and we apply the evaluation method to a slightly larger interval using \footnote{\ The corresponding CSV-files are \texttt{coefs\_w0pos\_n(d)}.}
\begin{equation*}
    \sqrt{5}-2 < \frac{9}{38}.
\end{equation*}
This gives the bound
\begin{equation}
   \frac{\wtil_0}{y} - \frac{\widehat{\varepsilon}}{p_1 y} \geq 
   \left\{\begin{array}{@{}l l}
        \frac{82}{245}, & d=3, \\
        \frac{23}{39}, & d=4, \\
        \frac{71}{106}, & d=5, \\
        \frac{22}{43}, & d=6,
    \end{array}\right.\qquad \text{for } y\in [0, 1].
\end{equation}
Similarly for the second bound we also replace the exponential term with the lower bound in (\ref{eqn:exp_bounds}) and apply the evaluation method with variable transformation (\ref{eqn:variable_transform_03}) to compute an upper bound on the maximum. The method is again applied to a larger interval using
\begin{equation*}
    \sqrt{5} - 2 > \frac{4}{17}
\end{equation*}
for the left endpoint and (\ref{eqn:yast_transformed}) for the right endpoint.\footnote{\ The corresponding CSV-files are \texttt{coefs\_w0dneg\_n(d)}.} This gives the bounds
\begin{equation}
   \wtil_0' + \frac{\widehat{\varepsilon}}{p_3} \leq 
   \left\{\begin{array}{@{}l l}
        - \frac{93}{121}, & d=3, \\
        - \frac{70}{127}, & d=4, \\
        - \frac{56}{191}, & d=5, \\
        - \frac{44}{487}, & d=6,
    \end{array}\right.\qquad \text{for } y\in [1, y^\ast].
\end{equation}

\section{Tables of Coefficients}\label{Appendix:Coeff}
\hspace{1pt}

\begin{table}[ht!]
\centering
\caption{Coefficients of \(\ftil_0\) for \(d=3\)}
\label{tab:coef_f0_d3}
\vspace{-0.3cm}
\begin{tabular}{c|c c|c c|c c|c} 
\textbf{\(k\)} & \textbf{\(c_k(\ftil_0)\)} & \textbf{\(k\)} & \textbf{\(c_k(\ftil_0)\)} & \textbf{\(k\)} & \textbf{\(c_k(\ftil_0)\)} & \textbf{\(k\)} & \textbf{\(c_k(\ftil_0)\)} \\ \hline
0 & \( \frac{17917987}{9041390} \) & 1 & \( \frac{-1787334}{10331077} \) & 2 & \( \frac{73964}{2277441} \) & 3 & \( \frac{-19969}{2807730} \) \\ 
4 & \( \frac{2443}{1310683} \) & 5 & \( \frac{-559}{1078135} \) & 6 & \( \frac{279}{1751716} \) & 7 & \( \frac{-112}{2211443} \) \\ 
8 & \( \frac{71}{4121060} \) & 9 & \( \frac{-34}{5675721} \) & 10 & \( \frac{5}{2282942} \) & 11 & \( \frac{-1}{1231629} \) \\ 
12 & \( \frac{1}{3193614} \) & 13 & \( \frac{-1}{8212397} \) & 14 & \( \frac{1}{20443233} \) & 15 & \( \frac{-1}{50686210} \) \\ 
16 & \( \frac{1}{122132135} \) & 17 & \( \frac{-1}{294420759} \) & 18 & \( \frac{1}{691376676} \) & 19 & \( \frac{-1}{1596862914} \) \\ 
20 & \( \frac{1}{3804447767} \) & 21 & \( \frac{-1}{6917614371} \) & 22 & \( \frac{1}{14073643006} \) & &  \\ 
\end{tabular}
\end{table}

\begin{table}[ht!]
\centering
\caption{Coefficients of \(\ftil_0\) for \(d=4\)}
\label{tab:coef_f0_d4}
\vspace{-0.3cm}
\begin{tabular}{c|c c|c c|c c|c} 
\textbf{\(k\)} & \textbf{\(c_k(\ftil_0)\)} & \textbf{\(k\)} & \textbf{\(c_k(\ftil_0)\)} & \textbf{\(k\)} & \textbf{\(c_k(\ftil_0)\)} & \textbf{\(k\)} & \textbf{\(c_k(\ftil_0)\)} \\ \hline
0 & \( \frac{1678407}{1120145} \) & 1 & \( \frac{-2537013}{16423454} \) & 2 & \( \frac{216093}{7214351} \) & 3 & \( \frac{-14695}{2150586} \) \\ 
4 & \( \frac{17007}{9639343} \) & 5 & \( \frac{-4182}{8629697} \) & 6 & \( \frac{160}{1126623} \) & 7 & \( \frac{-335}{7759386} \) \\ 
8 & \( \frac{99}{7219261} \) & 9 & \( \frac{-13}{2911069} \) & 10 & \( \frac{23}{15298106} \) & 11 & \( \frac{-1}{1941051} \) \\ 
12 & \( \frac{1}{5513680} \) & 13 & \( \frac{-1}{15468537} \) & 14 & \( \frac{1}{42408044} \) & 15 & \( \frac{-1}{115311156} \) \\ 
16 & \( \frac{1}{307176919} \) & 17 & \( \frac{-1}{814242470} \) & 18 & \( \frac{1}{2118692601} \) & 19 & \( \frac{-1}{5508577091} \) \\ 
20 & \( \frac{1}{13839342729} \) & 21 & \( \frac{-1}{36680946502} \) & 22 & \( \frac{1}{74602017724} \) & 23 & \( \frac{-1}{160169802612} \) \\ 
\end{tabular}
\end{table}

\begin{table}[ht!]
\centering
\caption{Coefficients of \(\ftil_0\) for \(d=5\)}
\label{tab:coef_f0_d5}
\vspace{-0.3cm}
\begin{tabular}{c|c c|c c|c c|c} 
\textbf{\(k\)} & \textbf{\(c_k(\ftil_0)\)} & \textbf{\(k\)} & \textbf{\(c_k(\ftil_0)\)} & \textbf{\(k\)} & \textbf{\(c_k(\ftil_0)\)} & \textbf{\(k\)} & \textbf{\(c_k(\ftil_0)\)} \\ \hline
0 & \( \frac{35906507}{27171432} \) & 1 & \( \frac{-21019909}{114472475} \) & 2 & \( \frac{209749}{4768329} \) & 3 & \( \frac{-2918307}{236413223} \) \\ 
4 & \( \frac{87910}{23206417} \) & 5 & \( \frac{-142361}{115811713} \) & 6 & \( \frac{30995}{74455296} \) & 7 & \( \frac{-24641}{169239825} \) \\ 
8 & \( \frac{1635}{31240396} \) & 9 & \( \frac{-1670}{86801659} \) & 10 & \( \frac{109}{15106525} \) & 11 & \( \frac{-163}{59214350} \) \\ 
12 & \( \frac{106}{99366663} \) & 13 & \( \frac{-14}{33404487} \) & 14 & \( \frac{16}{95936727} \) & 15 & \( \frac{-3}{44694464} \) \\ 
16 & \( \frac{1}{36622696} \) & 17 & \( \frac{-1}{89170688} \) & 18 & \( \frac{1}{215145821} \) & 19 & \( \frac{-1}{514878404} \) \\ 
20 & \( \frac{1}{1222530085} \) & 21 & \( \frac{-1}{2882474052} \) & 22 & \( \frac{1}{6749831264} \) & 23 & \( \frac{-1}{15710621422} \) \\ 
24 & \( \frac{1}{36342676884} \) & 25 & \( \frac{-1}{83664242268} \) & 26 & \( \frac{1}{191410485855} \) & 27 & \( \frac{-1}{436135649340} \) \\ 
28 & \( \frac{1}{992887333445} \) & 29 & \( \frac{-1}{2183756191730} \) & 30 & \( \frac{1}{5155418015968} \) & 31 & \( \frac{-1}{9113413538542} \) \\ 
32 & \( \frac{1}{16722637027588} \) & &  & &  & &  \\ 
\end{tabular}
\end{table}

\begin{table}[ht!]
\centering
\caption{Coefficients of \(\ftil_0\) for \(d=6\)}
\label{tab:coef_f0_d6}
\vspace{-0.3cm}
\begin{tabular}{c|c c|c c|c c|c} 
\textbf{\(k\)} & \textbf{\(c_k(\ftil_0)\)} & \textbf{\(k\)} & \textbf{\(c_k(\ftil_0)\)} & \textbf{\(k\)} & \textbf{\(c_k(\ftil_0)\)} & \textbf{\(k\)} & \textbf{\(c_k(\ftil_0)\)} \\ \hline
0 & \( \frac{3181181323}{2539737309} \) & 1 & \( \frac{-807089099}{3187292536} \) & 2 & \( \frac{458518339}{5341055703} \) & 3 & \( \frac{-302197667}{8954113777} \) \\ 
4 & \( \frac{135400793}{9476708028} \) & 5 & \( \frac{-102660047}{16214513459} \) & 6 & \( \frac{38274411}{13214175724} \) & 7 & \( \frac{-10583201}{7798195796} \) \\ 
8 & \( \frac{1703288}{2628406219} \) & 9 & \( \frac{-502044}{1597449671} \) & 10 & \( \frac{1152857}{7465200546} \) & 11 & \( \frac{-168586}{2196583141} \) \\ 
12 & \( \frac{286457}{7435715019} \) & 13 & \( \frac{-73604}{3772695307} \) & 14 & \( \frac{332375}{33373113963} \) & 15 & \( \frac{-39527}{7718373033} \) \\ 
16 & \( \frac{13295}{5015245563} \) & 17 & \( \frac{-7748}{5611740159} \) & 18 & \( \frac{1013}{1400702663} \) & 19 & \( \frac{-1873}{4918072611} \) \\ 
20 & \( \frac{1507}{7477051893} \) & 21 & \( \frac{-186}{1735641941} \) & 22 & \( \frac{147}{2568511333} \) & 23 & \( \frac{-131}{4268198649} \) \\ 
24 & \( \frac{59}{3570443384} \) & 25 & \( \frac{-23}{2575546883} \) & 26 & \( \frac{43}{8878425910} \) & 27 & \( \frac{-16}{6070734717} \) \\ 
28 & \( \frac{19}{13204573173} \) & 29 & \( \frac{-2}{2538101677} \) & 30 & \( \frac{3}{6931479602} \) & 31 & \( \frac{-1}{4194716286} \) \\ 
32 & \( \frac{1}{7594934235} \) & 33 & \( \frac{-1}{13715673692} \) & 34 & \( \frac{1}{24707413670} \) & 35 & \( \frac{-1}{44401449457} \) \\ 
36 & \( \frac{1}{79609903024} \) & 37 & \( \frac{-1}{142421390209} \) & 38 & \( \frac{1}{254248102472} \) & 39 & \( \frac{-1}{452948657393} \) \\ 
40 & \( \frac{1}{805343108516} \) & 41 & \( \frac{-1}{1429179159267} \) & 42 & \( \frac{1}{2531576930286} \) & 43 & \( \frac{-1}{4476422230126} \) \\ 
44 & \( \frac{1}{7901770349367} \) & 45 & \( \frac{-1}{13924970285842} \) & 46 & \( \frac{1}{24503395881364} \) & 47 & \( \frac{-1}{43036104656234} \) \\ 
48 & \( \frac{1}{75539854140593} \) & 49 & \( \frac{-1}{132189039954001} \) & 50 & \( \frac{1}{231454552391722} \) & 51 & \( \frac{-1}{404345929357245} \) \\ 
52 & \( \frac{1}{698920210540690} \) & 53 & \( \frac{-1}{1240781608305367} \) & 54 & \( \frac{1}{1990463536013411} \) & 55 & \( \frac{-1}{3697906578379550} \) \\ 
56 & \( \frac{1}{4678996313335638} \) & 57 & \( \frac{-1}{6623962259232476} \) & &  & &  \\ 
\end{tabular}
\end{table}

\begin{table}[ht!]

\begin{minipage}[t]{0.5\linewidth}
\caption{Coefficients of \(\vtil_0\) for \(d=3\)}
\centering
\label{tab:coef_v0_d3}
\vspace{-0.3cm}
\begin{tabular}{c|c}
\textbf{\(k\)} & \textbf{\(c_{k, N}(\vtil_0)\)} \\
\hline
0 & \(1680\) \\ 
1 & \(0\) \\ 
2 & \(144\) \\ 
3 & \(0\) \\ 
4 & \(12\) \\  
\hline
\textbf{\(k\)} & \textbf{\(c_{k, D}(\vtil_0)\)} \\
\hline
0 & \(1680\) \\ 
1 & \(0\) \\ 
2 & \(480\) \\ 
3 & \(0\) \\ 
4 & \(72\) \\ 
5 & \(0\) \\ 
6 & \(8\) \\ 
7 & \(0\) \\ 
8 & \(1\) \\ 
\end{tabular}
\end{minipage}%
\begin{minipage}[t]{0.5\linewidth}
\caption{Coefficients of \(\vtil_0\) for \(d=5\)}
\centering
\label{tab:coef_v0_d5}
\vspace{-0.3cm}
\begin{tabular}{c|c}
\textbf{\(k\)} & \textbf{\(c_{k, N}(\vtil_0)\)} \\
\hline
0 & \(1680\) \\ 
1 & \(0\) \\ 
2 & \(120\) \\ 
\hline
\textbf{\(k\)} & \textbf{\(c_{k, D}(\vtil_0)\)} \\
\hline
0 & \(1680\) \\ 
1 & \(0\) \\ 
2 & \(480\) \\ 
3 & \(0\) \\ 
4 & \(72\) \\ 
5 & \(0\) \\ 
6 & \(8\) \\ 
7 & \(0\) \\ 
8 & \(1\) \\ 
\end{tabular}
\end{minipage}

\end{table}


\begin{table}[ht!]
\centering
\caption{Coefficients of \(\vtil_0\) for \(d=4\)}
\label{tab:coef_v0_d4}
\vspace{-0.3cm}\resizebox{.75\linewidth}{!}{
\begin{tabular}{c|c}
\textbf{\(k\)} & \textbf{\(c_{k, N}(\vtil_0)\)} \\
\hline
0 & \(579366515126277800970711700618505876785582870162990694400\) \\ 
1 & \(1193764586299578819416080927604020919434885633379197255680\) \\ 
2 & \(677436184484477983704065888154602186570197945061791498240\) \\ 
3 & \(93347599518394264229682388561028304480121080136324136960\) \\ 
4 & \(34612529389692782776324853572974768010557313876947026944\) \\ 
5 & \(6035849857615260007987783579326728373580191580323140096\) \\ 
6 & \(1765310305619229813592352943991965857236661039250215040\) \\ 
7 & \(241875323693909850166395643425619884960878762112698880\) \\ 
8 & \(133689567968387754607166779571556514083067908778713600\) \\ 
9 & \(-49391262615220223047263124844768449073699276986339328\) \\ 
10 & \(28855505866526366713167482609816632801664854421025280\) \\ 
\hline
\textbf{\(k\)} & \textbf{\(c_{k, D}(\vtil_0)\)} \\
\hline
0 & \(579366515126277800970711700618505876785582870162990694400\) \\ 
1 & \(1193764586299578819416080927604020919434885633379197255680\) \\ 
2 & \(798137541802452525572964159116790910900527709679081226240\) \\ 
3 & \(342048554997473184941365915145199329362388920423656898560\) \\ 
4 & \(187689472975216926753781638335816816057870768721739551744\) \\ 
5 & \(50094298476564583251265296847766216144246504970491306496\) \\ 
6 & \(23670684323321090095952051454732380690285053738707678720\) \\ 
7 & \(4924237385299052798359299547178072718288902305960107360\) \\ 
8 & \(2095582266046305333734588069995722799632331405002707790\) \\ 
9 & \(302693354082594553165646671332552830739585151949025027\) \\ 
10 & \(171310074915842619848910589859205035961504678684996528\) \\ 
11 & \(8402064008142629461528492684804552206140612182073595\) \\ 
12 & \(13540080848032974183739871760176862984975590010462660\) \\ 
13 & \(-293274498015283437065450785243991013554686200092400\) \\ 
14 & \(-1368081848006810155428453031921853347922068041525024\) \\ 
15 & \(799264722155204759810747032516227755159750939218740\) \\ 
\end{tabular}
}
\end{table}

\begin{table}[ht!]
\centering
\caption{Coefficients of \(\vtil_0\) for \(d=6\)}
\label{tab:coef_v0_d6}
\vspace{-0.3cm}\resizebox{.75\linewidth}{!}{
\begin{tabular}{c|c}
\textbf{\(k\)} & \textbf{\(c_{k, N}(\vtil_0)\)} \\
\hline
0 & \(5580692290047319809357206997409109485890437269425532724551700824719360\) \\ 
1 & \(11449616475552574765621120314918627769858387831876202125946150788792320\) \\ 
2 & \(6377957380372636613068591422206295021217735432271576161066642569953280\) \\ 
3 & \(741484200397125637106997435101220790149027494541364168246283733565440\) \\ 
4 & \(258077236913689397172354452124076290242430776833494444811607808196608\) \\ 
5 & \(37559978442428274803119169493491938977403798545703908636986515439616\) \\ 
6 & \(11576290130903980212442756605289061868662737858769252761939987489280\) \\ 
7 & \(822461754606230398948114002368940720709523800135239846369963812864\) \\ 
8 & \(792594470939244312895315467519605842119568274649305643911519240192\) \\ 
9 & \(-262948844569745663443589808206931392090037191175032988171951480832\) \\ 
10 & \(125369016643285531080861310782486158757584668151017183653402494976\) \\ 
\hline
\textbf{\(k\)} & \textbf{\(c_{k, D}(\vtil_0)\)} \\
\hline
0 & \(5580692290047319809357206997409109485890437269425532724551700824719360\) \\ 
1 & \(11449616475552574765621120314918627769858387831876202125946150788792320\) \\ 
2 & \(7598733818820487821365480452889537721256268584958411444562327125360640\) \\ 
3 & \(3246087804424251367086617503989670614805549832764283383297004218613760\) \\ 
4 & \(1782962910455287847162653285241797863012904550115877978916352354385920\) \\ 
5 & \(465873780207181616730609490741281780692259061412812487040000133693440\) \\ 
6 & \(225090384238373975621658432339037993997150767769559144723907085058560\) \\ 
7 & \(44372096335032259813884072431788645641552854412607647004887409034240\) \\ 
8 & \(19829386815053728528646829587072111379695810515677070923378879701520\) \\ 
9 & \(2831692269409545475466251225033455558690946894078936066784555200800\) \\ 
10 & \(1474990680230044734901674378336502438177329058286300486310404367414\) \\ 
11 & \(107656186628585981007139942252160766946643164603166884543907867073\) \\ 
12 & \(101024492479652905906099822577098991396118557148770596371602857716\) \\ 
13 & \(747496715514403170091026810481731135409010646233453156655541677\) \\ 
14 & \(6147078775874973465854430855402518829797086967096257089214686492\) \\ 
15 & \(-196174075342472880099303658283267844351174934535083512359708260\) \\ 
16 & \(-606948682423063498715672497636737943389645505198051996845768688\) \\ 
17 & \(289381608018947426145454019916960238609481655888356425336333884\) \\ 
\end{tabular}
}
\end{table}

\begin{table}[ht!]

\begin{minipage}[t]{0.5\linewidth}
\centering
\caption{Coefficients of \(\vtil_1\) for \(d=3\)}
\label{tab:coef_v1_d3}
\vspace{-0.3cm}
\begin{tabular}{c|c}
\textbf{\(k\)} & \textbf{\(c_{k, N}(\vtil_1)\)} \\
\hline
0 & \(2\) \\ 
1 & \(0\) \\ 
2 & \(1\) \\  
\hline
\textbf{\(k\)} & \textbf{\(c_{k, D}(\vtil_1)\)} \\
\hline
0 & \(1\) \\ 
\end{tabular}
\end{minipage}%
\begin{minipage}[t]{0.5\linewidth}
\centering
\caption{Coefficients of \(\vtil_1\) for \(d=4\)}
\label{tab:coef_v1_d4}
\vspace{-0.3cm}
\begin{tabular}{c|c}
\textbf{\(k\)} & \textbf{\(c_{k, N}(\vtil_1)\)} \\
\hline
0 & \(3675760\) \\ 
1 & \(3171168\) \\ 
2 & \(2108106\) \\ 
3 & \(1299597\) \\ 
4 & \(702702\) \\ 
5 & \(161694\) \\ 
\hline
\textbf{\(k\)} & \textbf{\(c_{k, D}(\vtil_1)\)} \\
\hline
0 & \(814515\) \\ 
1 & \(702702\) \\ 
2 & \(161694\) \\ 
\end{tabular}
\end{minipage}

\end{table}
\begin{table}[ht!]
\begin{minipage}[t]{0.5\linewidth}
\centering
\caption{Coefficients of \(\vtil_1\) for \(d=5\)}
\label{tab:coef_v1_d5}
\vspace{-0.3cm}
\begin{tabular}{c|c}
\textbf{\(k\)} & \textbf{\(c_{k, N}(\vtil_1)\)} \\
\hline
0 & \(12\) \\ 
1 & \(0\) \\ 
2 & \(4\) \\ 
3 & \(0\) \\ 
4 & \(1\) \\ 
\hline
\textbf{\(k\)} & \textbf{\(c_{k, D}(\vtil_1)\)} \\
\hline
0 & \(1\) \\ 
\end{tabular}
\end{minipage}%
\begin{minipage}[t]{0.5\linewidth}
\centering
\caption{Coefficients of \(\vtil_1\) for \(d=6\)}
\label{tab:coef_v1_d6}
\vspace{-0.3cm}
\begin{tabular}{c|c}
\textbf{\(k\)} & \textbf{\(c_{k, N}(\vtil_1)\)} \\
\hline
0 & \(35555190528\) \\ 
1 & \(38060923648\) \\ 
2 & \(26338094640\) \\ 
3 & \(14414091120\) \\ 
4 & \(6841779945\) \\ 
5 & \(3093704874\) \\ 
6 & \(1403257284\) \\ 
7 & \(421773300\) \\ 
8 & \(69802590\) \\ 
\hline
\textbf{\(k\)} & \textbf{\(c_{k, D}(\vtil_1)\)} \\
\hline
0 & \(984838374\) \\ 
1 & \(1054244334\) \\ 
2 & \(421773300\) \\ 
3 & \(69802590\) \\ 
\end{tabular}
\end{minipage}
\end{table}

\begin{table}[ht!]
\centering
\caption{Coefficients of \(\wtil_0\) for \(d=3\)}
\label{tab:coef_w0_d3}
\vspace{-0.3cm}
\begin{tabular}{c|c c|c c|c c|c c|c} 
\textbf{\(k\)} & \textbf{\(c_k(\wtil_0)\)} & \textbf{\(k\)} & \textbf{\(c_k(\wtil_0)\)} & \textbf{\(k\)} & \textbf{\(c_k(\wtil_0)\)} & \textbf{\(k\)} & \textbf{\(c_k(\wtil_0)\)} & \textbf{\(k\)} & \textbf{\(c_k(\wtil_0)\)} \\ \hline
0 & \( 1 \) & 1 & \( \frac{-650018}{302905} \) & 2 & \( \frac{2356067}{1728292} \) & 3 & \( \frac{-543204}{831913} \) & 4 & \( \frac{62301}{167735} \) \\ 
5 & \( \frac{-65326}{316559} \) & 6 & \( \frac{29919}{342124} \) & 7 & \( \frac{-100995}{1684276} \) & 8 & \( \frac{10466}{351979} \) & 9 & \( \frac{-5753}{549017} \) \\ 
10 & \( \frac{5544}{500639} \) & 11 & \( \frac{-743}{243232} \) & 12 & \( \frac{517}{321767} \) & 13 & \( \frac{-1168}{544833} \) & 14 & \( \frac{-67}{390178} \) \\ 
15 & \( \frac{-296}{486165} \) & 16 & \( \frac{207}{957608} \) & 17 & \( \frac{23}{200732} \) & 18 & \( \frac{42}{179105} \) & 19 & \( \frac{49}{567796} \) \\ 
20 & \( \frac{8}{196489} \) & 21 & \( \frac{-3}{102988} \) & 22 & \( \frac{-3}{81830} \) & 23 & \( \frac{-6}{164369} \) & 24 & \( \frac{-4}{209105} \) \\ 
25 & \( \frac{-3}{476996} \) & 26 & \( \frac{1}{267011} \) & 27 & \( \frac{2}{282875} \) & 28 & \( \frac{2}{286991} \) & 29 & \( \frac{1}{218864} \) \\ 
30 & \( \frac{1}{494167} \) & 31 & \( \frac{-1}{3484331490} \) & 32 & \( \frac{-1}{949515} \) & 33 & \( \frac{-1}{742170} \) & 34 & \( \frac{-1}{899237} \) \\ 
35 & \( \frac{-1}{1445920} \) & 36 & \( \frac{-1}{3850545} \) & 37 & \( \frac{1}{29028033} \) & 38 & \( \frac{1}{4970842} \) & 39 & \( \frac{1}{4414530} \) \\ 
40 & \( \frac{1}{5230596} \) & 41 & \( \frac{1}{10802323} \) & 42 & \( \frac{-1}{30791846} \) & &  & &  \\ 
\end{tabular}
\end{table}

\begin{table}[ht!]
\centering
\caption{Coefficients of \(\wtil_0\) for \(d=4\)}
\label{tab:coef_w0_d4}
\vspace{-0.3cm}
\begin{tabular}{c|c c|c c|c c|c c|c} 
\textbf{\(k\)} & \textbf{\(c_k(\wtil_0)\)} & \textbf{\(k\)} & \textbf{\(c_k(\wtil_0)\)} & \textbf{\(k\)} & \textbf{\(c_k(\wtil_0)\)} & \textbf{\(k\)} & \textbf{\(c_k(\wtil_0)\)} & \textbf{\(k\)} & \textbf{\(c_k(\wtil_0)\)} \\ \hline
0 & \( 1 \) & 1 & \( \frac{-90134}{48295} \) & 2 & \( \frac{988153}{858840} \) & 3 & \( \frac{-808427}{1556222} \) & 4 & \( \frac{74443}{291243} \) \\ 
5 & \( \frac{-9672}{73595} \) & 6 & \( \frac{40617}{757252} \) & 7 & \( \frac{-4929}{159305} \) & 8 & \( \frac{1247}{94867} \) & 9 & \( \frac{-4055}{775024} \) \\ 
10 & \( \frac{309}{67256} \) & 11 & \( \frac{-25}{116339} \) & 12 & \( \frac{212}{183409} \) & 13 & \( \frac{-208}{384851} \) & 14 & \( \frac{-103}{246042} \) \\ 
15 & \( \frac{-401}{658488} \) & 16 & \( \frac{-135}{568436} \) & 17 & \( \frac{-1}{428810} \) & 18 & \( \frac{57}{304454} \) & 19 & \( \frac{169}{868474} \) \\ 
20 & \( \frac{21}{168661} \) & 21 & \( \frac{13}{491256} \) & 22 & \( \frac{-19}{509239} \) & 23 & \( \frac{-44}{756931} \) & 24 & \( \frac{-22}{483891} \) \\ 
25 & \( \frac{-1}{47535} \) & 26 & \( \frac{-1}{29987759} \) & 27 & \( \frac{7}{650271} \) & 28 & \( \frac{5}{404094} \) & 29 & \( \frac{2}{226717} \) \\ 
30 & \( \frac{1}{244941} \) & 31 & \( \frac{1}{2392713} \) & 32 & \( \frac{-1}{661776} \) & 33 & \( \frac{-1}{492905} \) & 34 & \( \frac{-1}{579472} \) \\ 
35 & \( \frac{-1}{888822} \) & 36 & \( \frac{-1}{1911602} \) & 37 & \( \frac{-1}{13856017} \) & 38 & \( \frac{1}{4824953} \) & 39 & \( \frac{1}{3160054} \) \\ 
40 & \( \frac{1}{3184290} \) & 41 & \( \frac{1}{4334020} \) & 42 & \( \frac{1}{7517208} \) & 43 & \( \frac{1}{29234864} \) & 44 & \( \frac{-1}{44244926} \) \\ 
45 & \( \frac{-1}{17509391} \) & 46 & \( \frac{-1}{21018669} \) & 47 & \( \frac{-1}{29831524} \) & 48 & \( \frac{-1}{21891443} \) & &  \\ 
\end{tabular}
\end{table}

\begin{table}[ht!]
\centering
\caption{Coefficients of \(\wtil_0\) for \(d=5\)}
\label{tab:coef_w0_d5}
\vspace{-0.3cm}
\begin{tabular}{c|c c|c c|c c|c c|c} 
\textbf{\(k\)} & \textbf{\(c_k(\wtil_0)\)} & \textbf{\(k\)} & \textbf{\(c_k(\wtil_0)\)} & \textbf{\(k\)} & \textbf{\(c_k(\wtil_0)\)} & \textbf{\(k\)} & \textbf{\(c_k(\wtil_0)\)} & \textbf{\(k\)} & \textbf{\(c_k(\wtil_0)\)} \\ \hline
0 & \( 1 \) & 1 & \( \frac{-20785558}{11521181} \) & 2 & \( \frac{15495806}{13273235} \) & 3 & \( \frac{-9243359}{15900034} \) & 4 & \( \frac{1586702}{4973477} \) \\ 
5 & \( \frac{-1178712}{6882937} \) & 6 & \( \frac{728591}{9906161} \) & 7 & \( \frac{-161224}{3367667} \) & 8 & \( \frac{278719}{12497371} \) & 9 & \( \frac{-41798}{5168753} \) \\ 
10 & \( \frac{32880}{3921227} \) & 11 & \( \frac{-77147}{46080381} \) & 12 & \( \frac{2947}{2275160} \) & 13 & \( \frac{-5708}{3758725} \) & 14 & \( \frac{-4624}{12462993} \) \\ 
15 & \( \frac{-6743}{11314019} \) & 16 & \( \frac{169}{4745664} \) & 17 & \( \frac{1148}{12087581} \) & 18 & \( \frac{1027}{4689522} \) & 19 & \( \frac{860}{6394503} \) \\ 
20 & \( \frac{994}{13471875} \) & 21 & \( \frac{-11}{4469968} \) & 22 & \( \frac{-173}{5135922} \) & 23 & \( \frac{-340}{7823313} \) & 24 & \( \frac{-263}{8237766} \) \\ 
25 & \( \frac{-44}{2691859} \) & 26 & \( \frac{-39}{19737940} \) & 27 & \( \frac{147}{23468371} \) & 28 & \( \frac{127}{14022874} \) & 29 & \( \frac{139}{17808651} \) \\ 
30 & \( \frac{15}{3100814} \) & 31 & \( \frac{17}{9738220} \) & 32 & \( \frac{-3}{6158147} \) & 33 & \( \frac{-10}{6165203} \) & 34 & \( \frac{-3}{1649366} \) \\ 
35 & \( \frac{-5}{3472282} \) & 36 & \( \frac{-14}{16711441} \) & 37 & \( \frac{-1}{3689959} \) & 38 & \( \frac{1}{7860666} \) & 39 & \( \frac{1}{3067439} \) \\ 
40 & \( \frac{2}{5554409} \) & 41 & \( \frac{1}{3438757} \) & 42 & \( \frac{1}{5582427} \) & 43 & \( \frac{1}{14252949} \) & 44 & \( \frac{-1}{91669933} \) \\ 
45 & \( \frac{-1}{17718539} \) & 46 & \( \frac{-1}{14202646} \) & 47 & \( \frac{-1}{15962753} \) & 48 & \( \frac{-1}{22832272} \) & 49 & \( \frac{-1}{44010043} \) \\ 
50 & \( \frac{-1}{200747207} \) & 51 & \( \frac{1}{147482463} \) & 52 & \( \frac{1}{80167552} \) & 53 & \( \frac{1}{75627642} \) & 54 & \( \frac{1}{91491874} \) \\ 
55 & \( \frac{1}{138903121} \) & 56 & \( \frac{1}{286846120} \) & 57 & \( \frac{1}{2226377220} \) & 58 & \( \frac{-1}{698795889} \) & 59 & \( \frac{-1}{426937724} \) \\ 
60 & \( \frac{-1}{430620135} \) & 61 & \( \frac{-1}{545716544} \) & 62 & \( \frac{-1}{1092557571} \) & 63 & \( \frac{1}{14948803387} \) & &  \\ 
\end{tabular}
\end{table}

\begin{table}[ht!]
\centering
\caption{Coefficients of \(\wtil_0\) for \(d=6\)}
\label{tab:coef_w0_d6}
\vspace{-0.3cm}
\begin{tabular}{c|c c|c c|c c|c c|c} 
\textbf{\(k\)} & \textbf{\(c_k(\wtil_0)\)} & \textbf{\(k\)} & \textbf{\(c_k(\wtil_0)\)} & \textbf{\(k\)} & \textbf{\(c_k(\wtil_0)\)} & \textbf{\(k\)} & \textbf{\(c_k(\wtil_0)\)} & \textbf{\(k\)} & \textbf{\(c_k(\wtil_0)\)} \\ \hline
0 & \( 1 \) & 1 & \( \frac{-47652279}{26191717} \) & 2 & \( \frac{34644161}{25569907} \) & 3 & \( \frac{-11380805}{12877876} \) & 4 & \( \frac{32420909}{56818391} \) \\ 
5 & \( \frac{-37769270}{102927907} \) & 6 & \( \frac{4580539}{20023311} \) & 7 & \( \frac{-13416447}{92129545} \) & 8 & \( \frac{2558049}{28195928} \) & 9 & \( \frac{-1638758}{29087257} \) \\ 
10 & \( \frac{899323}{25117084} \) & 11 & \( \frac{-662771}{30633181} \) & 12 & \( \frac{330178}{23819583} \) & 13 & \( \frac{-339550}{39926651} \) & 14 & \( \frac{254327}{49496718} \) \\ 
15 & \( \frac{-177613}{51630619} \) & 16 & \( \frac{56186}{29583583} \) & 17 & \( \frac{-18031}{13818204} \) & 18 & \( \frac{108206}{135172509} \) & 19 & \( \frac{-13661}{32491369} \) \\ 
20 & \( \frac{20053}{55706850} \) & 21 & \( \frac{-1514}{10377671} \) & 22 & \( \frac{802}{6523541} \) & 23 & \( \frac{-2403}{28767967} \) & 24 & \( \frac{648}{29402173} \) \\ 
25 & \( \frac{-1538}{33549791} \) & 26 & \( \frac{367}{58310946} \) & 27 & \( \frac{-233}{19519592} \) & 28 & \( \frac{686}{65555331} \) & 29 & \( \frac{25}{12025148} \) \\ 
30 & \( \frac{150}{19922981} \) & 31 & \( \frac{61}{44062539} \) & 32 & \( \frac{43}{26290988} \) & 33 & \( \frac{-21}{15980876} \) & 34 & \( \frac{-27}{27989407} \) \\ 
35 & \( \frac{-25}{17061021} \) & 36 & \( \frac{-65}{91604749} \) & 37 & \( \frac{-18}{38307307} \) & 38 & \( \frac{1}{42645472} \) & 39 & \( \frac{7}{47930373} \) \\ 
40 & \( \frac{11}{41069434} \) & 41 & \( \frac{1}{4651679} \) & 42 & \( \frac{1}{5925529} \) & 43 & \( \frac{3}{36997513} \) & 44 & \( \frac{1}{40587070} \) \\ 
45 & \( \frac{-1}{43232326} \) & 46 & \( \frac{-2}{47952243} \) & 47 & \( \frac{-1}{21180596} \) & 48 & \( \frac{-1}{26146339} \) & 49 & \( \frac{-1}{39069371} \) \\ 
50 & \( \frac{-1}{87874316} \) & 51 & \( \frac{-1}{1805862275} \) & 52 & \( \frac{1}{151908411} \) & 53 & \( \frac{1}{106344548} \) & 54 & \( \frac{1}{108141322} \) \\ 
55 & \( \frac{1}{140526234} \) & 56 & \( \frac{1}{230114612} \) & 57 & \( \frac{1}{586953492} \) & 58 & \( \frac{-1}{3795425332} \) & 59 & \( \frac{-1}{695287583} \) \\ 
60 & \( \frac{-1}{533793708} \) & 61 & \( \frac{-1}{562469187} \) & 62 & \( \frac{-1}{730261268} \) & 63 & \( \frac{-1}{1175024011} \) & 64 & \( \frac{-1}{2757167363} \) \\ 
65 & \( \frac{1}{76975187354} \) & 66 & \( \frac{1}{4034497647} \) & 67 & \( \frac{1}{2851751791} \) & 68 & \( \frac{1}{2847902087} \) & 69 & \( \frac{1}{3477674193} \) \\ 
70 & \( \frac{1}{5140425808} \) & 71 & \( \frac{1}{9946705258} \) & 72 & \( \frac{1}{45349217889} \) & 73 & \( \frac{-1}{31617003600} \) & 74 & \( \frac{-1}{16395151285} \) \\ 
75 & \( \frac{-1}{14565805234} \) & 76 & \( \frac{-1}{16049911071} \) & 77 & \( \frac{-1}{21320418839} \) & 78 & \( \frac{-1}{33670810531} \) & 79 & \( \frac{-1}{77619629444} \) \\ 
80 & \( \frac{-1}{1420087279918} \) & 81 & \( \frac{1}{125158988872} \) & 82 & \( \frac{1}{89743734804} \) & 83 & \( \frac{1}{84631820919} \) & 84 & \( \frac{1}{116761871589} \) \\ 
85 & \( \frac{1}{195350597374} \) & 86 & \( \frac{1}{176578936899} \) & &  & &  & &  \\ 
\end{tabular}
\end{table}

\begin{table}[ht!]
\centering
\caption{Coefficients of \(\wtil_1\) for \(d=3\)}
\label{tab:coef_w1_d3}
\vspace{-0.3cm}
\begin{tabular}{c|c c|c c|c c|c c|c} 
\textbf{\(k\)} & \textbf{\(c_k(\wtil_1)\)} & \textbf{\(k\)} & \textbf{\(c_k(\wtil_1)\)} & \textbf{\(k\)} & \textbf{\(c_k(\wtil_1)\)} & \textbf{\(k\)} & \textbf{\(c_k(\wtil_1)\)} & \textbf{\(k\)} & \textbf{\(c_k(\wtil_1)\)} \\ \hline
0 & \( 1 \) & 1 & \( \frac{11803625}{8800252} \) & 2 & \( \frac{-3893}{481164} \) & 3 & \( \frac{-864831}{2757703} \) & 4 & \( \frac{167774}{2999557} \) \\ 
5 & \( \frac{50497}{701038} \) & 6 & \( \frac{-114461}{2889844} \) & 7 & \( \frac{-7856}{3180127} \) & 8 & \( \frac{12779}{1300443} \) & 9 & \( \frac{-3942}{1107947} \) \\ 
10 & \( \frac{-867}{2903818} \) & 11 & \( \frac{1300}{1817057} \) & 12 & \( \frac{-803}{2803689} \) & 13 & \( \frac{105}{2015003} \) & 14 & \( \frac{11}{1153578} \) \\ 
15 & \( \frac{-37}{1744403} \) & 16 & \( \frac{40}{2380829} \) & 17 & \( \frac{-2}{311935} \) & 18 & \( \frac{-1}{1266945} \) & 19 & \( \frac{7}{2922589} \) \\ 
20 & \( \frac{-1}{779671} \) & 21 & \( \frac{1}{5746076} \) & 22 & \( \frac{1}{4938094} \) & 23 & \( \frac{-1}{6166019} \) & 24 & \( \frac{1}{18271223} \) \\ 
25 & \( \frac{1}{272724931} \) & 26 & \( \frac{-1}{56968973} \) & 27 & \( \frac{1}{139599039} \) & 28 & \( \frac{1}{1317546995} \) & 29 & \(c_{29}(\wtil_1)\) \\
30 & \(c_{30}(\wtil_1)\) & &  & &  & &  & & \\
\end{tabular}
\end{table}

\begin{table}[ht!]
\centering
\caption{Coefficients of \(\wtil_1\) for \(d=4\)}
\label{tab:coef_w1_d4}
\vspace{-0.3cm}
\begin{tabular}{c|c c|c c|c c|c c|c} 
\textbf{\(k\)} & \textbf{\(c_k(\wtil_1)\)} & \textbf{\(k\)} & \textbf{\(c_k(\wtil_1)\)} & \textbf{\(k\)} & \textbf{\(c_k(\wtil_1)\)} & \textbf{\(k\)} & \textbf{\(c_k(\wtil_1)\)} & \textbf{\(k\)} & \textbf{\(c_k(\wtil_1)\)} \\ \hline
0 & \( 1 \) & 1 & \( \frac{97762735}{65380431} \) & 2 & \( \frac{30504083}{106243288} \) & 3 & \( \frac{-30128105}{83898268} \) & 4 & \( \frac{-5633586}{47450267} \) \\ 
5 & \( \frac{2012842}{16321441} \) & 6 & \( \frac{1517799}{66254210} \) & 7 & \( \frac{-1321075}{32104308} \) & 8 & \( \frac{140793}{124498868} \) & 9 & \( \frac{1100359}{100849093} \) \\ 
10 & \( \frac{-93223}{38848750} \) & 11 & \( \frac{-361358}{165125365} \) & 12 & \( \frac{12733}{14243143} \) & 13 & \( \frac{51863}{138790902} \) & 14 & \( \frac{-11460}{48469339} \) \\ 
15 & \( \frac{-1562}{22160843} \) & 16 & \( \frac{7264}{114973269} \) & 17 & \( \frac{753}{51856541} \) & 18 & \( \frac{-1157}{60265016} \) & 19 & \( \frac{-43}{27833766} \) \\ 
20 & \( \frac{492}{89100845} \) & 21 & \( \frac{-19}{38373933} \) & 22 & \( \frac{-13}{9731089} \) & 23 & \( \frac{14}{41339867} \) & 24 & \( \frac{23}{85935810} \) \\ 
25 & \( \frac{-4}{35742063} \) & 26 & \( \frac{-4}{76202427} \) & 27 & \( \frac{1}{30295308} \) & 28 & \( \frac{1}{104467578} \) & 29 & \( \frac{-1}{110872658} \) \\ 
30 & \( \frac{-1}{511931277} \) & 31 & \( \frac{1}{337707502} \) & 32 & \( \frac{-1}{9664024017} \) & 33 & \( \frac{-1}{1652769343} \) & 34 & \( \frac{1}{1146857116493} \) \\ 
35 & \( \frac{1}{4165370537} \) & 36 & \( \frac{-1}{9742752883} \) & 37 & \( \frac{1}{106587665632} \) & 38 & \( \frac{-1}{50876909278} \) & 39 & \( \frac{1}{28821129587} \) \\ 
40 & \( \frac{-1}{37469561746} \) & 41 & \( \frac{1}{60878490681} \) & 42 & \( \frac{-1}{77216963898} \) & 43 & \( \frac{1}{84968894530} \) & 44 & \( \frac{-1}{101847610801} \) \\ 
45 & \( \frac{1}{129424165177} \) & 46 & \( \frac{-1}{159746819643} \) & 47 & \( \frac{1}{191448867071} \) & 48 & \( \frac{-1}{230116445748} \) & 49 & \( \frac{1}{275416015993} \) \\ 
50 & \( \frac{-1}{332555168008} \) & 51 & \( \frac{1}{401161361773} \) & 52 & \( \frac{-1}{463399831633} \) & 53 & \( \frac{1}{562437067981} \) & 54 & \( \frac{-1}{682033486903} \) \\ 
55 & \( \frac{1}{732113686242} \) & 56 & \( \frac{-1}{958808221604} \) & 57 & \( \frac{1}{1153554662275} \) & 58 & \( \frac{-1}{1027965331803} \) & 59 & \( \frac{1}{1898618895409} \) \\ 
60 & \( \frac{-1}{1867497289776} \) & 61 & \( \frac{1}{1181497339970} \) & 62 & \( \frac{-1}{50957953147904} \) & 63 & \( \frac{1}{2292720722907} \) & 64 & \( \frac{-1}{1080950121897} \) \\ 
65 & \( \frac{-1}{1406382606167} \) & 66 & \( \frac{-1}{928110102950} \) & 67 & \( \frac{1}{522894080259} \) & 68 & \(c_{68}(\wtil_1)\) & 69 & \(c_{69}(\wtil_1)\) \\ 
\end{tabular}
\end{table}

\begin{table}[ht!]
\centering
\caption{Coefficients of \(\wtil_1\) for \(d=5\)}
\label{tab:coef_w1_d5}
\vspace{-0.3cm}
\begin{tabular}{c|c c|c c|c c|c c|c} 
\textbf{\(k\)} & \textbf{\(c_k(\wtil_1)\)} & \textbf{\(k\)} & \textbf{\(c_k(\wtil_1)\)} & \textbf{\(k\)} & \textbf{\(c_k(\wtil_1)\)} & \textbf{\(k\)} & \textbf{\(c_k(\wtil_1)\)} & \textbf{\(k\)} & \textbf{\(c_k(\wtil_1)\)} \\ \hline
0 & \( 1 \) & 1 & \( \frac{26041005}{16807873} \) & 2 & \( \frac{4074199}{8786101} \) & 3 & \( \frac{-9573986}{31160203} \) & 4 & \( \frac{-7023207}{29585884} \) \\ 
5 & \( \frac{5816978}{67979357} \) & 6 & \( \frac{3715267}{42062527} \) & 7 & \( \frac{-376193}{11758379} \) & 8 & \( \frac{-758862}{26640751} \) & 9 & \( \frac{658053}{52983181} \) \\ 
10 & \( \frac{185690}{22513119} \) & 11 & \( \frac{-62335}{14162584} \) & 12 & \( \frac{-31329}{13597135} \) & 13 & \( \frac{63473}{43496631} \) & 14 & \( \frac{7557}{11416580} \) \\ 
15 & \( \frac{-6812}{14075011} \) & 16 & \( \frac{-6505}{33801961} \) & 17 & \( \frac{231}{1410914} \) & 18 & \( \frac{431}{8040597} \) & 19 & \( \frac{-1361}{24750956} \) \\ 
20 & \( \frac{-384}{27324749} \) & 21 & \( \frac{953}{52852186} \) & 22 & \( \frac{113}{31961419} \) & 23 & \( \frac{-164}{28134331} \) & 24 & \( \frac{-43}{50314655} \) \\ 
25 & \( \frac{127}{67488911} \) & 26 & \( \frac{5}{26927627} \) & 27 & \( \frac{-9}{14847415} \) & 28 & \( \frac{-1}{33618767} \) & 29 & \( \frac{3}{15506287} \) \\ 
30 & \( \frac{-1}{33429764811} \) & 31 & \( \frac{-1}{16368585} \) & 32 & \( \frac{1}{343942617} \) & 33 & \( \frac{1}{52209119} \) & 34 & \( \frac{-1}{551792563} \) \\ 
35 & \( \frac{-1}{167394615} \) & 36 & \( \frac{1}{1163144143} \) & 37 & \( \frac{1}{538720307} \) & 38 & \( \frac{-1}{2755610661} \) & 39 & \( \frac{-1}{1714550509} \) \\ 
40 & \( \frac{1}{7706223108} \) & 41 & \( \frac{1}{5341882956} \) & 42 & \( \frac{-1}{49770421587} \) & 43 & \( \frac{-1}{22574341232} \) & 44 & \( \frac{-1}{144162788918} \) \\ 
45 & \(c_{45}(\wtil_1)\) & 46 & \(c_{46}(\wtil_1)\) & &  & &  & &  \\
\end{tabular}
\end{table}

\begin{table}[ht!]
\centering
\caption{Coefficients of \(\wtil_1\) for \(d=6\)}
\label{tab:coef_w1_d6}
\vspace{-0.3cm}
\begin{tabular}{c|c c|c c|c c|c c|c} 
\textbf{\(k\)} & \textbf{\(c_k(\wtil_1)\)} & \textbf{\(k\)} & \textbf{\(c_k(\wtil_1)\)} & \textbf{\(k\)} & \textbf{\(c_k(\wtil_1)\)} & \textbf{\(k\)} & \textbf{\(c_k(\wtil_1)\)} & \textbf{\(k\)} & \textbf{\(c_k(\wtil_1)\)} \\ \hline
0 & \( 1 \) & 1 & \( \frac{158338445}{100094528} \) & 2 & \( \frac{46363281}{79772759} \) & 3 & \( \frac{-8520156}{35818093} \) & 4 & \( \frac{-4046322}{13230331} \) \\ 
5 & \( \frac{1007353}{47444741} \) & 6 & \( \frac{14449094}{116634659} \) & 7 & \( \frac{35341}{9966240} \) & 8 & \( \frac{-609003}{12870683} \) & 9 & \( \frac{-114867}{33971347} \) \\ 
10 & \( \frac{1070420}{61351359} \) & 11 & \( \frac{221995}{122111503} \) & 12 & \( \frac{-377407}{59580110} \) & 13 & \( \frac{-75481}{85762660} \) & 14 & \( \frac{633310}{275456861} \) \\ 
15 & \( \frac{22444}{56766371} \) & 16 & \( \frac{-31375}{37644457} \) & 17 & \( \frac{-14918}{88395113} \) & 18 & \( \frac{10443}{34788227} \) & 19 & \( \frac{25897}{366736009} \) \\ 
20 & \( \frac{-25835}{239923926} \) & 21 & \( \frac{-711}{24367261} \) & 22 & \( \frac{2975}{77074861} \) & 23 & \( \frac{971}{81940254} \) & 24 & \( \frac{-1464}{106083811} \) \\ 
25 & \( \frac{-142}{29930755} \) & 26 & \( \frac{91}{18528450} \) & 27 & \( \frac{299}{158699121} \) & 28 & \( \frac{-135}{77499397} \) & 29 & \( \frac{-55}{74042421} \) \\ 
30 & \( \frac{22}{35701675} \) & 31 & \( \frac{31}{106719848} \) & 32 & \( \frac{-38}{175004821} \) & 33 & \( \frac{-10}{88629059} \) & 34 & \( \frac{7}{91921724} \) \\ 
35 & \( \frac{1}{22932004} \) & 36 & \( \frac{-1}{37606009} \) & 37 & \( \frac{-1}{59624775} \) & 38 & \( \frac{1}{108189443} \) & 39 & \( \frac{1}{155788487} \) \\ 
40 & \( \frac{-1}{313018571} \) & 41 & \( \frac{-1}{408759847} \) & 42 & \( \frac{1}{911539877} \) & 43 & \( \frac{1}{1076457553} \) & 44 & \( \frac{-1}{2673485287} \) \\ 
45 & \( \frac{-1}{2846389642} \) & 46 & \( \frac{1}{7899944000} \) & 47 & \( \frac{1}{7560441790} \) & 48 & \( \frac{-1}{23380299632} \) & 49 & \( \frac{-1}{20038567545} \) \\ 
50 & \( \frac{1}{68829888534} \) & 51 & \( \frac{1}{51183267058} \) & 52 & \( \frac{-1}{233258372207} \) & 53 & \( \frac{-1}{123513075974} \) & 54 & \( \frac{-1}{5375176710463} \) \\ 
55 & \( \frac{1}{426885927463} \) & 56 & \( \frac{1}{1267188251902} \) & 57 & \(c_{57}(\wtil_1)\) & 58 & \(c_{58}(\wtil_1)\) & &  \\ 
\end{tabular}
\end{table}
        
\clearpage
\bibliography{Literatur}

\begin{thebibliography}{10}

\bibitem{Abramowitz_1964}
Milton Abramowitz and Irene~A. Stegun, editors.
\newblock {\em Handbook of Mathematical Functions}.
\newblock Number~55 in Applied Mathematics Series. National Bureau of
  Standards, 1972.
\newblock 10th printing.

\bibitem{BieBiz2011}
Pawe\l Biernat and Piotr Bizo\'{n}.
\newblock Shrinkers, expanders, and the unique continuation beyond generic
  blowup in the heat flow for harmonic maps between spheres.
\newblock {\em Nonlinearity}, 24(8):2211--2228, 2011.

\bibitem{BieDon18}
Pawe{\l} Biernat and Roland Donninger.
\newblock Construction of a spectrally stable self-similar blowup solution to
  the supercritical corotational harmonic map heat flow.
\newblock {\em Nonlinearity}, 31(8):3543--3566, 2018.

\bibitem{BieDonSch17}
Pawe{\l} Biernat, Roland Donninger, and Birgit Sch\"{o}rkhuber.
\newblock Stable self-similar blowup in the supercritical heat flow of harmonic
  maps.
\newblock {\em Calc. Var. Partial Differential Equations}, 56(6):Art. 171, 31,
  2017.

\bibitem{BieSek19}
Pawe{\l} Biernat and Yukihiro Seki.
\newblock Type {II} blow-up mechanism for supercritical harmonic map heat flow.
\newblock {\em Int. Math. Res. Not. IMRN}, (2):407--456, 2019.

\bibitem{BieSek20}
Pawe{\l} Biernat and Yukihiro Seki.
\newblock Transition of blow-up mechanisms in {$k$}-equivariant harmonic map
  heat flow.
\newblock {\em Nonlinearity}, 33(6):2756--2796, 2020.

\bibitem{BizonWasserman2015}
Piotr Bizo\'{n} and Arthur Wasserman.
\newblock Nonexistence of shrinkers for the harmonic map flow in higher
  dimensions.
\newblock {\em Int. Math. Res. Not. IMRN}, (17):7757--7762, 2015.

\bibitem{ChaDinYe92}
Kung-Ching Chang, Wei~Yue Ding, and Rugang Ye.
\newblock Finite-time blow-up of the heat flow of harmonic maps from surfaces.
\newblock {\em J. Differential Geom.}, 36(2):507--515, 1992.

\bibitem{CorGhi89}
Jean-Michel Coron and Jean-Michel Ghidaglia.
\newblock Explosion en temps fini pour le flot des applications harmoniques.
\newblock {\em C. R. Acad. Sci. Paris S\'{e}r. I Math.}, 308(12):339--344,
  1989.

\bibitem{Andrade_2003}
Eliana~X.L. {de Andrade}, John~H. McCabe, and A.~{Sri Ranga}.
\newblock The q-d algorithm for transforming series expansions into a
  corresponding continued fraction: an extension to cope with zero
  coefficients.
\newblock {\em Journal of Computational and Applied Mathematics},
  156(2):487--497, 2003.

\bibitem{DonSch26}
Roland Donninger and Birgit Sch\"{o}rkhuber.
\newblock Self-{S}imilar {B}lowup for the {C}ubic {S}chr\"{o}dinger {E}quation.
\newblock {\em Comm. Pure Appl. Math.}, 79(8):1831--1918, 2026.

\bibitem{DonSchWit2026}
Roland Donninger, Birgit Sch\"orkhuber, and Alexander Wittenstein.
\newblock Stable blowup for supercritical wave maps into perturbed spheres.
\newblock {\em J. Funct. Anal.}, 291(9):Paper No. 111601, 2026.

\bibitem{EngelNagel}
Klaus-Jochen Engel and Rainer Nagel.
\newblock {\em One-parameter semigroups for linear evolution equations}, volume
  194 of {\em Graduate Texts in Mathematics}.
\newblock Springer-Verlag, New York, 2000.
\newblock With contributions by S. Brendle, M. Campiti, T. Hahn, G. Metafune,
  G. Nickel, D. Pallara, C. Perazzoli, A. Rhandi, S. Romanelli and R.
  Schnaubelt.

\bibitem{Fan1999}
Huijun Fan.
\newblock Existence of the self-similar solutions in the heat flow of harmonic
  maps.
\newblock {\em Sci. China Ser. A}, 42(2):113--132, 1999.

\bibitem{Gas02}
Andreas Gastel.
\newblock Singularities of first kind in the harmonic map and {Y}ang-{M}ills
  heat flows.
\newblock {\em Math. Z.}, 242(1):47--62, 2002.

\bibitem{GerGhoMiu2017}
Pierre Germain, Tej-Eddine Ghoul, and Hideyuki Miura.
\newblock On uniqueness for the harmonic map heat flow in supercritical
  dimensions.
\newblock {\em Comm. Pure Appl. Math.}, 70(12):2247--2299, 2017.

\bibitem{GerRup2011}
Pierre Germain and Melanie Rupflin.
\newblock Selfsimilar expanders of the harmonic map flow.
\newblock {\em Ann. Inst. H. Poincar\'e{} C Anal. Non Lin\'eaire},
  28(5):743--773, 2011.

\bibitem{GesSimTes1996}
F.~Gesztesy, B.~Simon, and G.~Teschl.
\newblock Zeros of the {W}ronskian and renormalized oscillation theory.
\newblock {\em Amer. J. Math.}, 118(3):571--594, 1996.

\bibitem{GhoIbrNgu19}
Tej-eddine Ghoul, Slim Ibrahim, and Van~Tien Nguyen.
\newblock On the stability of type ii blowup for the 1-corotational
  energy-supercritical harmonic heat flow.
\newblock {\em Anal. PDE}, 12(1):113--187, 2019.

\bibitem{Glo22}
Irfan Glogi\'{c}.
\newblock Stable blowup for the supercritical hyperbolic {Y}ang-{M}ills
  equations.
\newblock {\em Adv. Math.}, 408(part B):Paper No. 108633, 52, 2022.

\bibitem{Glo2025}
Irfan Glogi\'{c}.
\newblock Globally stable blowup profile for supercritical wave maps in all
  dimensions.
\newblock {\em Calc. Var. Partial Differential Equations}, 64(2):Paper No. 46,
  34, 2025.

\bibitem{Glo25}
Irfan Glogi\'c.
\newblock Globally stable blowup profile for supercritical wave maps in all
  dimensions.
\newblock {\em Calc. Var. Partial Differential Equations}, 64(2):Paper No. 46,
  34, 2025.

\bibitem{GloKisSch24}
Irfan Glogi\'{c}, Sarah Kistner, and Birgit Sch\"{o}rkhuber.
\newblock Existence and stability of shrinkers for the harmonic map heat flow
  in higher dimensions.
\newblock {\em Calc. Var. Partial Differential Equations}, 63(4):Paper No. 96,
  33, 2024.

\bibitem{GloKisSch26}
Irfan Glogić, Sarah Kistner, and Birgit Schörkhuber.
\newblock Stable blowup profile for a semilinear heat equation with spatially
  inhomogeneous nonlinearity, 2026.

\bibitem{GraOh14}
Loukas Grafakos and Seungly Oh.
\newblock The {K}ato-{P}once inequality.
\newblock {\em Comm. Partial Differential Equations}, 39(6):1128--1157, 2014.

\bibitem{HanOls2010}
Harald Hanche-Olsen and Helge Holden.
\newblock The kolmogorov–riesz compactness theorem.
\newblock {\em Expositiones Mathematicae}, 28(4):385–394, 2010.

\bibitem{LinWan08}
Fanghua Lin and Changyou Wang.
\newblock {\em The analysis of harmonic maps and their heat flows}.
\newblock World Scientific Publishing Co. Pte. Ltd., Hackensack, NJ, 2008.

\bibitem{McCabe_1974}
John~H. McCabe.
\newblock A continued fraction expansion, with a truncation error estimate, for
  dawson's integral.
\newblock {\em Mathematics of Computation}, 28(127):811--816, 1974.

\bibitem{Picone_1910}
Mauro Picone.
\newblock Sui valori eccezionali di un parametro da cui dipende un'equazione
  differenziale lineare ordinaria del second'ordine.
\newblock {\em Ann. Scuola Norm. Sup. Pisa Cl. Sci.}, 11:144, 1910.

\bibitem{ShaTah94}
Jalal Shatah and A.~Shadi Tahvildar-Zadeh.
\newblock On the {C}auchy problem for equivariant wave maps.
\newblock {\em Comm. Pure Appl. Math.}, 47(5):719--754, 1994.

\bibitem{Str1988}
Michael Struwe.
\newblock On the evolution of harmonic maps in higher dimensions.
\newblock {\em J. Differential Geom.}, 28(3):485--502, 1988.

\bibitem{Weidmann_1987}
Joachim Weidmann.
\newblock {\em Spectral Theory of Ordinary Differential Operators}.
\newblock Lecture Notes in Mathematics. Springer-Verlag, Berlin, Heidelberg,
  1987.

\end{thebibliography}
\bibliographystyle{plain}

\end{document}